\documentclass{amsart}
\usepackage{dsfont}
\usepackage[dvipsnames,svgnames,table]{xcolor}
\usepackage[utf8]{inputenc}
\usepackage[T1]{fontenc}
\usepackage{ tipa }
\usepackage{helvet}
\usepackage{color}
\usepackage{mathrsfs}  
\usepackage{stmaryrd}  
\usepackage{amsmath,amsxtra,amsthm,amssymb,xr,enumerate,fullpage,comment, graphicx, mathtools}
\usepackage[all]{xy}
\usepackage[bbgreekl]{mathbbol}
\usepackage{bbm}

\DeclareMathSymbol{\invques}{\mathord}{operators}{`>}
\DeclareUnicodeCharacter{00BF}{\tmquestiondown}
\DeclareRobustCommand{\tmquestiondown}{%
  \ifmmode\invques\else\textquestiondown\fi
}

\DeclareRobustCommand\iff{\;\Longleftrightarrow\;}

\usepackage{verbatim}
\numberwithin{equation}{section}

\makeatletter
\newcommand{\mylabel}[2]{#2\def\@currentlabel{#2}\label{#1}}
\makeatother

\newtheorem{theorem}{Theorem}[section]
\newtheorem{lemma}[theorem]{Lemma}
\newtheorem{conj}[theorem]{Conjecture}
\newtheorem{proposition}[theorem]{Proposition}
\newtheorem{corollary}[theorem]{Corollary}
\newtheorem{defn}[theorem]{Definition}

\newtheorem{remark}[theorem]{Remark}

\newcommand{\Gal}{\operatorname{Gal}}

\newcommand{\QQ}{\mathbb{Q}}
\newcommand{\Qp}{{\mathbb{Q}_p}}
\newcommand{\Zp}{\mathbb{Z}_p}
\newcommand{\ZZ}{\mathbb{Z}}

\newcommand{\g}{\mathbf{g}}

\newcommand{\ord}{\mathrm{ord}}

\newcommand{\cL}{\mathcal{L}}

\newcommand{\cO}{\mathcal{O}}

\newcommand{\GL}{\mathrm{GL}}

\newcommand{\cyc}{\textup{cyc}}

\newcommand{\Ind}{\mathrm{Ind}}

\newcommand{\LL}{\Lambda}

\newcommand{\f}{\textup{\bf f}}

\newcommand{\bl}{\textup{\bf l}}
\newcommand{\bm}{\textup{\bf m}}
\newcommand{\h}{\textup{\bf h}}

\newcommand{\lra}{\longrightarrow}

\newcommand{\res}{\textup{res}}

\newcommand{\p}{\mathfrak{p}}
\newcommand{\q}{\mathfrak{q}}

\newcommand{\cW}{\mathcal{W}}

\newcommand{\cR}{\mathcal{R}}

\newcommand{\Cp}{\mathbb{C}_p}

\newcommand{\cA}{\mathcal{A}}

\newcommand{\fg}{\mathfrak{g}}

\newcommand{\Spf}{{\rm Spec}}
\newcommand{\cl}{{\rm cl}}
\newcommand{\bal}{{\rm bal}}
\newcommand{\wt}{{\rm w}}

\newcommand{\Ad}{{\rm ad}}

\newcommand{\Sym}{{\rm Sym}}
\newcommand{\rmw}{{\rm w}}
\newcommand{\cris}{{\rm crys}}

\newcommand{\hatotimes}{{\,\widehat\otimes\,}}

\newcommand{\hf}{\f}
\newcommand{\hg}{\g}
\newcommand{\hh}{\h}
\newcommand{\Q}{\QQ}
\newcommand{\barQ}{{\overline{\QQ}}}
\newcommand{\C}{{\mathbb C}}

\newcommand{\comm}[1]{}

\usepackage{hyperref}

\definecolor{pinegreen}{rgb}{0.0, 0.47, 0.44}

 \definecolor{pAlgae}{RGB}{87,115,135}
\definecolor{airforceblue}{rgb}{0.36, 0.54, 0.66}
	\definecolor{bondiblue}{rgb}{0.0, 0.58, 0.71}
\definecolor{britishracinggreen}{rgb}{0.0, 0.26, 0.15}
\definecolor{camouflagegreen}{rgb}{0.47, 0.53, 0.42}
\definecolor{darkcyan}{rgb}{0.0, 0.55, 0.55}

\hypersetup{
    colorlinks = true,
    linkcolor=blue,
     filecolor=blue,
     citecolor = darkcyan,      
     urlcolor=cyan,
    linkbordercolor = {white},
}


\begin{document}

\title{O\lowercase{n the} A\lowercase{rtin formalism for triple product $p$-adic} $L$-\lowercase{functions}: S\lowercase{uper-factorization}}

\author{K\^az\i m B\"uy\"ukboduk}
\address{K\^az\i m B\"uy\"ukboduk\newline UCD School of Mathematics and Statistics\\ University College Dublin\\ Ireland}
\email{kazim.buyukboduk@ucd.ie}

\author{Daniele Casazza} 
\address{Daniele Casazza\newline UCD School of Mathematics and Statistics\\ University College Dublin\\ Ireland}
\email{daniele.casazza@ucd.ie}

\begin{abstract}
We prove the factorization conjecture for triple-product $p$-adic $L$-functions formulated in a companion article in the special case when two of the (three) factors have CM.
\end{abstract}
\maketitle

\tableofcontents

\section{Introduction}

Let $\hf$ and $\hg$ be cuspidal primitive Hida families of elliptic modular forms and let $\hg^c$ denote the conjugate family of $\hg$. Our objective in this paper is to prove Conjecture~2.1 in \cite{BS_Part1} (on the $p$-adic Artin formalism for the triple product $p$-adic $L$-function denoted by $\cL^{\g}_p(\f\otimes \g \otimes \g)^2_{\vert_{\cW_2}}$ in op. cit.) in the special case when the family $\g$ has CM. In \S\ref{sec_super_factorization}, we also settle the algebraic counterpart of this statement in this setting, supplementing and refining the results in \cite{BS_Part1}.

Our main results are proved in \S\ref{subsec_75_2022_06_02_0900}, and our proof relies on what we call \emph{super-factorization}: When $\g$ has CM, the $3$-variable $p$-adic $L$-function $\cL^{\g}_p(\f\otimes \g \otimes \g)^2$ factors as a product two Hida--BDP $p$-adic $L$-functions. To conclude with the proof of \cite[Conjecture 2.1]{BS_Part1}, we prove in \S\ref{subsec_8_2_2022_09_13_1731} an (irregular) factorization statement for a certain Hida--BDP $p$-adic $L$-function (relying crucially on the work of Bertolini--Darmon--Venerucci~\cite{BDV}), where we express its restriction to the central critical line as the product of the Perrin-Riou logarithm of the big Beilinson--Kato elements associated to the family $\f$, and the values of the Mazur--Kitagawa $p$-adic $L$-function (attached to the quadratic twist $\f\otimes\epsilon_K$ of $\f$, where $\epsilon_K$ is the unique non-trivial Dirichlet character determined by $K/\QQ$) at the central critical line. We remark that the factorization of this Hida--BDP $p$-adic $L$-function also reflects what we call in \cite{BS_Part1} the BDP-principle.

Before stating our results in detail, we introduce the notation that we will rely on throughout this work in \S\ref{subsec_the_set_up_intro}.

\subsection{The set-up}
\label{subsec_the_set_up_intro}
Let $p$ be an odd prime and let $\cO$ be the ring of integers of a finite extension $E$ of $\Qp$. In what follows, we will enlarge $E$ as befitting our needs.
\subsubsection{Hida families}
\label{subsubsec_2022_05_16_1506}
Let us put $[\,\cdot\,]: \ZZ_p^\times \hookrightarrow \LL_{\rm wt}^\times$ for the natural injection, where $\LL_{\rm wt}:=\Zp[[\Zp^\times]]$. The universal weight character $\bbchi$ is the composite map $G_\QQ\stackrel{\chi_\cyc}{\lra}\ZZ_p^\times\hookrightarrow \LL_{\rm wt}^\times$, where $\chi_\cyc$ is the $p$-adic cyclotomic character. 
We will regard integers as elements of the weight space ${\rm Spec}(\LL_{\rm wt})(\cO)$ via $\ZZ\ni n \mapsto (\nu_n:[x] \mapsto x^n)$.

For any integer $k$, let us define $\LL_{\rm wt}^{(k)}\cong\LL(1+p\ZZ_p)$ as the component that is determined by the weight $k$, in the sense that the map $\LL_{\rm wt} \stackrel{\nu_{k} }{\lra} \ZZ_p$ factors through $\LL_{\rm wt}^{(k)}$. We let $\h=\sum_{n=1}^{\infty} \mathbb{a}_{n}(\h)q^n \in \cR_\h[[q]]$ denote the branch of the primitive Hida family of tame conductor $N_\h$ and tame nebentype character $\varepsilon_\h$\,, which admits a crystalline specialization $h_{\circ}$ of weight $k$. Here, $\cR_\h$ is the irreducible component of Hida's universal ordinary Hecke algebra determined by $h_{\circ}$. It is finite flat over $\LL_{\rm wt}^{(k)}$ and one has $a_p(\h) \in \mathcal \cR_\h^\times$. 
The universal weight character $\bbchi$ gives rise to the character
\[
    \bbchi_{\h}:  G_\QQ\xrightarrow{\bbchi} \LL_{\rm wt}^\times \twoheadrightarrow \LL_{\rm wt}^{(k),\times}\lra \cR_\h^\times\,.
\]
For any $\kappa\in \cW_\hh :={\rm Spec}(\cR_\h)(\C_p)$, let us write ${\rm wt}(\kappa)\in {\rm Spec}(\LL_{\rm wt})(\C_p)$ for the point that $\kappa$ lies over and call it the weight character of $\kappa$.
We call $\cW_\h$ the weight space for the Hida family $\h$. We say that $\kappa \in \cW_\h$ is an arithmetic (or classical) classical specialization of $\h$ whenever the restriction $\wt(\kappa)$ of $\kappa$ to $\LL_{\rm wt}^{(k)}$ agrees with $\nu_{n}$ on an open subgroup of $\ZZ_p^\times$ for some integer $n\in \ZZ_{\geq 2}$. In this case, we will abusively treat $\wt(\kappa)$ as an integer, identifying it with $n$ as above.  We denote by $\cW_\hf^{\rm cl}\subset \cW_\hf$ the set of arithmetic specializations. We say that an arithmetic specialization $\kappa$ is crystalline if ${\rm w}(\kappa)\equiv k\mod (p-1)$. 

Let $G_{\Q,\Sigma} \xrightarrow{\rho_\h} \GL_2(\mathrm{Frac}(\cR_\h))$ denote
Hida's Galois representation attached to $\h$, where $\Sigma$ is a finite set of primes containing all those dividing $pN_\h\infty$. Let us denote by $T_\h\subset \mathrm{Frac}(\cR_\h)^{\oplus 2}$ the Ohta lattice (cf. \cite{ohta99,ohta00}, see also \cite{KLZ2} where our $T_\h$ corresponds to $M(\h)^*$ in op. cit.), that realizes the Galois representation $\rho_\h$ in the \'etale cohomology groups of a tower of modular curves. Assume in addition that 
    \begin{itemize}
       \item[\mylabel{item_Irr}{\bf (Irr)}]  the residual representation $\bar{\rho}_\h$ is absolutely irreducible
    \end{itemize}
then $T_\h$ is a free $\cR_\h$-module of rank $2$. By a theorem of Wiles, we have
${\rho_\h}_{|_{G_{\Qp}}} \simeq \begin{pmatrix} \delta_\h & * \\ 0 & \alpha_\h\end{pmatrix}$, where $G_{\QQ_p}\xrightarrow{\alpha_\h} \cR_\h^\times$ is the unramified character given by $\alpha_\h({\rm Frob}_p)=a_p(\h)$ and 
$\delta_\h:=\bbchi_{\h}\, \chi_\cyc^{-1}\, \alpha_\h^{-1}\,\varepsilon_\h$. If we have
  \begin{itemize}
         \item[\mylabel{item_Dist}{\bf (Dist)}] $\delta_\h \not\equiv \alpha_\h \mod \mathfrak{m}_{\h}$,\qquad $\mathfrak{m}_{\h}\subset \cR_\h$ is the maximal ideal,
    \end{itemize}
then the lattice $T_\h$ fits in an exact sequence
\begin{equation} \label{eqn:filtrationf}
    0\lra T_\h^+ \lra T_\h \lra T_\h^- \lra 0 
\end{equation} 
of $\cR_\h[[G_{\Qp}]]$-modules. 
We henceforth assume, unless we explicitly state otherwise, that \ref{item_Irr} and \ref{item_Dist} hold true for all Hida families that appear in this work.

\subsubsection{Self-dual triple products}
\label{subsubsec_211_2022_06_01_1635}
Let $\f$ and $\g$ be primitive Hida families of ordinary $p$-stabilized newforms of tame levels $N_\f$ and $N_\g$, such that $\varepsilon_\f=\mathds{1}$ is the trivial character. Let us put $N:={\rm LCM}(N_\f, N_\g)$ and put $T_3 := T_\f\,\widehat\otimes\,_{\ZZ_p} T_\g\,\widehat\otimes\,_{\ZZ_p} T_{\g^c}$, where $\hg^c=\hg\otimes \varepsilon_\hg^{-1}$ is the conjugate Hida family. Note that we have identified the ring $\cR_{\g^c}$ with $\cR_\g$, through which we shall treat a specialization $\lambda \in \cW_\hg$ of $\cR_{\hg}$ also as a specialization of  $\cR_{\g^c}$. When $\lambda$ is crystalline, the specialization $\hg^c_\lambda$ of the Hida family $\hg^c$ coincides with the conjugate of the overconvergent eigenform $\g_\lambda$; namely, $\hg^c_\lambda=\g_\lambda\otimes \epsilon_{\hg}^{-1}\,.$

Let us put $\cR_3 := \cR_\f \,\widehat\otimes\,_{\ZZ_p} \cR_\g \,\widehat\otimes\,_{\ZZ_p} \cR_\g$, and consider the associated weight space 
$$\cW_3 := {\rm Spec}(\cR_3)(\Cp):=\cW_\hf\times \cW_\hg \times \cW_\hh.$$ 
Since $\varepsilon_\f=\mathds{1}$, we have a perfect $G_\QQ$-equivariant pairing
$T_3\otimes T_3\to \bbchi_{3}\chi_\cyc^{-{3}}$\,, where $\bbchi_{3}:=\bbchi_{\f}\otimes \bbchi_{\g} \otimes \bbchi_{\g^c}$. Since $p>2$ by assumption, there exists a unique character $\bbchi_{3}^{\frac{1}{2}}:G_\Q\to \cR_3^{\times}$ with $(\bbchi_3^{\frac{1}{2}})^2=\bbchi_3$. Then the Galois representation $T_3^\dagger := T_3\otimes \bbchi_3^{-\frac{1}{2}}\chi_\cyc^2$
is self-dual.

\subsubsection{}
Since $\varepsilon_\f=\mathds{1}$, we have $\det(\rho_\f)=\bbchi_{\f}\chi_\cyc^{-1}$. As above, we may and will define the central critical twist of $T_\f$ on setting $T_\f^\dagger:=T_\f\otimes \bbchi_\f^{-\frac{1}{2}}\chi_\cyc$.  The Galois representation $M_2^\dagger := T_\hf^\dagger\hatotimes\,_{\ZZ_p} \Ad^0(T_\hg)$, where $\Ad^0$ denotes the trace-zero endomorphisms of $T_\hg$, is then self-dual as well. We note that $M_2^\dagger$ is a Galois representation of rank-$6$ over the complete local Noetherian ring 
$\cR_2 := \cR_\hf\hatotimes_{\ZZ_p} \cR_\hg$ of Krull-dimension $3$.
Let us consider the associated weight space $\cW_2 := \Spf(\cR_2)(\Cp)$. We decompose the set of classical points in the weight space $\cW_2^{\rm cl}\subset \cW_2$ as follows:
\begin{equation}
\label{eqn_2022_05_17_1552}
     \cW_2^{\rm cl} = \underbrace{\left\lbrace (\kappa,\lambda) \in \cW_2^{\cl}: 2{\rm w}(\lambda) > {\rm w}(\kappa) \right\rbrace}_{\cW_2^{\Ad}}\quad \bigsqcup\quad\underbrace{\left\lbrace (\kappa,\lambda) \in \cW_2^{\cl}: 2{\rm w}(\lambda) \leq {\rm w}(\kappa) \right\rbrace}_{\cW_2^{\hf}}\,.
\end{equation}

\subsubsection{} 
Let us consider the natural homomorphisms
\begin{align}
\begin{aligned}
     \label{eqn_2022_05_17_1410}
     \cR_3&\xrightarrow{\iota_{2,3}^*} \cR_2& \qquad\qquad \cR_\hf&\xhookrightarrow{\varpi_{2,1}^*} \cR_2\\
    a\otimes b\otimes c &\longmapsto a\otimes bc &\qquad\qquad a&\longmapsto a\otimes 1
\end{aligned}
\end{align}
of complete local Noetherian $\ZZ_p$-algebras, which induce the morphisms
\begin{align}
\begin{aligned}
     \label{eqn_2022_05_17_1421}
     \cW_2&\xrightarrow{\iota_{2,3}} \cW_3 &\qquad\qquad \cW_2&\stackrel{\varpi_{2,1}}{\twoheadrightarrow} \cW_{\hf}\\
    (\kappa,\lambda) &\longmapsto (\kappa,\lambda,\lambda)&\qquad\qquad  (\kappa,\lambda)&\longmapsto \kappa\,.\\
\end{aligned}
\end{align}
Let us put $T_2^\dagger:=\iota_{2,3}^* \bigl( T_3^\dagger\bigr)$. We then have a split exact sequence
\begin{equation} \label{eqn_factorisationrepresentations_tr}
    0\lra M_2^\dagger\xrightarrow{\iota_{\rm tr}} T_2^\dagger\xrightarrow{{\rm id}\otimes {\rm tr}} \varpi^*_{2,1}(T_\hf^\dagger)\lra 0\,,
\end{equation}
where $\Ad(T_\g)\xrightarrow{{\rm tr}} \cR_\g$ is the trace map. The self-duality of $M_2^\dagger$, $T_2^\dagger$ and $T_\f^\dagger$ gives rise to the exact sequence 
\begin{equation} \label{eqn_factorisationrepresentations_dual_trace}
    0\lra \varpi^*_{2,1}(T_\hf^\dagger) \xrightarrow{{\rm id}\otimes {\rm tr}^*} T_2^\dagger\xrightarrow{\pi_{{\rm tr}^*}} M_2^\dagger\lra 0\,, 
\end{equation}
where ${\rm tr}^*$ is given by transposing the trace map and using the self-duality of $\Ad(T_\g)$. 

\subsubsection{L-functions and periods}
\label{subsubsec_213_2022_05_17_1455}
For an arithmetic specialization $x=(\kappa,\lambda,\mu)\in \cW_3^{\rm cl}:=\cW_\hf^\cl\times\cW_\hg^\cl\times\cW_\hh^\cl$
of weight $({\rm w}(\kappa), {\rm w}(\lambda), {\rm w}(\mu))$, let $f,g,h$
whose ordinary $p$-stabilizations are $(\hf_\kappa,\hg_\lambda,\hh_\mu)$. Let us write 
$$\mathscr{M}(x)=\mathscr{M}(f\otimes g\otimes h):=\mathscr{M}(f) \otimes \mathscr{M}(g)\otimes \mathscr{M}(h)$$ 
for the motive associated to $f\otimes g\otimes h$, where $\mathscr{M}(?)$ (for $?=f,g,h$) stands for Scholl's motive, whose self-dual twist is $\mathscr{M}^\dagger(x):= \mathscr{M}(x)(c(x))$, with $$c(x):=\frac{{\rm w}(\kappa)+{\rm w}(\lambda)+{\rm w}(\mu)}{2}-1\,.$$
The $L$-function $L(f\otimes g\otimes h,s)$ associated with the motive $\mathscr{M}(x)$ is known to have an analytic continuation and satisfies a functional equation of the form
\begin{equation} \label{eqn:functionalequationtriple}
    \Lambda(f\otimes g \otimes h, s) \doteq \varepsilon(f\otimes g \otimes h) \cdot N(f\otimes g\otimes h)^{-s} \cdot \Lambda(f\otimes g\otimes h, 2c-s),
\end{equation}
where $\Lambda(f\otimes g\otimes h,s)$ denotes the completed $L$-function, $N(f\otimes g\otimes h)\in \ZZ^+$, and where $\varepsilon(f\otimes g \otimes h)\in \{\pm 1\}$ is the global root number (which is given as the product of local root numbers).  

The set of classical points can be subdivided into the following regions:
\begin{align*}
	\cW_3^\bal &:= \{ x\in \cW_\hf^\cl\times\cW_\hg^\cl\times\cW_{\g^c}^\cl \mid {\rm w}(\kappa)+{\rm w}(\lambda)+{\rm w}(\mu) > 2\max\{{\rm w}(\kappa), {\rm w}(\lambda), {\rm w}(\mu)\} \} 
 \\
\cW_3^\f &:= \{ x\in \cW_\hf^\cl\times\cW_\hg^\cl\times\cW_\hh^\cl \mid {\rm w}(\kappa)+{\rm w}(\lambda)+{\rm w}(\mu) \leq 2{\rm w}(\kappa)\} 
\end{align*}
and we similarly define $\cW_3^\g$ and $\cW_3^{\g^c}$. Since the non-archimedean local root numbers are constant in families, and the archimedean root number equals $+1$ in $\cW_3^\bal$ and $-1$ otherwise, we have $\varepsilon(f\otimes g \otimes h) = \pm 1$ if $(\kappa,\lambda,\mu)\in \cW_3^\bal$ and $\varepsilon(f\otimes g \otimes h)=\mp 1$ has the opposite sign otherwise.

In the present paper, we are primarily interested in the setting when 
\begin{equation}
\label{eqn_2022_05_16_1626}
\varepsilon(f\otimes g \otimes h)=-1 \quad  \hbox{ if } (\kappa,\lambda,\mu)\in \cW_3^\bal\,.
\end{equation}
We remark that this is precisely the scenario of \cite{BSV, DarmonRotger}, where the authors have studied the cycles that account for the systematic vanishing (enforced by the functional equation \eqref{eqn:functionalequationtriple} above) of the $L$-series at the central critical point.

\subsubsection{Triple product $p$-adic $L$-functions} Under certain technical assumptions (cf. \S\ref{subsec:tripleproducts}), Hsieh has constructed four $p$-adic $L$-functions
\[
    \cL_p^\bullet(\hf\otimes\hg\otimes \hh) : \cW_3\lra \Cp, \qquad  \bullet\in \{\bal,\hf,\hg,\hh\}.
\]
These are uniquely determined by the following interpolation property: For all classical $(\kappa,\lambda,\mu)\in\cW_3^\bullet$,
\begin{equation}
\label{eqn_2022_05_16_1610}
    \cL_p^\bullet(\hf\otimes\hg\otimes \hh)(\kappa,\lambda,\mu) \doteq \frac{\Lambda(f\otimes g\otimes h,c)}{\Omega^\bullet}, \qquad    \bullet\in \{\bal, \hf,\hg,\hh\}\,.
\end{equation}
See Theorem~\ref{thm:unbalancedinterpolation} (due to Hsieh) below for the precise formula in the unbalanced case (and also \cite{Hsieh}, Theorems A and B).

\subsection{The Artin formalism} 
\label{subsec_Artin_formalism_intro}
Let $x=(\kappa,\lambda,\lambda)\in \cW_3^{\cl}$ be a classical point and let $f=\hf_\kappa$, $g=\hg_\lambda$ and $g^c=\hg^c_\lambda$ denote the corresponding specializations. The self-dual motive $\mathscr{M}^\dagger(x)$ attached to $f\otimes g \otimes h$ (cf. \S\ref{subsubsec_213_2022_05_17_1455}) then decomposes:
\begin{equation}
\label{eqn_2022_05_17_1544}
    \mathscr{M}^\dagger(x)=(\mathscr{M}^\dagger(f)\otimes {\rm ad}^0\mathscr{M}(g))\oplus \mathscr{M}^\dagger(f)\,.
\end{equation}
The Artin formalism applied to this decomposition shows that the complex analytic triple product $L$-function naturally factors as
\begin{equation} \label{eqn:complexfactorization}
    L(f\otimes g \otimes g^c,s) = L(f\otimes \Ad^0(g),s-{\rm w}(\lambda)+1) \cdot L(f,s-{\rm w}(\lambda)+1).
\end{equation}
At the central critical point ${\rm w}(\kappa)/2+{\rm w}(\lambda)-1$, the factorization \eqref{eqn:complexfactorization} reads
\begin{equation}
   \label{eqn_2022_05_16_1424}
    L(f\otimes g \otimes g^c,{\rm w}(\kappa)/2+{\rm w}(\lambda)-1) = L(f\otimes \Ad^0(g),{\rm w}(\kappa)/2) \cdot L(f,{\rm w}(\kappa)/2)\,.
    \end{equation}
 \subsubsection{$p$-adic $L$-functions of families of degree 6 and degree 2 motives.}
 \label{subsubsec_intro_padic_artin_formalism_1}
Our main goal in the present work is to establish $p$-adic analogues of the factorization~\eqref{eqn_2022_05_16_1424}.  Recall that there are four $p$-adic $L$-functions one may consider, 
and the $p$-adic Artin formalism in our set-up amounts to the factorization of these four $p$-adic $L$-functions, into a product of two $p$-adic $L$-functions.

We briefly discuss the nature of these two $p$-adic $L$-functions to motivate the reader for the main problem at hand and for our results. The $p$-adic $L$-function that is related to the second summand in \eqref{eqn_2022_05_17_1544} (as $f$ varies in the Hida family $\hf$) is the Mazur--Kitagawa $p$-adic $L$-function 
$\cL_p^{\rm Kit}(\hf)=\cL_p^{\rm Kit}(\hf)(\kappa,\pmb{\sigma})$\,, where $\pmb{\sigma}$ denotes the cyclotomic variable. The putative $p$-adic $L$-functions that correspond to the first summand in \eqref{eqn_2022_05_17_1544} (as both $f$ and $g$ vary in the respective Hida families) are given by the interpolation formula (in very rough form; see \S\ref{sec:Adjoint} for details)
\begin{align}
\label{eqn_2022_05_17_1606}
\begin{aligned}
 \cL_p^\Ad(\hf\otimes\Ad^0\hg) &\doteq L(\hf_\kappa\otimes \Ad^0(\hg_\lambda),{\rm w}(\kappa)/2), \qquad  (\kappa,\lambda)\in \cW_2^\Ad\,,
 \\
     \cL_p^\hf(\hf\otimes\Ad^0\hg) &\doteq L(\hf_\kappa\otimes \Ad^0(\hg_\lambda),{\rm w}(\kappa)/2), \qquad (\kappa,\lambda)\in \cW_2^\hf\,,
\end{aligned}
\end{align}
where $\cW_2^\Ad$ and $\cW_2^\hf$ are given as in \eqref{eqn_2022_05_17_1552}. We construct the $p$-adic $L$-function $\cL_p^\Ad(\hf\otimes\Ad^0\hg)$ when the Hida family $\g$ has CM.

\subsubsection{} 
Let us denote by $\varepsilon(\hf)=\pm 1$ the global root number of some (any) member of the Hida family $\f$ at its central critical point. In view of \eqref{eqn_2022_05_17_1606} and the interpolation formula for the Mazur--Kitagawa $p$-adic $L$-function, we infer that either $\cL_p^\Ad(\hf\otimes\Ad^0\hg)=0$ (if $\varepsilon(\hf)=+1$)\,, or else $\cL_p^{\rm Kit}(\hf)(\kappa,{\rm w}(\kappa)/2)=0$ identically (if $\varepsilon(\hf)=-1$). In all cases, we have the rather uninteresting factorization 
$$ \cL_p^{\bal}(\f\otimes\g\otimes\g^c)^2\,_{\vert_{\cW_2}}=0=\cL_p^\Ad(\hf\otimes\Ad^0\hg)\cdot L_p^{\rm Kit}(\hf)\,_{\vert_{\sigma=\frac{\wt(\kappa)}{2}}}$$
of the balanced $p$-adic $L$-function.
   
\subsubsection{} 
 \label{subsubsec_intro_padic_artin_formalism_5}

The factorization of the $\f$-dominant $p$-adic $L$-function
\begin{equation}
\label{eqn_2022_05_17_1725}
\cL_p^{\hf}(\f\otimes\g\otimes\g^c)^2\,_{\vert_{\cW_2}}\,\dot{=}\,\cL_p^\hf(\hf\otimes\Ad^0\hg)\cdot \cL_p^{\rm Kit}(\hf)\,_{\vert_{\sigma=\frac{\wt(\kappa)}{2}}}
\end{equation}
reduces, once we are provided with the $p$-adic $L$-function $\cL_p^\hf(\hf\otimes\Ad^0\hg)$ with the expected interpolative properties, to a comparison of periods at the specializations of both sides at $(\kappa,\lambda)\in \cW_2^{\hf}$. When $\varepsilon(\hf)=-1$, this again reduces to $0=0$.

\subsubsection{Main results} 
\label{subsubsec_2017_05_17_1800}
It remains to consider the factorization problem for $\cL_p^{\hg}(\f\otimes\g\otimes\g^c)\,_{\vert_{\cW_2}}$. This is the most challenging case, since the interpolation range for $\cL_p^{\hg}(\f\otimes\g\otimes\g^c)\,_{\vert_{\cW_2}}:=\iota_{2,3}^*\circ \cL_p^\g(\f\otimes\g\otimes\g^c)$ is empty, as $\iota_{2,3}(\cW_2)\cap \cW_3^\hg =\emptyset$. We will address this problem in the scenario when $\varepsilon(\hf)=-1$, and prove the following theorem (which is a special case of \cite{BS_Part1}, Conjecture 2.1):

\begin{theorem}[Theorem~\ref{Thm:8=6+2CM}]
\label{thm_main_6_plus_2}
Suppose that the family $\g$ has complex multiplication by a quadratic imaginary field discriminant coprime to $p$. Assume also that $\varepsilon(\hf)=-1$ and \eqref{eqn_2022_05_16_1626} holds true. We have the factorization of $p$-adic L-functions
\[
    \cL_p^\hg(\hf\otimes\hg\otimes\hg^c)^2\,_{\vert_{\cW_2}} = \mathscr C(\kappa)\cdot \cL_p^\Ad(\hf\otimes \Ad^0\hg)\cdot {\rm Log}_{\omega_\f}({\rm BK}_{\f}^\dagger)\,,
        \]
over sufficiently small wide-open discs $U \times U'$ in ${\rm Spm}(\cR_\f[\frac{1}{p}]) \times {\rm Spm}(\cR_\g[\frac{1}{p}])$; where $\mathscr C(\kappa)$ is a meromorphic function with an explicit algebraicity property at crystalline specializations $\kappa$ (cf. Theorem~\ref{Thm:8=6+2CM}).
\end{theorem}
Here, ${\rm Log}_{\omega_\f}({\rm BK}_\f^\dagger)$ denotes the logarithm of the big Beilinson--Kato class (constructed by Ochiai); cf. \cite[\S6.1.2]{BS_Part1} for relevant definitions. We refer the reader to \S6.2.7 in op. cit. to explain how it encodes information about the derivatives of $p$-adic $L$-functions in a precise manner.

\subsubsection{} Even though Theorem~\ref{thm_main_6_plus_2} has the same flavour as the main results of \cite{Dasgupta2016,Gross1980Factorization}, there are key technical differences. These points of divergence are detailed in \cite[\S2.4]{BS_Part1}.

\subsubsection{} 
\label{subsubsec_231_intro_2022_09_28_1640}
We note that, along the way to prove Theorem~\ref{thm_main_6_plus_2}, we prove a factorization formula for the BDP $p$-adic $L$-function (Proposition~\ref{prop:factorizationBDPKitagawa}), an ingredient in our proof of Theorem~\ref{thm_main_6_plus_2}. This factorization result might be of independent interest.

\subsubsection{} 
We also prove the algebraic counterpart of Theorem~\ref{thm_main_6_plus_2}, which amounts to a factorization of the modules of leading terms (under mild hypotheses) that Sakamoto and the first named author introduced in \cite[\S6.1.3 and \S6.1.4]{BS_Part1}. As in op. cit., we denote the module of algebraic $p$-adic $L$-functions by $\delta(T,\Delta)$ for the pairs $(T,\Delta)\in \{(M_2^\dagger,{\rm tr}^*\Delta_\g), (T_2^\dagger,\Delta_\g)\}$.

\begin{theorem}[Theorem~\ref{thm_main_8_4_4_factorization_bis} below]
\label{thm_main_8_4_4_factorization_intro}
Under the hypotheses of Corollary~\ref{cor_2022_09_20_1303}, we have
\begin{equation}
\label{eqn_2022_09_13_1733_intro}
\delta(T_2^\dagger,\Delta_\g)\,{=}\,  {\rm Exp}_{F^-T_\f^\dagger}^*(\delta(M_2^\dagger,{\rm tr}^*\Delta_{\g}))\cdot \varpi_{2,1}^*{\rm Log}_{\omega_\f}({\rm BK}_\f^\dagger).
\end{equation}
\end{theorem}

We remark that Theorem~\ref{thm_main_8_4_4_factorization_intro} supplements \cite[Theorem 2.2]{BS_Part1}, where the authors require that $\g$ does not have CM. It strengthens this result in op. cit. employing Corollary~\ref{cor_2022_09_20_1303}, where we give explicit sufficient conditions to prove the non-vanishing of the module of leading terms $\delta(T_2^\dagger,\Delta_\g)$ when $\g$ has CM, in terms of the BDP and Mazur--Kitagawa $p$-adic $L$-functions, and the Beilinson--Kato element. It is also worthwhile to note that our strategy to prove the factorization in Theorem~\ref{thm_main_8_4_4_factorization_bis} of the modules of algebraic $p$-adic $L$-functions parallels the proof of its analytic counterpart (Theorem~\ref{thm_main_6_plus_2}), as both rely on what we call `super-factorization': When $\g$ has CM, the underlying family of motives of rank $8$ decomposes into a direct sum of three motives of ranks $4$, $2$ and $2$, respectively.

\subsection*{Acknowledgements} KB thanks Henri Darmon for his encouragement. DC thanks Ming-Lun Hsieh for helpful discussions on his work. KB’s research in this publication was conducted with the financial support of Taighde \'{E}ireann -- Research Ireland under Grant number IRCLA/2023/849 (HighCritical).

\section{\texorpdfstring{$p$}{}-adic \texorpdfstring{$L$}{}-functions}
\label{section_2_L_functions}
Our goal in the present section is to introduce the $p$-adic $L$-functions that we work with in this paper.


Throughout this section, for any weight space $\cW_{?}$ that will appear, we will denote by $\cW_{?}^{\rm crys}$ the subset of the classical points in $\cW_{?}^\cl$ which are crystalline at $p$, in the sense that the $p$-adic realization of the motive associated to $\kappa\in \cW_{?}^{\rm crys}$  is required to be crystalline at $p$. In particular, given $\kappa \in \cW_\hf^{\rm crys}$ of weight ${\rm w}(\kappa)\geq 2$, the specialization $\f_\kappa$ is $p$-old and arises as the unique $p$-ordinary stabilization of a newform $\hf_\kappa^\circ \in S_{{\rm w}(\kappa)}(\Gamma_1(N_\hf),\varepsilon_\f)$ (and similarly for $\g$ and $\g^c$).

\subsection{Mazur--Kitagawa / Greenberg--Stevens \texorpdfstring{$p$}{}-adic \texorpdfstring{$L$}{}-function}
\label{subsec_3_1_2022_05_25}
We set $\Gamma_\cyc:=\Gal(\QQ(\mu_{p^\infty})/\QQ))\xrightarrow[\chi_\cyc]{\sim}\ZZ_p^\times$, and let $\Lambda(\Gamma_\cyc):=\ZZ_p[[\Gamma_\cyc]]$ be the cyclotomic Iwasawa algebra. Let us put $\cW_\cyc:={\rm Spec}(\LL)(\mathbb{C}_p)$ and refer to it as the cyclotomic weight space. Note that $\eta\chi_\cyc^j\in \cW_\cyc$ for any integer $j$ and character $\eta$ of $\Gamma_\cyc$ of finite order, where the latter set of characters can be identified by the collection of Dirichlet characters with $p$-power conductor. For notational simplicity, we will sometimes write $\eta+j$ in place of the character $\eta\chi_\cyc^j$. We define the classical points $\cW_\cyc^\cl$ in the cyclotomic weight space on setting
\[
\cW_\cyc^\cl:=\{\eta+j \, \colon j\in \ZZ \hbox{ and } \eta \hbox{ is a Dirichlet character with $p$-power conductor} \}\,.
\]

The following interpolative property for the Mazur--Kitagawa $p$-adic $L$-function is borrowed from \cite[Theorem 6.7]{Ochiai2006} (see also \cite{Vatsal_Periods_1999} concerning the choice of periods):
\begin{theorem}[Kitagawa]
\label{thm_31_2022_06_02_0854}
    There exists an element (the Mazur--Kitagawa $p$-adic $L$-function) 
    $$\cL_p^{\rm Kit}(\hf)\in \cR_\hf[[\Gamma_\cyc]] = \cR_\hf \,\, \widehat{\otimes}_{\ZZ_p} \, \Lambda(\Gamma_\cyc)$$
such that  for every $\kappa\in \cW_\f^\cris$, any $j \in \ZZ\cap [1,\rmw(\kappa)-1]$ and all Dirichlet characters $\eta$ of conductor $p^r \geq 1$,
\begin{align}
\label{eqn_2022_05_26_1001}
\cL_p^{\rm Kit}(\hf)(\kappa,\eta+j)=(-1)^j\,\Gamma(j)\,\mathcal{V}(\f_\kappa^\circ,\eta,j)\,\tau(\eta)\,\frac{p^{(j-1)r}}{a_p(\f_\kappa)^r}\frac{L(\f_\kappa^\circ,\eta^{-1},j)}{(2\pi \sqrt{-1})^{j}\Omega_{\f_\kappa}^{\pm}}C_{\f_\kappa}^\pm
\end{align}
where, 
\begin{itemize}
\item $\tau(\eta)$ is the Gauss sum $($normalized to have norm $p^{r/2}$$)$, 
\item  $\displaystyle{\mathcal{V}(\f_\kappa^\circ,\eta,j)=\left(1-{p^{j-1}\eta(p)}\big{/}{a_p(\f_\kappa)}\right)\left(1-{p^{\kappa-1-j}\eta^{-1}(p)}\big{/}{a_p(\f_\kappa)}\right)}$,
\item $\Omega_{\f_\kappa}^{+}$ and $\Omega_{\f_\kappa}^{-}$ are canonical periods in the sense of \cite[\S1.3]{Vatsal_Periods_1999},
\item $C_{\f_\kappa}^+$ and $C_{\f_\kappa}^-$ are 
non-zero $p$-adic numbers that only depend on $\kappa$,
\item the sign $\pm$ is determined so as to ensure that $(-1)^{(j-1)}\eta(-1)=\pm1$.
\end{itemize}
Moreover, $\cL_p^{\rm Kit}(\hf)$ as above is unique up to multiplication by an element of $\cR_\f$ that is non-vanishing on $\cW_\f^\cl$\,.
\end{theorem}

\comm{\color{red} In particular, when $j=\frac{{\rm w}(\kappa)}{2}$ is the central critical point and $\eta = \mathds{1}$ is the trivial character, the interpolation formula \eqref{eqn_2022_05_26_1001} reads 
\begin{equation}
    \label{eqn_2022_05_26_1013}
     \frac{\cL_p^{\rm Kit}(\hf)(\kappa, \frac{{\rm w}(\kappa)}{2})}{C_{\f_\kappa}^{\pm}}  = (-1)^{\frac{{\rm w}(\kappa)}{2}}\,\Gamma\left(\frac{\rmw(\kappa)}{2}\right)\,\left( 1- \frac{p^{\frac{{\rm w}(\kappa)}{2}-1} }{a_p(\hf_\kappa) }\right)^2\, \frac{ L(\hf_\kappa^\circ, \frac{{\rm w}(\kappa)}{2}) }{ (2\pi\sqrt{-1})^{\frac{{\rm w}(\kappa)}{2}} \Omega_{\hf_\kappa}^{\pm} }
\end{equation}
where $\pm$ is the sign of $(-1)^{\frac{{\rm w}(\kappa)}{2}-1}$.
}


\subsection{Hida--Rankin \texorpdfstring{$p$}{}-adic \texorpdfstring{$L$}{}-functions}
\label{subsec_32_2022_05_23_1426}
Let $\hf$ and $\hg$ denote two primitive Hida families of ordinary $p$-stabilized newforms of respective tame levels $N_\hf$, $N_\hg$ and tame nebentypes $\varepsilon_\f$, $\varepsilon_\hg$. 
Let us consider the weight space $\cW_{\f\g\bf{j}}:=\cW_\hf\times \cW_\hg \times \cW_\cyc$ of relative dimension $3$ over ${\rm Spec}(\ZZ_p)$, whose set of classical (resp. crystalline-at-$p$) points are given by $\cW_{\f\g\bf{j}}^\cl:=\cW_\f^\cl\times \cW_\g^\cl\times \cW_\cyc^\cl$ (resp. $\cW_{\f\g\bf{j}}^{\rm crys}:=\cW_\f^{\rm crys}\times \cW_\g^{\rm crys}\times \cW_\cyc^{\rm crys}$).

Associated to the parameters $(\kappa,\lambda, \eta+j)\in \cW_\f^\cl\times \cW_\g^\cl\times \cW_\cyc^\cl$ we consider the motive
$$\mathscr M(\kappa,\lambda, \eta+j):=\mathscr M(\f_\kappa)\otimes \mathscr M(\g_\lambda)\otimes \mathscr M(\eta^{-1}+1-j),$$ where $\mathscr M(\eta^{-1}+1-j)= \mathscr{M}(\eta^{-1})(1-j)$ is the $(1-j)$-fold Tate-twist of the Artin motive $ \mathscr{M}(\eta^{-1})$. We remark that we have the following explicit description of the $p$-adic realization of $\mathscr M(\kappa,\lambda, \eta+j)$ in terms of Deligne's Galois representations: $\mathscr M_p(\kappa,\lambda, \eta+j)=T_{\f_\kappa}\otimes T_{\g_\lambda}(\eta^{-1}+1-j)\otimes_{\ZZ_p}\QQ_p\,.$

The motive $\mathscr M(\kappa,\lambda, \eta+j)$ admits critical values if and only if $\rmw(\kappa)\neq \rmw(\lambda)$. 
For a choice of crystalline specializations $(\kappa,\lambda)$ we put
\[
    L_{\infty}(\hf_\kappa^\circ\otimes\hg_\lambda^\circ,s) := \begin{cases}
\Gamma_{\C}(s)\ \Gamma_{\C}(s+1-\rmw(\lambda)) , &  \rmw(\kappa)>\rmw(\lambda), \\
\Gamma_{\C}(s)\ \Gamma_{\C}(s+1-\rmw(\kappa)), & \rmw(\lambda)>\rmw(\kappa)
\end{cases}
\]
and define the completed $L$-series $\Lambda(\hf_\kappa^\circ\otimes\hg_\lambda^\circ,s) := L_\infty(\hf_\kappa^\circ\otimes\hg_\lambda^\circ, s) \cdot L(\hf_\kappa^\circ\otimes\hg_\lambda^\circ,s)$.

\subsubsection{Modified Hida periods}
\label{subsubsec_321_2022_05_26_1036}
For any primitive cuspidal Hida family $\f$ as before, consider the generator $\mathfrak c_\f$ of the congruence ideal of $\f$ (cf. \cite[Thm. 4.4(4.6)]{AIF_Hida88}). For a crystalline specialization $\kappa \in \cW_\f^\cris$ with $\wt(\kappa)\geq 2$, we define the modified canonical period on setting
\begin{equation} \label{eqn:modifiedHidaperiods}
    \Omega_{\f_\kappa} := \mathcal E_p(\f_\kappa^\circ,\Ad) \cdot \frac{(-2\sqrt{-1})^{k+1}}{\mathfrak c_\f(\kappa)} \cdot \langle \f_\kappa^\circ, \f_\kappa^\circ \rangle, \qquad \text{where } \quad \mathcal E(\f^\circ_\kappa,\Ad) := \left( 1 - \frac{\beta_{\f^\circ_\kappa}}{\alpha_{\f^\circ_\kappa}} \right) \left( 1-\frac{\beta_{\f^\circ_\kappa}}{p\alpha_{\f^\circ_\kappa}}\right).
\end{equation}
We remark that $\mathfrak c_\f(\kappa) := \kappa(\mathfrak c_\hf)$ generates the congruence ideal of $\f_\kappa$ (cf. \cite[(0.3)]{Hida81}). By \cite[Corollary 6.24 \& Theorem 6.28]{hida_2016}, the period $\Omega_{\f_\kappa}$ is equal to the product $\Omega_{\f_\kappa}^+\Omega_{\f_\kappa}^-$ of the plus/minus canonical periods given as in \cite[Page 488]{hida_1994} up to a $p$-adic unit. 

We also define the modified Euler factor
\begin{align} 
\begin{aligned}
\label{eqn_20220601_1700}
\mathcal E_p^{\f}({\f_\kappa^\circ}\otimes {\g_\lambda^\circ},s)\,\, = \left(1-\frac{p^{s-1}}{\alpha_{\f_\kappa^\circ}\alpha_{\g_\lambda^\circ}}\right)
\left(1-\frac{p^{s-1}}{\alpha_{\f_\kappa^\circ}\beta_{\g_\lambda^\circ}}\right)
\left(1-\frac{\beta_{\f_\kappa^\circ}\alpha_{\g_\lambda^\circ}}{p^s}\right) 
\left(1-\frac{\beta_{\f_\kappa^\circ}\beta_{\g_\lambda^\circ}}{p^s}\right), 
\end{aligned}
\end{align}
and similarly $\mathcal E_p^{\g}({\f_\kappa^\circ}\otimes {\g_\lambda^\circ},s)$ if $\rmw(\lambda)>\rmw(\kappa)$.

\subsubsection{$p$-optimal Hida--Rankin $p$-adic $L$-function} The following interpolation formula is borrowed from \cite{ChenHsieh2020}, where the authors have introduced a $p$-optimal version of Hida's $p$-adic $L$-function:
\begin{theorem}[Hida \cite{AIF_Hida88}] \label{thm:Hida-interp}
      There exist a unique pair of $p$-adic $L$-functions $
    \cL_p^?(\hf\otimes \hg) \in \cR_\hf \hatotimes_{\ZZ_p} \cR_\hg \hatotimes_{\ZZ_p} \Lambda(\Gamma_\cyc)$ ($?=\f,\g$)   with the following properties: for all $(\kappa, \lambda,j) \in \cW_{\hf\hg\bf{j}}^? \cap \cW_{\f\g\bf{j}}^{\rm crys}$ we have
    \begin{align*}
        \cL_p^?({\hf\otimes \hg})(\kappa, \lambda,j) &= \mathcal E_p^{?}({\f_\kappa^\circ}\otimes {\g_\lambda^\circ},j) \cdot (\sqrt{-1})^{\rmw(\kappa)+\rmw(\lambda)-\rmw(\mu)+1-2j}\cdot \mathfrak C_{\rm exc}(\hf\otimes\hg) \cdot \frac{\Lambda({\f_\kappa^\circ}\otimes {\g_\lambda^\circ},j)}{\Omega_{?_{\mu}}}\,.
    \end{align*}
    Here, $\mu=\kappa$ or $\lambda$ depending on whether $?=\f$ or $\g$, respectively, and $\mathfrak C_{\rm exc}(\hf\otimes\hg)= {\prod}_{q\in \Sigma_{\rm exc}(\hf\otimes \hg)}(1+q^{-1})$ as in \cite[\S1.5]{ChenHsieh2020}. 
\end{theorem}

\subsection{Triple product \texorpdfstring{$p$}{}-adic \texorpdfstring{$L$}{}-functions}
\label{subsec:tripleproducts}

Let $\f$, $\g$, and $\h$ be three primitive Hida families of ordinary $p$-stabilized newforms of respective tame levels $N_\f, N_\g, N_\h$ and nebentypes $\varepsilon_\f,\varepsilon_\g,\varepsilon_\h$, all verifying the hypotheses \ref{item_Irr} and \ref{item_Dist}. We assume that $\varepsilon_\f\varepsilon_\hg\varepsilon_\hh$ is the trivial Dirichlet character modulo $N:={\rm LCM}(N_\f,N_\g,N_\h)$. Given a $p$-crystalline classical point $x=(\kappa,\lambda,\mu) \in \cW_3^{\cris} := \cW_\hf^\cris \times \cW_{\hg}^{\cris}\times \cW_{\hh}^{\cris}$, recall that we have put
$c(x) := \frac{{\rm w}(\kappa)+{\rm w}(\lambda)+{\rm w}(\mu)}{2}-1$. 
We will write $c$ in place of $c(x)$ whenever the choice of the point $x=(\kappa,\lambda,\mu)$ is clear from the context.

\subsubsection{}  Let us put
\[
L_{\infty}(\f_\kappa^\circ\otimes\g_\lambda^\circ\otimes\h_\mu^\circ,s) := \begin{cases}
\Gamma_{\C}(s)\ \Gamma_{\C}(s+\wt(\kappa)-2c)\ \Gamma_{\C}(s+1-\wt(\lambda))\ \Gamma_{\C}(s+1-\wt(\mu)), & \quad x\in \cW_3^\hf, \\
\Gamma_{\C}(s)\ \Gamma_{\C}(s+1-\wt(\kappa))\ \Gamma_{\C}(s+1-\wt(\lambda))\ \Gamma_{\C}(s+1-\wt(\mu)), & \quad x\in \cW_3^\bal
\end{cases}
\]
(where $\cW_3^\f$ and $\cW_3^\bal$ are given as in \S\ref{subsubsec_213_2022_05_17_1455}) and consider the completed $L$-series 
\[
    \Lambda(\f_\kappa^\circ\otimes\g_\lambda^\circ\otimes\h_\mu^\circ, s) := L_{\infty}(\f_\kappa^\circ\otimes\g_\lambda^\circ\otimes\h_\mu^\circ,s)\, L(\f_\kappa^\circ\otimes\g_\lambda^\circ\otimes\h_\mu^\circ,s)\,.
\]
The completed $L$-series $\Lambda(\f_\kappa^\circ\otimes\g_\lambda^\circ\otimes\h_\mu^\circ,s+c)$ coincides with Garrett's automorphic $L$-function $L(s+1/2,\Pi)$ associated with the irreducible unitary automorphic triple product representation $\Pi := \pi_f\times\pi_g\times \pi_h$ of $\GL_2(\mathbb A) \times\GL_2(\mathbb A) \times\GL_2(\mathbb A)$. In particular, by the work of Piatetski-Shapiro and Rallis~\cite{PSR87}, there exists $\varepsilon(\f_\kappa^\circ\otimes\g_\lambda^\circ\otimes\h_\mu^\circ)\in \{\pm 1\}$ and a positive integer $N(\f_\kappa^\circ\otimes\g_\lambda^\circ\otimes\h_\mu^\circ)$ such that we have a functional equation
\begin{equation}\label{eqn:functionaltriple}
    \Lambda(\f_\kappa^\circ\otimes\g_\lambda^\circ\otimes\h_\mu^\circ,s) = \varepsilon(\f_\kappa^\circ\otimes\g_\lambda^\circ\otimes\h_\mu^\circ) \cdot N(\f_\kappa^\circ\otimes\g_\lambda^\circ\otimes\h_\mu^\circ)^{c-s} \cdot \Lambda(\f_\kappa^\circ\otimes\g_\lambda^\circ\otimes\h_\mu^\circ, 2c-s)\,.
\end{equation}

\subsubsection{}\label{subsubsec_2022_06_02_0902}  We are especially interested in the central critical value $s=c(x)$. We review Hsieh's $p$-optimal construction of his $p$-adic $L$-functions in \cite{Hsieh} which interpolate this value $x$ varies. To that end, for $?=\f,\bal$, we define the modified Euler factors $\mathcal E_p^?(x):=\mathcal E_p^?(\f_\kappa^\circ\otimes \g_\lambda^\circ\otimes \h_\mu^\circ,c)$ on setting
\begin{align} \label{eqn:Eulertriple}
\begin{aligned}
    \mathcal E_p^?(x)= \begin{cases}
 \left(1-\frac{\beta_{\f_\kappa^\circ}\alpha_{\g_\lambda^\circ}\alpha_{\h_\mu^\circ}}{p^c}\right)^2
\left(1-\frac{\beta_{\f_\kappa^\circ}\alpha_{\g_\lambda^\circ}\beta_{\h_\mu^\circ}}{p^c}\right)^2
\left(1-\frac{\beta_{\f_\kappa^\circ}\beta_{\g_\lambda^\circ}\alpha_{\h_\mu^\circ}}{p^c}\right)^2
\left(1-\frac{\beta_{\f_\kappa^\circ}\beta_{\g_\lambda^\circ}\beta_{\h_\mu^\circ}}{p^c}\right)^2 & \text{ if } ?=\f\\
\left(1-\frac{\alpha_{\f_\kappa^\circ}\beta_{\g_\lambda^\circ}\beta_{\h_\mu^\circ}}{p^c}\right)^2
\left(1-\frac{\beta_{\f_\kappa^\circ}\alpha_{\g_\lambda^\circ}\beta_{\h_\mu^\circ}}{p^c}\right)^2
\left(1-\frac{\beta_{\f_\kappa^\circ}\beta_{\g_\lambda^\circ}\alpha_{\h_\mu^\circ}}{p^c}\right)^2
\left(1-\frac{\beta_{\f_\kappa^\circ}\beta_{\g_\lambda^\circ}\beta_{\h_\mu^\circ}}{p^c}\right)^2 & \text{ if } ?=\bal
\end{cases}
\end{aligned}
\end{align}
as in op. cit., where $\alpha_{\f_\kappa^\circ}$ and $\beta_{\f_\kappa^\circ}$ denote the roots of the Hecke polynomial of $\f_\kappa^\circ$ at $p$ with $v_p(\alpha_{\f_\kappa^\circ})=0$ (and we similarly define $\{\alpha_{\g_\lambda^\circ},\beta_{\g_\lambda^\circ}\}$ and $\{\alpha_{\h_\mu^\circ},\beta_{\h_\mu^\circ}\}$).
Let us denote by $\Sigma_\mathrm{exc}$ the set of primes introduced in \cite[Equation (1.5)]{Hsieh}. Also as in op. cit., let $\Sigma^-$ denote the set of primes $v\mid N$ where the local root number $\varepsilon_v(x)$ equals $-1$ (the set $\Sigma^-$ is independent of the choice of $x\in \cW_3^{\rm cl}$).

\begin{theorem}[Hsieh]
\label{thm:unbalancedinterpolation}
    Assume that $N$ is square-free and that $\Sigma^- = \emptyset$. Then there exists a unique element $\cL_p^\hf(\hf\otimes\hg\otimes\hh)\in \mathcal R_3$ with the following interpolative property:
    \[
        (\cL_p^{\hf}(\hf\otimes\hg\otimes\hh)(x))^2 = \mathcal E_p^\hf(x)^2  \cdot (\sqrt{-1})^{-2{\rm w}(\kappa)} \cdot \mathfrak C_{\rm exc}(\f\otimes\g\otimes\h)\cdot  \frac{\Lambda(\f_\kappa^\circ\otimes \g_\lambda^\circ\otimes \h_\mu^\circ, c)}{  \Omega_{\f_\kappa}^2 }\,,
    \]
    where $\mathfrak C_{\rm exc}(\f\otimes\g\otimes\h) =  \prod_{q\in\Sigma_\mathrm{exc}(\hf\otimes\hg\otimes\hh)} (1+q^{-1})^2$.
\end{theorem}
Exchanging the roles of $\f$, $\g$, and $\h$, we analogously have the $p$-adic $L$-functions $\cL_p^?(\hf\otimes\hg\otimes\hh)$ for $?=\g,\h$. We defer the reader to \cite[Theorem B]{Hsieh} for the balanced case, which we will not need in this work.

\comm{\color{red}
\begin{theorem}[Hsieh]
\label{thm:balancedinterpolation}
    Assume that $N$ is square-free and suppose that it factors as $N=N^+N^-$, with $\gcd(N^+, N^-)=1$, where $N^- = \prod_{\ell\in \Sigma^-}\ell$. Assume further that $\Sigma^-$ has odd cardinality and that all three residual representations $\overline{\rho}_\hf$, $\overline{\rho}_\hg$, and $\overline{\rho}_\hh$ are ramified at all primes $\ell\in \Sigma^-$ with $\ell\equiv 1\pmod{p}$. Then there exists a unique element $\cL_p^\bal(\hf\otimes\hg\otimes\hh)\in \mathcal R_3$ with the following interpolative property for all $x = (\kappa,\lambda,\mu) \in \cW_3^\bal$:
    \[
        (\cL_p^{\bal}(\hf\otimes\hg\otimes\hh)(x))^2 = \mathcal E_p^\bal(x)^2 \cdot  (\sqrt{-1})^{1-\rmw(\kappa)-\rmw(\lambda)-\rmw(\mu)} \cdot \mathfrak C_{\rm exc}\cdot \frac{\Lambda(\f_\kappa^\circ\otimes\h_\lambda^\circ\otimes\h_\mu^\circ, c)}{\Omega_{\f_\kappa} \Omega_{\g_\lambda} \Omega_{\h_\mu} }\,.
    \]
\end{theorem}
}

\subsection{\texorpdfstring{$L$}{}-series for \texorpdfstring{$f\otimes \Ad^0(g)$}{} and \texorpdfstring{$p$}{}-adic interpolation} \label{sec:Adjoint}
We assume in \S\ref{sec:Adjoint} that $\varepsilon_\f=1$ and retain the notation and conventions of \S\ref{subsec:tripleproducts}, concentrating in the scenario when $\h=\g^c:=\g\otimes \varepsilon_\g^{-1}$ is the conjugate Hida family. 


\subsubsection{} When $y\in \cW_2^?\subset \cW_2^\cris$ ($?=\f,{\rm ad}$), we define the modified Euler factors
\begin{align*}
\begin{aligned} 
  \mathcal E^?_p(\kappa,\lambda,\Ad):=
\begin{cases}
  (1-\beta_{\hf_\kappa^\circ}\alpha_{\hg_\lambda^\circ}\beta_{\hg_\lambda^\circ}^{-1}p^{-\frac{\rmw(\kappa)}{2}})
  (1-{\beta_{\hf_\kappa^\circ}}{p^{-\frac{\rmw(\kappa)}{2}}}) 
  (1-{\beta_{\hf_\kappa^\circ}\beta_{\hg_\lambda^\circ}}{\alpha_{\hg_\lambda^\circ}^{-1}p^{-\frac{\rmw(\kappa)}{2}}}),& \text{ if } ?=\hf \\
 (1-{\alpha_{\hf_\kappa^\circ}\beta_{\hg_\lambda^\circ}}{\alpha_{\hg_\lambda^\circ}^{-1}p^{-\frac{\rmw(\kappa)}{2}}})
 (1-{\beta_{\hf_\kappa^\circ}}{p^{-\frac{\rmw(\kappa)}{2}}}) 
 (1-{\beta_{\hf_\kappa^\circ}\beta_{\hg_\lambda^\circ}}{\alpha_{\hg_\lambda^\circ}^{-1}}p^{\frac{-\rmw(\kappa)}{2}}),
 & \text{ if } ?=\Ad
\end{cases}
\end{aligned}
\end{align*}
where $\alpha_{\bullet}$ and $\beta_{\bullet}$ are as in \S\ref{subsubsec_2022_06_02_0902}. When $\kappa$ is non-crystalline, we put $\mathcal E^?_p(\kappa,\lambda,\Ad)=0$.

\subsubsection{} 
\label{subsubsec_conj_2022_06_02_0940}
The following conjecture is in the spirit of Hsieh's $p$-optimal construction of triple product $p$-adic $L$-functions (which in turn dwells on the Coates--Perrin-Riou formalism):

\begin{conj}  
\label{conj_2022_06_02_0940}
    There exist $p$-adic $L$-functions $\cL_p^{?}(\hf\otimes \Ad^0\hg )\in \mathcal R_2[\frac{1}{p}]$ (where $?=\f,\,\Ad$) with the following interpolative properties  for all crystalline specializations $y=(\kappa,\lambda)\in \cW_2^{?}$:
    \begin{align}
    \label{eqn_conj_2022_06_02_0940}
        \cL_p^{?}(\hf\otimes \Ad^0\hg )(y) =
        \begin{cases}
         C_{\f_\kappa}^- \cdot  \mathcal E^?_p(\kappa,\lambda,\Ad) \cdot   \dfrac{\Lambda(\hf_\kappa^\circ\otimes\Ad^0\hg_\lambda^\circ,\frac{\wt(\kappa)}{2})}{\Omega_{\hf_\kappa}^- \cdot \Omega_{\hf_\kappa}^2 } & \text{ if } ?=\hf\,,\\
    C_{\f_\kappa}^- \cdot  \mathcal E^?_p(\kappa,\lambda,\Ad) \cdot  \dfrac{\Lambda(\hf_\kappa^\circ\otimes\Ad^0\hg_\lambda^\circ,\frac{\wt(\kappa)}{2})}{\Omega_{\hf_\kappa}^- \cdot \Omega_{\hg_\lambda}^2}  & \text{ if } ?=\Ad\,.
      \end{cases}
    \end{align}
\end{conj}

 We note that the importance of $p$-adic $L$-function $\cL_p^\Ad(\hf\otimes \Ad^0\hg )$ in this article is due to its appearance in \cite[Conjecture 2.1]{BS_Part1}. 
Even though Conjecture~\ref{conj_2022_06_02_0940} is presently open, we prove it in Proposition~\ref{prop:factorizationAdBDP}(ii) and Corollary~\ref{eqn_cor_2023_01_09_1840}  when the family $\g$ has CM. 

\subsection{\texorpdfstring{$\hbox{\rm BDP}^2$}{} \texorpdfstring{$p$}{}-adic \texorpdfstring{$L$}{}-function}
\label{subsubsec_BDPsquared}
We assume throughout \S\ref{subsubsec_BDPsquared} that the primitive Hida family $\hg$ has CM by the imaginary quadratic field $K$. Note in this case that $p\cO_K=\p\p^c$ necessarily splits, where $c\in \Gal(K/\Q)$ is the unique non-trivial automorphism. We assume that $\p$ is the prime of $K$ that is determined by our fixed embedding $\iota_p:\overline{\Q}\hookrightarrow \overline{\Q}_p$. We let $\Gamma_\p$ denote the Galois group of the unique $\Zp$-extension of $K$ unramified outside $\p$ and let us denote by $\chi_\p: \Gamma_\p\stackrel{\sim}{\lra} 1+p\ZZ_p$ the canonical character. Our assumption that $\g$ has CM (alongside with our running hypotheses \ref{item_Irr} and \ref{item_Dist} on $\g$) allows us to identify $\cR_\g$ with a finite flat extension of $\ZZ_p[[\Gamma_\p]]$ and $T_\g$ with ${\rm Ind}_{K/\Q}\Psi_\hg$ for some $\cR_\g$-valued character $\Psi_\hg: G_K \to \cR_\hg^\times$. Let us denote by $H_{\mathfrak f_\hg \p^\infty}:=\varprojlim_n H_{\mathfrak f_\hg \p^n}$ the Galois group of the ray class field of conductor $\mathfrak f_\hg \p^\infty$ through which $\Psi_\hg$ factors. 

\subsubsection{} One can describe the correspondence between $\g$ and $\Psi_\g$ in more explicit terms. Let us put
$$   \Theta_{\mathfrak f_\hg} := \sum_{\mathfrak a \colon (\mathfrak a,\mathfrak p\mathfrak f_\hg)=1} [\mathfrak a] q^{N_{K/\Q}(\mathfrak a)} \in \Lambda(H_{\mathfrak f_\hg \p^\infty})[[q]]\,,$$ 
where $[\mathfrak a]\in H_{\mathfrak f_\g \p^\infty}$ is the element that corresponds to the ideal $\mathfrak a$ via the tower of arithmetically normalized Artin reciprocity maps. Then $\hg = \Psi_\hg(\Theta_{\mathfrak f_\g})$. For a given specialization $\lambda$ of weight ${\rm w}(\lambda)$, we have that $\Psi_{\g_\lambda}:= \lambda\circ \Psi_\hg$ is the $p$-adic avatar of an algebraic Hecke character with infinity type $({\rm w}(\lambda)-1,0)$. One can extend $\Psi_{\g_\lambda}$ to a character of conductor $\mathfrak f_\hg$ on setting $\Psi_{\hg_\lambda}^\circ(\p):= \varepsilon_\hg(p)p^{\wt(\lambda)-1}/\Psi_{\hg_\lambda}(\p^c)$. The theta-series of the character $\Psi_{\hg_\lambda}^\circ$   is the newform $\hg_\lambda^\circ$ of level $D_KN^K_\Q(\mathfrak f_\g)$ associated to $\hg_\lambda$.

\subsubsection{}
\label{subsubsec_812_2022_12_21_1028}
Let us consider two CM Hida families $\hg$ and $\hh$, and let us denote by $\hh^c$ the conjugate family. As above, we have the associated characters $\Psi_\hg: G_K \to \mathcal R_\hg^\times$ and $\Psi_{\hh^c}: G_K \to \mathcal R_\hh^\times$. We consider the $(\mathcal R_\hg \, \widehat{\otimes}_{\ZZ_p} \, \mathcal R_\h)^\times$-valued characters
\[
     {\displaystyle \Phi := \Psi_\g\Psi_{\h^c}^{c,-1} (\bbchi_{\h^c}\chi_\cyc^{-1})_{\vert_{G_K}}}\,\,,
    \qquad {\displaystyle \Phi' := \Psi_\g\Psi_{\h^c}^{-1} } \,.
\] 
We will denote by $\lambda\otimes \mu$ the specialization of $\mathcal R_\hg\otimes\mathcal R_\hh$ associated to the pair $(\lambda,\mu)$. The specialization $\Phi_{\lambda\otimes\mu}:=\Psi_{\hg_\lambda}\Psi_{\hh_\mu^c}^{c,-1} \chi_\cyc^{\wt(\mu)-1}{}_{\vert_{G_K}}$ (resp. $\Phi_{\lambda\otimes\mu}':=\Psi_{\hg_\lambda}\Psi_{\hh_\mu^c}^{-1}$) is the $p$-adic avatar of a Hecke character with infinity type  $(\wt(\lambda)+\wt(\mu)-2,0)$ (resp. $(\wt(\lambda)-\wt(\mu),0)$), and with the prime-to-$p$ part of its conductor $\mathfrak f$ (resp. $\mathfrak f'$), where $\mathfrak{f} \mid \mathfrak{f}_\fg \mathfrak{f}_{\hh}^c$ and $\mathfrak{f}'\mid \mathfrak{f}_\fg \mathfrak{f}_{\hh}$.  

The characters $\Phi$ and $\Phi'$ give rise to a pair of primitive Hida families $\Phi(\Theta_{\mathfrak{f}})$ and $\Phi'(\Theta_{\mathfrak f'})$. To ease our notation, we shall write in what follows $\Phi(\Theta)$ and $\Phi'(\Theta)$ in place of $\Phi(\Theta_{\mathfrak{f}})$ and $\Phi'(\Theta_{\mathfrak f'})$, respectively.

\subsubsection{} For every crystalline point  $x=(\kappa,\lambda,\mu)\in \cW_3^\hg$ (so that ${\rm w}(\lambda) \geq {\rm w}(\kappa)+{\rm w}(\mu)$), we define
\[
    c=\frac{{\rm w}(\kappa)+{\rm w}(\lambda)+{\rm w}(\mu)}{2}-1 \qquad \text{and} \qquad c':= c+1-{\rm w}(\mu) = \frac{{\rm w}(\kappa)+{\rm w}(\lambda)-{\rm w}(\mu)}{2}\,.
\]
The specialization of the representation $T_3^\dagger$ at $x=(\kappa, \lambda,\mu)$ has the following form:
    \begin{align} \label{eqn:factorizationtriplespecialized}
    \begin{split}
        T_3^\dagger(x) 
        &= T_{\hf_\kappa}\otimes T_{\hg_\lambda}\otimes T_{\hh_\mu} (1-c) 
        \cong T_{\hf_\kappa}\otimes T_{\hg_\lambda}\otimes T_{\hh_\mu^c}^* (1-c') 
        \\
        &\cong 
        T_{\hf_\kappa}\otimes \Bigl( \Ind_{K/\Q}(\Psi_{\hg_\lambda}\Psi_{\hh_\mu^c}^{c,-1}) \oplus \Ind_{K/\Q}(\Psi_{\hg_\lambda}\Psi_{\hh_\mu^c}^{-1}) \Bigr) (1-c') 
        \\
        &\cong
        \Bigl(T_{\hf_\kappa}\otimes \Ind_{K/\Q}(\Phi_{\lambda\otimes\mu}) (1-c) \Bigr) \oplus 
        \Bigl( T_{\hf_\kappa} \otimes \Ind_{K/\Q}(\Phi'_{\lambda\otimes\mu}) \Bigr) (1-c')\,.
    \end{split} 
\end{align}
The decomposition \eqref{eqn:factorizationtriplespecialized} implies that we have
\[
    \Lambda(\hf_\kappa^\circ\otimes \hg_\lambda^\circ\otimes \hh_\mu^\circ, c) = 
    \Lambda(\hf_\kappa^\circ\otimes \Phi(\Theta_{\mathfrak f})_{\lambda\otimes\mu}^\circ,c) 
    \cdot 
    \Lambda(\hf_\kappa^\circ\otimes \Phi'(\Theta_{\mathfrak f'})_{\lambda\otimes \mu}^\circ, c' )\,.
\]


\subsubsection{}  
For each $\pmb{\xi} \in \{\Phi,\Phi'\}$, let us denote by $\mathcal R_{\pmb\xi}$ the branch of Hida's universal (cuspidal) Hecke algebra of tame level $N_\xi=D_K \mathbf{N}_{K/\Q}(\mathfrak f_\xi)$ associated with the cuspidal Hida family $\hg_{\pmb\xi}$. Let us denote by $\mathfrak h_{N_\xi}^{\ord}$ Hida's universal $p$-ordinary Hecke algebra of tame level $N_\xi$. We then have the following natural ring homomorphism
\begin{equation} \label{map}
    \mathfrak h_{N_\xi}^{\ord} \to \cR_\g\hatotimes_{\Zp} \cR_\h
\end{equation}
given by
\begin{align*}
    &T_\ell \longmapsto \xi(\lambda)+\xi(\lambda^c), \qquad\hbox{ if } \ell\nmid N_\xi p \hbox{ and }  \ell \cO_K = \lambda\lambda^c \hbox{ splits},\\
    &T_\ell\longmapsto 0 \qquad\qquad\qquad\quad\,\, \hbox{ if } \ell\nmid N_\xi p \hbox{ is inert in } K/\QQ,\\
    &U_\ell \longmapsto \xi(\lambda) \qquad\qquad\qquad\hbox{ if } \hbox{ $\lambda\mid \ell$ and $\lambda \mid D_K\mathfrak{f}_\xi^c$},\\
    &U_p\longmapsto \xi(\p) {}\,.
\end{align*}
The map \eqref{map} necessarily factors through the unique irreducible local component corresponding to the Hida family $\g_\xi$, therefore gives rise to a morphism
\[
    u_{\pmb\xi} \colon \mathcal R_{\pmb\xi} \lra \mathcal R_\g \hatotimes_{\ZZ_p} \mathcal R_\h,
\]
of complete local Noetherian rings. For $\lambda\otimes\mu \in \mathcal W_\g\times\mathcal W_\h$, let us put $\nu:=u_{\pmb\xi}^*(\lambda\otimes\mu) \in  \cW_\xi$. We then have $\wt(\nu)=\wt(\lambda)+ (\wt(\mu)-1)$ (resp. $\wt(\nu)=\wt(\lambda)- (\wt(\mu)-1)$) if $\pmb{\xi}=\Phi$ (resp. if $\pmb{\xi}=\Phi'$). Moreover, $\hg_{\pmb\xi, \nu} = \pmb\xi(\Theta)_{\lambda\otimes\mu}$\,.

\subsubsection{} Given $\gamma\in \Gamma_\cyc$, let us denote by $[\gamma]\in \LL(\Gamma_\cyc)$ the corresponding group like element. Let us consider the morphism
\begin{align*}
    {\rm pr}^\dagger\,:\quad \cR_\hf\hatotimes_{\Zp}  
    \mathcal R_\hg \hatotimes_{\Zp} \mathcal R_\hh \hatotimes_{\Zp}  \Lambda(\Gamma_\cyc) &\lra \cR_\hf\hatotimes_{\Zp}  \mathcal R_\hg \widehat\otimes_{\Zp}  \mathcal R_\hh=:\cR_3 \\
    a\,\otimes \, b\, \otimes\, c\,\otimes\, [\gamma] \, & \longmapsto \,(\chi_\cyc(\gamma)^{2}\bbchi_3^{-\frac{1}{2}}(\gamma)) \left(a\,\, \otimes \,b\,\otimes\, c\right)
\end{align*}
so that ${\rm pr}^\dagger$ is the unique map for which we have
$$(T_\hf\hatotimes_{\Zp}  
    T_\hg \hatotimes_{\Zp} T_\hh \hatotimes_{\Zp}  \Lambda(\Gamma_\cyc))\otimes_{{\rm pr}^\dagger} \cR_3=T_3^\dagger$$
    as $G_\QQ$-representations. Recall also $\bbchi_3:=\bbchi_{\f} \otimes\bbchi_{\g}\otimes \bbchi_{\h}$ given as in \S\ref{subsubsec_2022_05_16_1506}, as well as its square-root in \S\ref{subsubsec_211_2022_06_01_1635}\,.

\subsubsection{} 
We define the \emph{${\rm BDP}^2$ $p$-adic $L$-functions} on setting
\begin{align*}
    \cL_p^{\rm BDP}(\hf_{/K}\otimes \Psi_\hg\Psi_{\hh^c}^{-1}) &:= {\rm pr^\dagger}\circ ({\rm id}_{\cR_\hf} \otimes u_{\Phi'} \otimes {\rm id}_{\Lambda(\Gamma_{\rm cyc})})\bigl( \cL_p^{\Phi'(\Theta)}(\hf\otimes \Phi'(\Theta) ) \bigr)  \in \cR_\hf \hatotimes (\mathcal R_\hg \hatotimes \mathcal R_\hh)\,, \\
    \cL_p^{\rm BDP}(\hf_{/K}\otimes \Psi_\hg\Psi_{\hh^c}^{c,-1}) &:= {\rm pr^\dagger}\circ  ({\rm id}_{\cR_\hf}\otimes u_{\Phi} \otimes {\rm id}_{\Lambda(\Gamma_{\rm cyc})})\bigl( \cL_p^{\Phi(\Theta)}(\hf\otimes \Phi(\Theta) ) \bigr)  \in \cR_\hf \hatotimes (\mathcal R_\hg \hatotimes \mathcal R_\hh).
\end{align*}
A direct calculation (based on Theorem~\ref{thm:Hida-interp}) shows that we have the interpolation formulae  
\begin{align}
\begin{aligned}
\label{eqn_2022_12_07_1301}
    \cL_p^{\rm BDP}(\hf_{/K} \otimes  \Psi_\hg\Psi_{\hh^c}^{-1})(\kappa,\lambda\otimes\mu) & =   \mathcal E_p^{\Phi'(\Theta)}\bigl(\hf_\kappa^\circ \otimes \Phi'(\Theta)_{\lambda\otimes\mu}^\circ, c'\bigr) \cdot (\sqrt{-1} )^{1-\wt(\lambda)+\wt(\mu)} \cdot 
    \mathfrak C_{\rm exc}(\hf\otimes\Phi'(\Theta))
  \\
  &\hspace{2cm}\times \frac{\Lambda(\hf_\kappa^\circ \otimes \Phi'(\Theta)_{\lambda\otimes\mu}^\circ, c')}{\Omega_{\Phi'(\Theta)_{\lambda\otimes\mu}}}, \qquad  \wt(\lambda)-\wt(\mu)+1 > \wt(\kappa)\,,
\end{aligned}
\end{align}
(see \eqref{eqn:modifiedHidaperiods} for the definition of the canonical Hida period $\Omega_{\Phi'(\Theta)_{\lambda\otimes\mu}}$. Notice that the element $\mathfrak c_{\Phi'(\Theta)}$ makes sense also in our CM setting, cf. \cite[Remark 4.6(ii)B]{AIF_Hida88}, eventually locally, if $\Phi'(\Theta)$ has residually reducible representation) and likewise,
\begin{align}
\begin{aligned}
\label{eqn_2022_12_07_1302}
    \cL_p^{\rm BDP}(\hf_{/K} \otimes  \Psi_\hg\Psi_{\hh^c}^{c,-1})(\kappa,\lambda\otimes\mu) & =   \mathcal E_p^{\Phi(\Theta)}\bigl(\hf_\kappa^\circ \otimes \Phi(\Theta)_{\lambda\otimes\mu}^\circ, c\bigr) \cdot (\sqrt{-1} )^{3-\wt(\lambda)-\wt(\mu)} 
    \cdot \mathfrak C_{\rm exc}(\hf\otimes\Phi(\Theta)) 
   \\ 
    &\hspace{2cm} \times \frac{\Lambda(\hf_\kappa^\circ \otimes \Phi(\Theta)_{\lambda\otimes\mu}^\circ, c)}{\Omega_{\Phi(\Theta)_{\lambda\otimes\mu}}}\,, \qquad \wt(\lambda)+\wt(\mu)-1 > \wt(\kappa)\,.
\end{aligned}
\end{align}
Here $\mathfrak{C}_{\rm exc}(\hf\otimes\Phi(\Theta))$ and $\mathfrak{C}_{\rm exc}(\hf\otimes\Phi'(\Theta))$ are given in the statement of Theorem~\ref{thm:Hida-interp}, and
\begin{align}
\label{eqn:EulerBDP}
\begin{aligned}
\mathcal E_p^{\Phi'(\Theta)}({\f_\kappa^\circ} \otimes {\Phi'(\Theta)^\circ_{\lambda\otimes\mu}},c')  
&=\left(1-\frac{\alpha_{\f_\kappa^\circ}\beta_{\g_\lambda^\circ} \beta_{\h_\lambda^\circ} }{p^c}\right)^2 \left(1-\frac{\beta_{\f_\kappa^\circ}\beta_{\g_\lambda^\circ} \beta_{\h_\lambda^\circ} }{p^c}\right)^2\,, \\
\mathcal E_p^{\Phi(\Theta)}({\f_\kappa^\circ} \otimes {\Phi(\Theta)^\circ_{\lambda\otimes\mu}},c)  
&=\left(1-\frac{\beta_{\f_\kappa^\circ}\beta_{\g_\lambda^\circ} \alpha_{\h_\lambda^\circ} }{p^c}\right)^2
\left(1-\frac{\beta_{\f_\kappa^\circ}\beta_{\g_\lambda^\circ} \beta_{\h_\lambda^\circ} }{p^c}\right)^2
\,,
\end{aligned}
\end{align}
which we obtain by working out the Euler-like factors in \eqref{eqn_20220601_1700}.

Even though the interpolation regions of the $p$-adic $L$-functions $\cL_p^{\rm BDP}(\hf_{/K}\otimes \Psi_\hg\Psi_{\hh^c}^{-1})$ and $\cL_p^{\rm BDP}(\hf_{/K}\otimes \Psi_\hg\Psi_{\hh^c}^{c,-1}) $ do not coincide, they both contain the set of $\g$-unbalanced crystalline points $\cW_3^\g \cap \cW_3^{\rm crys}$.


\section{Super-factorization of \texorpdfstring{$p$}{}-adic \texorpdfstring{$L$}{}-functions}
\label{subsec_75_2022_06_02_0900}
$\,$
    The goal of this section is to prove Theorem~\ref{thm_main_6_plus_2} (on passing to sufficiently small wide-open discs in the weight spaces for the families $\f$ and $\g$) in the scenario when the family $\hg$ has complex multiplication.

\subsection{The Factorization of the degree-8 \texorpdfstring{$p$}{}-adic \texorpdfstring{$L$}{}-function}
\label{sec:8=4+4}
Our goal in this subsection is to prove that when both $\hg$ and $\hh$ are CM (without requiring that $\hh=\hg^c$), the square of the $\g$-dominant $p$-adic $L$-function attached to the triple $(\hf, \hg, \hh)$ of Hida families is factorized into a product of two ${\rm BDP}^2$ $p$-adic $L$-functions. 



\begin{theorem} \label{thm:analyticfactorization}
    Suppose that $\hg$ and $\hh$ are two Hida families with CM by the same quadratic imaginary field $K$. We then have the following factorization of $p$-adic $L$-functions:
    \[
        \cL_p^\hg(\hf\otimes\hg\otimes \hh)^2(\kappa,\lambda,\mu) = \mathcal A(\lambda,\mu) \cdot \cL_p^{\rm BDP}(\hf_{/K}\otimes \Psi_\hg \Psi_{\hh^c}^{c,-1}) (\kappa, \lambda\otimes\mu)\cdot \cL_p^{\rm BDP}(\hf_{/K}\otimes \Psi_\hg \Psi_{\hh^c}^{-1} )(\kappa, \lambda\otimes\mu),
    \]
    where $\mathcal A \in \mathcal {\rm Frac}(R_\hg\hatotimes_{\ZZ_p} \cR_\hh)^\times$ is non-zero and assumes algebraic values at every classical specialization $\lambda \otimes \mu$ with $\wt(\lambda)\geq \wt(\mu)+2$.
\end{theorem}

\begin{proof}
   We will compare both sides of the claimed equality evaluated at the overlapping portions of the interpolation ranges. It follows from equation \eqref{eqn:factorizationtriplespecialized} that, for every crystalline point  $x=(\kappa,\lambda,\mu)\in \cW_3^\hg$, we have the following factorization of algebraic $L$-values:
  \begin{equation}
        \label{eqn_2022_12_07_1257}
        \frac{\Lambda(\hf_\kappa^\circ\otimes \hg_\lambda^\circ\otimes \hh_\mu^\circ, c)}{ \Omega_{\hg_\lambda}^2 }
        = 
        \frac{\Omega_{ \Phi(\Theta)_{\lambda\otimes\mu} }\Omega_{ \Phi'(\Theta)_{\lambda\otimes\mu} }}{\Omega_{\Psi_\hg(\Theta)_{\lambda}}^2} 
        \cdot \frac{\Lambda(\hf_\kappa^\circ\otimes \Phi(\Theta)_{\lambda\otimes\mu}^\circ, c)}{\Omega_{ \Phi(\Theta)_{\lambda\otimes\mu} }}
        \cdot 
        \frac{\Lambda(\hf_\kappa^\circ\otimes \Phi'(\Theta)_{\lambda\otimes \mu}^{\circ}, c' )}{\Omega_{ \Phi'(\Theta)_{\lambda\otimes\mu}}} \in \barQ.
     \end{equation}
    It follows from \eqref{eqn:Eulertriple} and \eqref{eqn:EulerBDP} that $\mathcal E_p^\g(\hf_\kappa^\circ\otimes \hg_\lambda^\circ\otimes \hh_\mu^\circ)^{2} = \mathcal E_p(\hf_\kappa^\circ\otimes \Phi(\Theta)_{\lambda\otimes\mu}^\circ, c) \cdot
    \mathcal E_p(\hf_\kappa^\circ\otimes \Phi'(\Theta)_{\lambda\otimes \mu}^{\circ}, c')$.
 Comparing Theorem~\ref{thm:unbalancedinterpolation}, \eqref{eqn_2022_12_07_1301} and \eqref{eqn_2022_12_07_1302} together with \eqref{eqn_2022_12_07_1257} above, we infer that
    \begin{align} \label{eqn:interpolationtriple}
    \begin{aligned}
         (\cL_p^{\hg}(\hf\otimes \hg\otimes\hh)(x))^2 
         & = 
         \mathfrak C_{\rm exc}(\hf\otimes\hg\otimes\hh) \times   
         \frac{\Omega_{ \Phi(\Theta)_{\lambda\otimes\mu}  }\Omega_{ \Phi'(\Theta)_{\lambda\otimes\mu}  }}{\Omega_{\Psi_\hg(\Theta)_{\lambda}}^2}
         \\
         &\hspace{1.5cm}  \times (\sqrt{-1} )^{3-\wt(\lambda)-\wt(\mu)}\times \mathcal E_p(\hf_\kappa^\circ\otimes \Phi(\Theta)_{\lambda\otimes\mu}^\circ, c) \times 
         \frac{\Lambda(\hf_\kappa^\circ\otimes \Phi(\Theta)_{\lambda\otimes \mu}, c )}{\Omega_{ \Phi(\Theta)_{\lambda\otimes\mu}^\circ}}  \\
         &\hspace{2cm}\times 
         (\sqrt{-1} )^{1-\wt(\lambda)+\wt(\mu)}
         \cdot  \mathcal E_p(\hf_\kappa^\circ\otimes \Phi'(\Theta)_{\lambda\otimes\mu}^\circ, c) \cdot
         \frac{\Lambda(\hf_\kappa^\circ\otimes \Phi'(\Theta)_{\lambda\otimes \mu}, c' )}{\Omega_{ \Phi'(\Theta)_{\lambda\otimes\mu}^\circ}} 
         \\ 
         & = A(\lambda,\mu) \, \times \cL_p^{\rm BDP}(\hf_{/K}\otimes \Psi_\hg \Psi_{\hh^c}^{c,-1})(\kappa, \lambda\otimes\mu) \times \cL_p^{\rm BDP}(\hf_{/K}\otimes \Psi_\hg \Psi_{\hh^c}^{-1})(\kappa, \lambda\otimes\mu)\,,
         \end{aligned}
    \end{align}
    where
    \[
        A(\lambda,\mu) := \frac{ \mathfrak C_{\rm exc}(\hf\otimes\hg\otimes\hh) }{\mathfrak C_{\rm exc}(\hf\otimes\Phi(\Theta)) \cdot \mathfrak C_{\rm exc}(\hf\otimes\Phi'(\Theta))} 
        \cdot 
        \frac{\Omega_{ \Phi(\Theta)_{\lambda\otimes\mu} }\Omega_{ \Phi'(\Theta)_{\lambda\otimes\mu} }}{\Omega_{\Psi_\hg(\Theta)_{\lambda}}^2}.
    \]
    We remark that $A(\lambda,\mu)$ does not depend on $\kappa$ since each one of the three factors $\mathfrak C_{\rm exc}$ is an explicit rational constant which depends only on the triple $(\f,\g,\h)$ (cf. the statements of Theorem \ref{thm:Hida-interp} and Theorem \ref{thm:unbalancedinterpolation}). It is also clear from the definition of $A(\lambda,\mu)$ that it is non-zero and that it can be interpolated (as $\lambda$ and $\mu$ varies) by an element $\cA\in \mathrm{Frac}(\cR_\hg \hatotimes_{\ZZ_p} \cR_\hh)$ since all other factors in the first and final line of \eqref{eqn:interpolationtriple} do. Therefore, the proof of our asserted equality follows from the density of $\hg$-dominant crystalline points $x$ in $\cW_3$, once we show the algebraicity property for
    $$\frac{\Omega_{ \Phi(\Theta)_{\lambda\otimes\mu} }\Omega_{ \Phi'(\Theta)_{\lambda\otimes\mu} }}{\Omega_{\Psi_\hg(\Theta)_{\lambda}}^2}$$ 
    This is clear thanks to the factorization
    $$\frac{\Lambda(\hf_{\kappa_0}\otimes \hg_{\lambda}^\circ\otimes \hh_\mu^\circ, j)}{ \Omega_{\hg_\lambda}^2 }
        = 
        \frac{\Omega_{ \Phi(\Theta)_{\lambda\otimes\mu} }\Omega_{ \Phi'(\Theta)_{\lambda\otimes\mu} }}{\Omega_{\Psi_\hg(\Theta)_{\lambda}}^2} 
        \cdot \frac{\Lambda(\hf_{\kappa_0}\otimes \Phi(\Theta)_{\lambda\otimes\mu}^\circ, j)}{\Omega_{ \Phi(\Theta)_{\lambda\otimes\mu} }}
        \cdot 
        \frac{\Lambda(\hf_{\kappa_0}\otimes \Phi'(\Theta)_{\lambda\otimes \mu}^{\circ}, j)}{\Omega_{ \Phi'(\Theta)_{\lambda\otimes\mu}}} \in \overline{\QQ}^\times$$
of complex $L$-functions,  where we choose a weight-one specialization $\kappa_0$ and a $j$ so that the  $L$-values on the right-hand-side are non-zero critical values (this choice is possible since we assumed $\wt(\lambda)\geq \wt(\mu)+2$).
\end{proof}

\begin{remark}
We will need to compute $\cA(\lambda,\lambda)$ explicitly when $\h=\g^c$ is the conjugate family and $\lambda\in \cW_\g^{\rm crys}$ is a crystalline point. Even though the proof of Theorem~\ref{thm:analyticfactorization} does not offer a direct way to carry out this task, we prove in Proposition~\ref{prop:Alambdalambda} below that $\cA(\lambda,\lambda)$ can be expressed in terms Katz $p$-adic $L$-functions.

\end{remark}

\subsubsection{}
\label{subsubsec_813_2022_12_05}
When $\h=\g^c$ is the conjugate CM family, we put $\mathds{1}(\lambda\otimes\mu):=\Psi_{\hg_\lambda}\otimes \Psi_{\hg_\mu}^{-1}$ and $\Psi^\Ad(\lambda\otimes\mu):= \Psi_{\hg_\lambda}\otimes\Psi_{\hg_\mu}^{c,-1}$. Via the natural morphism (which restricts our attention to specializations of the type $\lambda\otimes\lambda$)
$$\mathcal R_\hg\hatotimes_{\ZZ_p} \mathcal R_{\hg} \lra  \mathcal R_\hg\,,\qquad a\otimes b \mapsto ab\,,$$ 
we define
\begin{align}
    \begin{aligned}
     \label{eqn_2022_05_18_0910}
      \cL_p^{\rm BDP}(\hf_{/K}\otimes \mathds{1})(\kappa,\lambda) &:= \cL_p^{\rm BDP}(\hf_{/K}\otimes \Psi_\hg \Psi_{\hh^c}^{-1})(\kappa, \lambda\otimes\lambda)\,,\\
      \cL_p^{\rm BDP}(\hf_{/K}\otimes \Psi^{\rm ad})(\kappa,\lambda) &:=\cL_p^{\rm BDP}(\hf_{/K}\otimes \Psi_\hg \Psi_{\hh^c}^{c,-1})(\kappa, \lambda\otimes\lambda)\,.
    \end{aligned}
\end{align}
We note that we have the following identification of ${\rm BDP}^2$ $p$-adic $L$-functions:
\begin{equation}
\label{eqn_2022_12_08_1245}
    \cL_p^{\rm BDP}(\hf_{/K}\otimes \mathds{1})(\kappa,\lambda)=\cL_p^{\rm BDP}(\hf_{/K}\otimes \bbpsi)(\kappa, 1)\,,
\end{equation} 
where,
\begin{itemize}
    \item $\bbpsi$ is the Galois character that is associated with the universal Hecke character $\pmb{\psi}$ given as in \cite[\S4.2]{BDV} via the (geometrically normalized) Artin map of class field theory. We remark that the theta series of $\pmb{\psi}$ gives rise to the canonical Hida family $\hg_K$ of tame conductor $-D_K$ and tame character $\epsilon_K$, and that has CM by $K$, and which admits as a specialization the unique  $p$-stabilization $g_{\rm Eis}:={\rm Eis}_1(\epsilon_K)(q)-{\rm Eis}_1(\epsilon_K)(q^p)$ of the Eisenstein series $${\rm Eis}_1(\epsilon_K):=\frac{L(\epsilon_K,0)}{2}+\sum_{n\geq 1}\left(\sum_{d\mid n}\epsilon_K(d)\right)q^n$$ 
    of weight one, whose associated Galois representation is isomorphic to ${\rm Ind}_{K/\QQ}(\epsilon_K)$. 
    \item the specialization $(\kappa,1)$ on the right stands for the specialization of $\hg_K$ to $g_{\rm Eis}$.
\end{itemize}

Notice that specializations of the type $\lambda\otimes\lambda$ may fall within the interpolation range for $\cL_p^{\rm BDP}(\hf_{/K}\otimes \Psi^{\rm ad})(\kappa,\lambda)$, but they never do for $\cL_p^{\rm BDP}(\hf_{/K}\otimes \mathds{1})$. 

Theorem~\ref{thm:analyticfactorization} has the following immediate consequence in terms of the objects we have introduced in \S\ref{subsubsec_813_2022_12_05}:

\begin{corollary} \label{cor:analyticfactorization}
    In the setting of Theorem~\ref{thm:analyticfactorization}, suppose that $\h=\g^c$. 
    We then have the following factorization of $p$-adic $L$-functions: 
    \[
        \cL_p^\hg(\hf\otimes \hg\otimes \hg^c)^2(\kappa,\lambda,\lambda)= \mathcal A(\lambda,\lambda) \cdot
        \cL_p^{\rm BDP}(\hf_{/K} \otimes \Psi^\Ad )(\kappa, \lambda )\cdot \cL_p^{\rm BDP}(\hf_{/K} \otimes \mathds{1})(\kappa,\lambda )\,,
    \]
    where we note that we regard both $\mathds{1}$ and $\Psi^\Ad$ as $\cR_\hg$-valued characters.
\end{corollary}


\subsubsection{}  
\label{subsubsec_2023_01_06_1256}
We are now set to describe $\mathcal A(\lambda,\lambda)$ in explicit terms. Before doing so, we report a result on Petersson products which is essentially due to Hida and Shimura (cf. \cite{shimura76,Hida81}).

Recall that we have denoted by $-D_K$ the discriminant of the quadratic imaginary field $K$. Let us consider a Hecke character $\rho$ of $K$ with infinity type $(\ell-1,0)$ and conductor $\mathfrak f_\rho$. It gives rise to a theta series $\theta(\rho)\in S_\ell(\Gamma_1(M),\chi)$, where $M=D_K \mathbf{N}_{K/\QQ}(\mathfrak f_\rho)$ and $\chi=\varepsilon_K\varepsilon_\rho$ with $\varepsilon_\rho(m) = \rho((m))m^{1-\ell}$ for $m$ coprime to $\mathfrak{f}_\rho$. We denote by $M_\chi$ the conductor of $\chi$.
Being mindful of the different normalizations for the Petersson product we have adopted, a tedious but straightforward calculation that essentially follows from \cite[Theorem 5.1, Eqn. (8.2)]{Hida81} shows
    \begin{align}
    \begin{aligned} \label{eqn_2023_01_04_1800}
        \langle \theta(\rho),\theta(\rho)\rangle 
        &= \mathscr O(\rho)  \cdot \frac{(\ell-1)!}{2^{2\ell-1}\pi^{\ell}} \cdot L_{\mathbf{N}_{K/\QQ}(\mathfrak{f}_\rho)}(\rho/\rho^c,1)\,,
    \end{aligned}
    \end{align}
    where
    \begin{align}
    \begin{aligned} \label{eqn_2023_01_05_1121}\mathscr O(\rho) = \frac{M\prod_{q\mid D_K } (1-q^{-1}) }{\prod_{\substack{q\mid M, \\ q\nmid \frac{M}{M_{\chi}}}} (1-q^{-1})}\times \frac{L_M(\varepsilon_K,1) }{\pi}\,.
    \end{aligned}
    \end{align}
    Notice that $\mathscr O(\rho)\sqrt{D_K}\in \QQ^\times$ by the class number formula.

     \subsubsection{}
     Our goal in this subsection is to explicate $\mathcal A(\lambda,\lambda)$ in terms of the special values of various Katz $p$-adic $L$-functions.
    
    \begin{proposition} \label{prop:Alambdalambda}
        We have
        \[
        \mathcal A(\lambda,\lambda) =   \frac{ \mathfrak C_{\rm exc}(\hf\otimes\hg\otimes\hg^c) }{\mathfrak C_{\rm exc}(\hf\otimes\Phi(\Theta)) \cdot \mathfrak C_{\rm exc}(\hf\otimes\Phi'(\Theta))} 
            \cdot
        \frac{\Omega_{ \Phi(\Theta)_{\lambda\otimes\lambda} } }{\Omega_{\Psi_\hg(\Theta)_{\lambda}}^2}
        \cdot \frac{4h_K(1-p^{-1})}{\omega_K \cdot \mathfrak{c}_{g_{\rm Eis}}} \cdot \log_p(u)\,,
        \]
        where $h_K$ is the class number of $K$, $\omega_K=\#(\cO_K^{\times})_{\rm tor}$, and $u\in \cO_K[1/p]^\times$ is any element such that $(u)=\p^{h_K}$.
    \end{proposition}
    \begin{proof}
Let us consider now the families of theta series associated with the characters $\Psi_\hg$, $\Phi$, and $\Phi'$ given as in Section \ref{subsubsec_812_2022_12_21_1028}. Recall that at crystalline specializations $\lambda\otimes\mu$ of $\cR_\g  \hatotimes_{\ZZ_p}  \cR_\g$ with $\wt(\lambda)>\wt(\mu)$, these families specialize to classical cuspidal modular forms of weights $\wt(\lambda)$, $\wt(\lambda)+\wt(\mu)-1$, and $\wt(\lambda)-\wt(\mu)+1$, respectively. To ease the notation, recall from the Theorem \ref{thm:analyticfactorization} that when $\wt(\lambda)>\wt(\mu)$ we have
    \begin{equation}
        \label{eqn_2023_01_05_1111}
        \mathcal A(\lambda,\mu) := \mathfrak C_{\rm exc} 
        \cdot 
        \frac{\Omega_{ \Phi(\Theta)_{\lambda\otimes\mu} }\Omega_{ \Phi'(\Theta)_{\lambda\otimes\mu} }}{\Omega_{\Psi_\hg(\Theta)_{\lambda}}^2}, \qquad \text{where } \quad \mathfrak C_{\rm exc}:=\frac{ \mathfrak C_{\rm exc}(\hf\otimes\hg\otimes\hg^c) }{\mathfrak C_{\rm exc}(\hf\otimes\Phi(\Theta)) \cdot \mathfrak C_{\rm exc}(\hf\otimes\Phi'(\Theta))}.
    \end{equation}

    Since we have that
    $2\wt(\lambda) = (\wt(\lambda)+\wt(\mu)-1) + (\wt(\lambda)-\wt(\mu)+1)$,  
    it follows from the definition of $\mathcal A(\lambda,\mu)$ (cf. \eqref{eqn_2023_01_05_1111}), the definition of the modified Hida periods given as in Equation \eqref{eqn:modifiedHidaperiods}, together with \eqref{eqn_2023_01_04_1800} that
    \begin{align}
    \begin{aligned}
        \label{eqn_2022_12_08_1510}
         \mathcal A(\lambda,\mu) 
         &=   
         \mathfrak C_{\rm exc} \cdot \mathfrak c(\lambda\otimes\mu) \cdot \frac{\mathscr O({\Phi_{\lambda\otimes\mu}}) \, \mathscr O({\Phi_{\lambda\otimes\mu}'}) }{\mathscr O({{\Psi_{\hg_\lambda} }})^2} \cdot \frac{\mathcal E(\Phi(\Theta)_{\lambda\otimes\mu},\Ad)  \,\mathcal E(\Phi'(\Theta)_{\lambda\otimes\mu},\Ad)}{\mathcal E(\Psi_\hg(\Theta)_\lambda,\Ad)^2} 
        \\
        &\qquad\times 
        \frac{
        (\wt(\lambda)+\wt(\mu)-2)!\,  L(\widehat{\Phi}^{-1}_{\lambda\otimes\mu},0) \,  (\wt(\lambda)-\wt(\mu))!  \, L(\widehat{\Phi}'^{-1}_{\lambda\otimes\mu},0)}{ (\wt(\lambda)-1)!^2 \, L(\widehat{\Psi}_{\hg_\lambda}^{-1},0)^2}\,.
    \end{aligned}
    \end{align}
    for $\wt(\lambda)>\wt(\mu)$, where we have put
    $$\mathfrak c(\lambda\otimes\mu) := \frac{\mathfrak c_\hg(\lambda)^2}{ \mathfrak c_{\Phi(\Theta)}(\lambda\otimes\mu) \, \mathfrak c_{\Phi'(\Theta) }(\lambda\otimes\mu) }  $$
    to simplify our notation. Thanks to a theorem Goldstein and Schappacher (cf. \cite{BDP2}, Proposition 2.11 and Theorem 2.12), it follows that the ratio \eqref{eqn_2022_12_08_1510}  is a non-zero algebraic number. Moreover, for each $\rho \in \{{\Psi}_{\hg_\lambda},{\Phi}_{\lambda\otimes\mu},{\Phi}_{\lambda\otimes\mu}' \}$ of infinity type $(\ell-1,0)$, with $\ell\in \{\wt(\lambda), \wt(\lambda)+\wt(\mu)-1, \wt(\lambda)-\wt(\mu)+1\}$ as above, it is easy to see that 
    \begin{equation}
        \label{eqn_2022_12_08_1518}
        \mathcal E(\theta(\rho),\Ad) = (1-\widetilde\rho(\p)/p)(1-\widetilde\rho^{-1}(\p^c))\,,
    \end{equation} 
    where $\widetilde{\rho}:=\rho\rho^{c,-1}\mathbf{N}_K$. Let us consider a choice of complex CM period $\Omega(K)$ and $p$-adic CM period $\Omega_p(K)$ defined as in \cite{BDP2}, Eqn. (2-15) and Eqn. (2-17) (where they are denoted by $\Omega(A)$ and $\Omega_p(A)$, respectively).  We have the following interpolation formula the Katz $p$-adic $L$-function, cf. \cite[Eqn. (3-1)]{BDP2}: 
     \begin{equation}
        \label{eqn_2022_12_08_1519}
        \mathcal L_p^{\rm Katz}(\widetilde\rho) = \left(\frac{\Omega_p(K)}{\Omega(K)}\right)^{2\ell-2}\left(\frac{\sqrt{D_K}}{2\pi}\right)^{2-\ell} (1-\widetilde\rho(\p)/p)(1-\widetilde\rho^{-1}(\p^c)) \cdot (\ell-1)! \cdot L(\widetilde\rho^{-1},0),\qquad  \forall \ell>1\,.
       \end{equation} 
    Combining \eqref{eqn_2022_12_08_1510}, \eqref{eqn_2022_12_08_1518} and \eqref{eqn_2022_12_08_1519}, we conclude that
    \begin{align}
    \label{eqn_2022_12_08_1517}
        \mathcal A(\lambda,\mu) &=   \mathfrak C_{\rm exc} \cdot \mathfrak c(\lambda\otimes\mu)\cdot \frac{\mathscr O({\Phi_{\lambda\otimes\mu}}) \cdot \mathscr O(\Phi_{\lambda\otimes\mu}') }{\mathscr O({{\Psi_{\hg_\lambda} }})^2}  \cdot \frac{\mathcal L_p^{\rm Katz}(\widetilde{\Phi}_{\lambda\otimes\mu}) \cdot \mathcal L_p^{\rm Katz}(\widetilde{\Phi'}_{\lambda\otimes\mu})}{\mathcal L_p^{\rm Katz}(\widetilde{\Psi}_{\hg_{\lambda}})^2}
    \end{align}
    whenever $\wt(\lambda)>\wt(\mu)$. Notice that $\mathfrak C_{\rm exc}, \, \mathscr O({\Phi_{\lambda\otimes\mu}}), \, \mathscr O(\Phi_{\lambda\otimes\mu}'), \, \mathscr O({{\Psi_{\hg_\lambda} }})$ are absolute nonzero rational constants (see \eqref{eqn_2023_01_05_1121} for the last three). 
    Since the CM family $\g_K$ is Eisenstein-irregular in weight one, by \cite[Remark 0.1]{BetinaDimitrovPozzi} we know that $\mathfrak c_{g_{\rm Eis}} := \mathfrak c_{\Phi'(\Theta)}(\mu\otimes \mu) \ne 0$. This ensures that we can specialize $\mathfrak c(\lambda\otimes \mu)$ at $\lambda=\mu$.
    
    As a result, we may compute $\mathcal A(\mu,\mu)$ as the limit of $\mathcal A(\lambda,\mu)$ as $\lambda\to\mu$ with $\wt(\lambda)>\wt(\mu)$, which in turn equals
\begin{align}
\label{eqn_2023_1_5_1130}
        \mathcal A(\mu,\mu) &=   \mathfrak C_{\rm exc} \cdot  \mathfrak c(\mu \otimes\mu) \cdot\frac{\mathscr O({\Phi_{\mu\otimes\mu}}) \cdot \mathscr O(\Phi_{\mu\otimes\mu}') }{\mathscr O({{\Psi_{\hg_\mu} }})^2}  \cdot \frac{\mathcal L_p^{\rm Katz}(\widetilde{\Phi}_{\mu\otimes\mu}) \cdot \mathcal L_p^{\rm Katz}(\widetilde{\Phi'}_{\mu\otimes\mu})}{\mathcal L_p^{\rm Katz}(\widetilde{\Psi}_{\hg_{\mu}})^2}
    \end{align}
    by continuity of the expressions involved, on substituting $\lambda:=\mu$ in \eqref{eqn_2022_12_08_1517}. Observe that both $\widetilde{\Phi}_{\mu\otimes\mu}$ and $\widetilde{\Psi}_{\g_\mu}$ fall within the region of interpolation of the Katz $p$-adic $L$-function, while $\widetilde{\Phi'}_{\mu\otimes\mu} = \mathbf{N}_K$ falls outside its interpolation region. Substituting the interpolation formula \eqref{eqn_2022_12_08_1519}, Equation \eqref{eqn_2023_01_04_1800}, and Equation \eqref{eqn:modifiedHidaperiods}  for $\widetilde{\Phi}_{\lambda\otimes\lambda}$ and $\widetilde{\Psi}_{\g_\lambda}$ in \eqref{eqn_2023_1_5_1130}, we deduce that
    \begin{align}
    \label{eqn_2023_01_06_1238}
    \begin{aligned}
        \mathcal A(\mu,\mu) 
        &=   \mathfrak C_{\rm exc} \cdot  \frac{\mathfrak c_{\Phi(\Theta)}(\mu\otimes\mu)^{-1}}{\mathfrak c_\hg(\mu)^{-2}}  \cdot \frac{\mathscr O({\Phi_{\mu\otimes\mu}})  }{\mathscr O({{\Psi_{\hg_\mu} }})^2}  \cdot \frac{\mathcal L_p^{\rm Katz}(\widetilde{\Phi}_{\mu\otimes\mu}) }{\mathcal L_p^{\rm Katz}(\widetilde{\Psi}_{\hg_{\mu}})^2}  \cdot \frac{\mathscr O(\Phi_{\mu\otimes\mu}')\cdot \mathcal L_p^{\rm Katz}(\mathbf{N}_K)}{\mathfrak c_{\Phi'(\Theta) }(\mu\otimes\mu)} 
        \\
        &=   \frac{\mathfrak C_{\rm exc}} {\mathfrak c_{g_{\rm Eis}} } \cdot \frac{ \Omega_{\Phi(\Theta)_{\mu\otimes\mu}} }{\Omega_{\g_\mu}^2 } \cdot \frac{-4}{\sqrt{D_K} } 
        \cdot \mathscr O(\Phi_{\mu\otimes\mu}')  \cdot \mathcal L_p^{\rm Katz}(\mathbf{N}_K)\,.
    \end{aligned}
    \end{align}
    In the last step we also used that $\Phi_{\mu\otimes\mu}'$ is the trivial character, so that $\Phi(\Theta)_{\mu\otimes\mu}^\circ = g_{\rm Eis}^\circ = {\rm Eis}_1(\varepsilon_K) \in M_1(D_K,\varepsilon_K)$. Hence, we have $M=M_\chi=D_K$ in the notation of \S\ref{subsubsec_2023_01_06_1256}  
 and hence,
    \begin{equation}
    \label{eqn_2023_01_06_1300}
        \mathscr O(\Phi_{\mu\otimes\mu}') = \frac{ D_K \cdot L(\varepsilon_K,1) }{\pi} = \frac{2h_K \sqrt{D_K}}{\omega_K}
    \end{equation}
    where the latter holds because of the class number formula. Combining \eqref{eqn_2023_01_06_1238} and \eqref{eqn_2023_01_06_1300}, we conclude that 
    \[
        \mathcal A(\mu,\mu) =  \mathfrak C_{\rm exc}  \cdot \frac{ \Omega_{\Phi(\Theta)_{\mu\otimes\mu}} }{\Omega_{\g_\mu}^2 } \cdot \frac{-8h_K }{\omega_K\cdot \mathfrak c_{g_{\rm Eis}} } \cdot \mathcal L_p^{\rm Katz}(\mathbf{N}_K)
    \]
    As the last step, the $p$-adic Kronecker limit formula for the Katz $p$-adic L-function (cf. \cite{Katz76}, \S 10.4) that tells us that $\mathcal L_p^{\rm Katz}( \mathbf{N}_K) = -\frac{1}{2}(1-p^{-1})\log_p(u)$, and this completes the proof.
    \end{proof}

\subsection{Factorization of the \texorpdfstring{${\rm BDP}^2$}{} \texorpdfstring{$p$}{}-adic \texorpdfstring{$L$}{}-function}
\label{subsec_8_2_2022_09_13_1731}
Our goal in \S\ref{subsec_8_2_2022_09_13_1731} is to explain the factorization of the ${\rm BDP}^2$ $p$-adic $L$-function $\cL_p^{\rm BDP}(\hf_{/K}\otimes \mathds{1})(\kappa,\lambda)=\cL_p^{\rm BDP}(\hf_{/K}\otimes \bbpsi)(\kappa, 1)$, in a manner to reflect the $p$-adic Artin formalism. The arguments heavily draw from \cite[\S 4.3]{BDV} and the verification Eqn. (30) in op. cit. For that reason, we will work with the notation of \cite{BDV} for the most part, and we refer to loc.cit. for more details\footnote{We also refer the reader to an earlier expanded version \cite[\S8.2]{BCSarxiv} of this paper for the calculations we omit here.}.

\subsubsection{}
\label{subsubsec_821_2022_12_12_1302}
Let us put $V_\f^{?}:= T_\f^{?} \otimes \Qp$ for $?=+,-,\{\}$. We recall from \S\ref{subsubsec_2022_05_16_1506} the definitions of the $G_{\Qp}$-representations $T_\f^{?}$ and recall that we work in this section under the hypothesis that the tame nebentype character $\varepsilon_\f$ of the Hida family $\f$ is trivial. We similarly define $V_{\g_K}^{?}$. 

Following \cite[\S8.2]{KLZ2}, we put ${\bf D}(V_\f^{-}):=\left(V_\f^{-}\hatotimes_{\Zp}\Zp^{\rm ur}\right)^{G_{\Qp}}$ and ${\bf D}(V_\f^{+}):=\left(V_\f^{+}(\bbchi_\f^{-1}\chi_\cyc)\hatotimes_{\Zp}\Zp^{\rm ur}\right)^{G_{\Qp}}$. 
We define the big Perrin-Riou logarithm maps 
\begin{equation}
\label{eqn_2022_12_12_1207}
{\rm Log}_\f\, \colon \, H^1_{\rm Iw}(\Qp(\mu_{p^\infty}),V_\f^{-})\stackrel{\sim}{\lra} {\bf D}(V_\f^{-})\hatotimes_{\Zp}\LL(\Gamma_\cyc)\,,\qquad \Gamma_\cyc:=\Gal(\Qp(\mu_{p^\infty})/\Qp)\,,
\end{equation}
\begin{equation}
\label{eqn_2022_12_15_1132}
{\rm Log}_{V_\f^+}\, \colon \, H^1_{\rm Iw}(\Qp(\mu_{p^\infty}),V_\f^{+})\stackrel{\sim}{\lra} {\bf D}(V_\f^{+})\hatotimes_{\Zp}\LL(\Gamma_\cyc)
\end{equation}
as in \cite[Theorem 8.2.3]{KLZ2}. Let us put $F^+V_\f^{\dagger}:=V_\f^{+}\otimes \bbchi_\f^{-\frac{1}{2}}\chi_\cyc$ and consider the map 
\begin{equation}
\label{eqn_2022_12_15_1140}
{\rm pr}_\f^\dagger\,: \,\cR_\f\hatotimes_{\Zp} \LL(\Gamma_\cyc)\lra  \cR_\f\,,\qquad \gamma\mapsto \bbchi_\f^{-\frac{1}{2}}\chi_\cyc(\gamma).
\end{equation}
We define the map
\begin{equation}
\label{eqn_2022_12_15_1134}
{\rm Log}_{V_\f^+}\otimes_{{\rm pr}_\f^\dagger} \cR_\f\, =:{\rm Log}_{F^+V_\f^{\dagger}}\,: \, {\rm im} \left(H^1_{\rm Iw}(\Qp(\mu_{p^\infty}),V_\f^{+})\xrightarrow{{\rm pr}_\f^\dagger} H^1(\Qp,F^+V_\f^{\dagger})\right)\lra  {\bf D}(V_\f^{+})\,.
\end{equation}

\subsubsection{} 
\label{subsubsec_2022_12_15_1115}
Recall from \cite[\S2.3.3]{BDV} (see also \cite{KLZ2}, Proposition 10.1.1) that overconvergent Eichler--Shimura theorem gives rise to a pair of canonical maps 
$$\omega_\f: {\bf D}(V_\f^{+}) \stackrel{\sim}{\lra} \cR_\f[{1}/{p}]\,\,,\,\,\quad \eta_\f: {\bf D}(V_\f^{-}) \hookrightarrow \frac{1}{H_\f}\cR_\f[{1}/{p}]\,,$$
where $H_\f$ is the congruence ideal associated with the cuspidal family $\f$. Finally, we put ${\rm Log}_{\omega_\f}:=\omega_\f\circ {\rm Log}_{F^+V_\f^{\dagger}} \circ\res_p$.

\subsubsection{}
\label{subsubsec_823_2022_12_21_1117}
One may similarly define $\omega_{\g_K}$ and $\eta_{\g_K}$ over a sufficiently small wide-open disc $U$ in ${\rm Spm}(\cR_{\g_K}[\frac{1}{p}])$ about the point corresponding to $g_{\rm Eis}$. In what follows, we will continue to write (rather abusively) $\cR_{\g_K}[\frac{1}{p}]$ in place of $\LL_U[\frac{1}{p}]$ (which is the ring that coincides with $\cO(\g_K)$ in \cite{BDV}) to denote the space of power-bounded functions on $U$. Since our ultimate results will require working with the weight-one specializations $\omega_{g_{\rm Eis}}$ and $\eta_{g_{\rm Eis}}$ of $\omega_{\g_K}$ and $\eta_{\g_K}$ (cf. \cite{BDV}, \S2.3.3.1), we will not keep further track of the choice of $U$.

\subsubsection{Beilinson--Kato elements and reciprocity laws} 
Given an element $\mathfrak{F}\in \LL(\Gamma_\cyc)[\frac{1}{p}]\simeq \LL(\Zp^\times)[\frac{1}{p}]$, we consider it as a function $\mathfrak{F}=\mathfrak{F}({\pmb \sigma})$ on $\Gamma_\cyc$ via the Amice transform, where ${\pmb \sigma}$ is the ``dummy cyclotomic variable''. 

We recall from Theorem~\ref{thm_31_2022_06_02_0854} the Mazur--Kitagawa $p$-adic $L$-function $\cL_p^{\rm Kit}(\hf\otimes\chi)$. Let us denote by $\res_p^-$ the composition of the following natural arrows:
\[
H^1(\QQ(\mu_{p^\infty}),V_\f)\xrightarrow{\res_p} H^1(\Qp(\mu_{p^\infty}),V_\f)\lra H^1(\QQ(\mu_{p^\infty}),V_\f^-)\,.
\]
The following statement ($p$-optimal interpolation of Beilinson--Kato elements in Hida families) was proved by Ochiai~\cite{Ochiai2006}:
    \begin{proposition}[Ochiai] \label{prop:KatoRec}
        Let $\chi$ denote either $\epsilon_K$ or the trivial character. There exist a unique element 
        $${\rm BK}_{\hf \otimes \chi} \in H^1_{\rm Iw}(\Q(\mu_{p^\infty}),T_{\hf} \otimes \chi)$$ 
        with the property that 
        \[
            \eta_\f \,\circ \,{\rm Log}_\hf(\res_p^-({\rm BK}_{\hf \otimes \chi})) = \cL_p^{\rm Kit}(\hf\otimes\chi)(1+\pmb{\sigma})\,.
        \]
    \end{proposition}

\subsubsection{Beilinson--Flach elements and reciprocity laws}
    \label{subsubsec_826_2022_12_21_1117}
Since we will be relying greatly on the results of \cite{BDV} for families, we will work until the end of \S\ref{subsec_75_2022_06_02_0900} over a sufficiently small wide-open disc $U_\f$ in ${\rm Spm}(\cR_{\f}[\frac{1}{p}])$ about the point corresponding to a fixed $p$-old  eigenform $f_\circ$ of weight $k\geq 2$, and with the coefficient ring $\LL_{U_\f}[\frac{1}{p}]$ power-bounded functions on $U_\f$ (which is the ring denoted by $\cO(\f)$ in \cite{BDV}) in place of  $\cR_{\f}[\frac{1}{p}]$. 

In what follows, we will (rather abusively) continue to write $\cR_{\f}[\frac{1}{p}]$ in place of $\LL_{U_\f}[\frac{1}{p}]$. We will also shrink $U_\f$ so as to ensure that $H_\f$ is a unit in $\LL_{U_\f}[\frac{1}{p}]$.

    \begin{proposition}
    \label{prop_86_2022_12_14_1124}
        For each integer $c$ coprime with $6Np$, there exists a Beilinson-Flach element
        \[
            _c\mathcal{BF}(\hf\otimes\hg_K) \in H^1_{\rm Iw,bal}(\Q(\mu_{p^\infty}), V_\hf\hatotimes_{\ZZ_p} V_{\hg_K})
        \]
        satisfying the explicit reciprocity law
        \[
            \mathscr L_{\pmb \xi}(_c\mathcal{BF}(\hf\otimes\hg_K)) = \mathscr N_{{\pmb \xi},c}\cdot \mathfrak c_{{\pmb \xi}}^{-1}\cdot \cL_p^{\pmb \xi}(\hf\otimes\hg_K, 1+\pmb{\sigma})
        \]
        for each ${\pmb \xi}\in \{\hf,\hg_K\}$, where $\mathfrak c_{\pmb \xi}$ is the congruence number of ${\pmb \xi}$ given as in  \S\ref{subsubsec_321_2022_05_26_1036}, and
        \[
            \mathscr N_{{\pmb \xi},c}(\kappa,\ell,\sigma) = - w_{\pmb \xi}\cdot \bigl( c^2-c^{2\sigma-\wt(\kappa)-\ell+4}\cdot \chi_\hf(c)^{-1}\chi_{\hg_K}(c)^{-1}\bigr)
        \]
        where $w_{\pmb \xi}\in \mathcal R_{\pmb \xi}^\times$ (Atkin--Lehner pseudo-eigenvalue) verifies $w_{\pmb \xi}(u)^2 = (-N_{\pmb \xi})^{2-u}$ for all $u\in \mathcal W_{\pmb \xi}^\cris$.
    \end{proposition}
    
    \begin{proof}
        This is \cite[Propositon 2.3]{BDV} (see also \cite[Theorem 10.2.2]{KLZ2}), with adjustments owing to our normalization of the Rankin-Selberg $p$-adic $L$-functions. The discrepancy between our statement and those in op. cit. is the factor $(-1)^{\pmb{\sigma}}\cdot \mathfrak c_{\pmb \xi}$, which is also reflected in the relevant interpolation formulae, cf. Theorem~\ref{thm:Hida-interp} and \cite[Theorem 2.7.4 \& Theorem 7.7.2]{KLZ2}.    
    \end{proof}

\subsubsection{First factorization}    We will make use of the following factorization result, which directly follows from interpolative properties of the $p$-adic $L$-functions involved in its statement, and it is a special case of \cite[Theorem 3.4]{Bertolini_Darmon_2014_IsrJ}, also disccused in \cite[\S3]{BSV2}.
    \begin{lemma}[Bertolini--Darmon] \label{lemma:RS-MKMK}
    Let us  put $\cL_p^\f(\f\otimes g_{\rm Eis})(\kappa, \pmb{\sigma}) := \cL_p^\f(\f\otimes \hg_K)(\kappa,1, \pmb{\sigma})$. 
        For sufficiently small $U_\f$ we have
        \[
            \cL_p^\f(\f\otimes g_{\rm Eis}) = \mathfrak c_\hf \cdot \mathcal C_1 \cdot\cL_p^{\rm Kit}(\hf)\cdot \cL_p^{\rm Kit}(\hf\otimes\varepsilon_K)\,,
        \]
        where  $\mathcal C_1\in \cR_{\f}[\frac{1}{p}]^\times$ is a bounded analytic function that verifies
        \begin{equation}
        \label{eqn_2022_12_21_1140}
        \mathcal C_1(\kappa)=\frac{ \mathfrak C_{\rm exc}(\hf\otimes\hg_K)}{(-2\sqrt{-1})^{\wt(\kappa)-1} \cdot C_{\hf_\kappa}^+C_{\hf_\kappa}^-\cdot \mathcal E(\hf_\kappa^\circ,\Ad)}\qquad \hbox{ for all arithmetic specializations } \kappa\in U_\f\,.
        \end{equation}
    \end{lemma}
    \begin{proof}
    Taking into account our normalizations (cf. Theorem \ref{thm_31_2022_06_02_0854} and Theorem~\ref{thm:Hida-interp}), the proof is essentially found in \cite[Theorem 4.2]{BDV} (see especially the portion of the proof after Equation (30) in op. cit.), \cite[ Theorem 3.4]{Bertolini_Darmon_2014_IsrJ}, or \cite[\S3.3-\S3.4]{BSV2} (note that our $\mathcal C_1$ is denoted by $\mathscr A_{\GL_2}$ in \cite{BSV2}).
    \end{proof}
\subsubsection{Irregular factorization of the ${\rm BDP}^2$ $p$-adic $L$-function}  
\label{subsubsec_2022_12_14_1014}
Our first goal in \S\ref{subsubsec_2022_12_14_1014} is to prove a comparison of Beilinson--Kato elements and Beilinson--Flach elements, which will play a key role in the proof of the factorization of Proposition~\ref{prop:factorizationBDPKitagawa}. Our arguments in this subsection are essentially copied from \cite[\S4.3]{BDV}.

Fix an isomorphism $V_{\hg_K} \simeq {\rm Ind}_{K/\Q}({\bbpsi})$ of $\mathcal R_{\hg_K}[\frac{1}{p}][[G_\Q]]$-modules (cf. \cite{BDV}, Eqn. 23). Since $p$ splits in $K$, the restriction of $V_{\hg_K}$ to both $G_K$ and $G_{\Qp}$ decomposes as ${\bbpsi}\oplus{\bbpsi}^c$. We identify  the first summand with the $G_{\Qp}$-representation $V_{\g_K}^+$ and the second summand with $V_{\g_K}^-$.

Let $\{v_{\hg_K}^+,v_{\hg_K}^-\}$ be the $\mathcal R_{\hg_K}[\frac{1}{p}]$-basis of $V_{\hg_K}$ given as in \cite[pp. 33]{BDV}. Complex conjugation maps $v_{\hg_K}^\pm$ to $v_{\hg_K}^\mp$, and denoting weight one specialization of $v_{\hg_K}^\pm$ by $v_{g_{\rm Eis}}^\pm$, we define (as in op. cit.)
    \[
        v_{g_{\rm Eis},\mathds{1}}:=v_{g_{\rm Eis}}^++v_{g_{\rm Eis}}^-\,,\qquad  v_{g_{\rm Eis},\varepsilon_K} = v_{g_{\rm Eis}}^+-v_{g_{\rm Eis}}^-\,,
    \]
    which form a basis for $V_{g_{\rm Eis}} \simeq \mathds{1}\oplus\varepsilon_K$. Note also that $\{v_{g_{\rm Eis}}^\pm\}$ is a basis of $V_{g_{\rm Eis}}^\pm$.
    
    Following \cite[\S4.3]{BDV}, we consider the modified Beilinson--Kato element
    \begin{equation}
            \label{eqn_2022_12_14_1632}
            \resizebox{0.92\hsize}{!}{  $\mathcal{BK}_{\hf\otimes g_{\rm Eis}} :=  
        \cL_p^{\rm Kit}(\hf\otimes\varepsilon_K,1+\pmb{\sigma}) \cdot {\rm BK}_\hf \otimes  v_{g_{\rm Eis},\mathds{1}}
        + 
        \cL_p^{\rm Kit}(\hf,1+\pmb{\sigma}) \cdot {\rm BK}_{\hf,\varepsilon_K} \otimes  v_{g_{\rm Eis},\varepsilon_K} 
        \in \, H^1_{\rm Iw}(K(\mu_{p^\infty}), V_\hf\otimes V_{g_{\rm Eis}})$\,,}
    \end{equation}
    where ${\rm BK}_{\hf,\varepsilon_K}\in  H^1_{\rm Iw}(\QQ(\mu_{p^\infty}), V_\hf\otimes \varepsilon_K)$ is constructed from  ${\rm BK}_{\hf\otimes\varepsilon_K}$ mimicking the normalizations in \S4.3 of \cite{BDV}. 
        Note that $\Gal(K/\Q)$ acts trivially on $\mathcal{BK}_{\hf\otimes g_{\rm Eis}}$. Therefore,
    $\mathcal{BK}_{\hf\otimes g_{\rm Eis}} \in H^1_{\rm Iw}(\Q(\mu_{p^\infty}), V_\hf\otimes V_{g_{\rm Eis}})$. We let ${}_c\mathcal{BF}_{\hf\otimes g_{\rm Eis}} \in H^1_{\rm Iw, bal}(\Q(\mu_{p^\infty}), V_\hf\otimes V_{g_{\rm Eis}})$
    denote the image of ${}_c\mathcal{BF}(\hf\otimes \g_{K})$ under the map induced from the  specialization morphism $V_{\g_K}\to V_{g_{\rm Eis}}$.

We are now ready to state the main result in \S\ref{subsubsec_2022_12_14_1014}:
    \begin{proposition}\label{prop:factorizationBDPKitagawa}
    We have the following factorization of the ${\rm BDP}^2$ $p$-adic $L$-function restricted to the central critical line in the weight space:
    \[
        \cL_p^{\hg_K}(\hf\otimes\hg_K)(\kappa,1,\wt(\kappa)/2) \, = \, \mathcal C(\kappa)\cdot \cL_p^{\rm Kit}(\f\otimes\epsilon_K)(\kappa, \rmw(\kappa)/2) \cdot {\rm Log}_{\omega_\f}({\rm BK}_{\f}^\dagger) \,.
    \]
    Here,
    \begin{equation}
    \label{eqn_2022_12_21_1139}
        \mathcal C 
        =
        \frac{\mathfrak c_{g_{\rm Eis}}\cdot \eta_{g_{\rm Eis}}( v_{g_{\rm Eis}}^-)}{ 2\, \omega_{g_{\rm Eis}}(v_{g_{\rm Eis}}^+)} \cdot \frac{w_\f}{w_{g_{\rm Eis}} } \cdot \mathcal C_1 
    \end{equation}
    with $\mathcal C_1 \in \cR_{\f}[\frac{1}{p}]^\times$ is as in Lemma~\ref{lemma:RS-MKMK}, and $w_\hf(\kappa)^2 = (-N_\hf)^{2-\wt(\kappa)}$ and $w_{g_{\rm Eis}}^2 = (-D_K)$.
    \end{proposition}

    \begin{proof}
    Following the proof of \cite[Theorem 4.2]{BDV}\footnote{The detailed calculation is available in the proof of \cite[Proposition 8.8]{BCSarxiv} for the interested reader. Since there are new ideas besides those readily present in \cite{BDV}, we decided to omit them here.}, one proves that
    \[
        2\,\omega_{g_{\rm Eis}}(v_{g_{\rm Eis}}^+)\cdot {{}_c}\mathcal{BF}_{\hf\otimes g_{\rm Eis}}   = \mathcal C_1 \cdot  \mathscr{N}_{\hf,c} \cdot \mathcal{BK}_{\hf\otimes g_{\rm Eis}}\,.
    \]
    where $\mathcal{C}_1$nis as in Lemma~\ref{lemma:RS-MKMK}, and $\mathscr N_{\hf,c}$ as in the statement of Proposition~\ref{prop_86_2022_12_14_1124}.  The proof then follows after a careful comparison\footnote{Details regarding this comparison can be found in \cite[\S8.2]{BCSarxiv}, especially as part of Proposition 8.9 in op. cit. The argument is an adaptation of some of the point-wise calculations of \cite[\S4.5 and \S4.6]{BDV}.} of the specialization of the logarithm maps $\mathscr L_{\g_K}$ and ${\rm Log}_{\omega_\f}$. 

    \end{proof}

    \subsection{Putting the pieces together}
    We are now ready to complete the proof of Theorem~\ref{thm_main_6_plus_2} in the scenario when $\g$ has CM, on passing to sufficiently small wide-open discs in ${\rm Spm}(R_{\pmb \f}[\frac{1}{p}])$ and ${\rm Spm}(R_{\pmb \g}[\frac{1}{p}])$ (cf. the discussion in \S\ref{subsubsec_823_2022_12_21_1117} and \S\ref{subsubsec_826_2022_12_21_1117}).

     \subsubsection{}
    \label{subsubsec_deg6_family_CM_case} 
     We begin this subsection with a rather ad hoc proof of a weak form of Conjecture~\ref{conj_2022_06_02_0940} on the existence of $\cL_p^{\Ad}(\hf\otimes \Ad^0\hg )$ in the situation when $\g$ has CM. 
     Recall the twist $\Phi=\Psi^\Ad \otimes \bbchi_\g\chi_\cyc^{-1}$ given as in \S\ref{subsubsec_812_2022_12_21_1028} of the the character $\Psi^\Ad$ given in \S\ref{subsubsec_813_2022_12_05}.
    
    \begin{proposition} \label{prop:factorizationAdBDP}
        \item[i)] There exists a unique element $ \mathcal B(\kappa,\lambda) \in {\rm Frac}(\cR_\f\,\hatotimes_{\ZZ_p}\, \cR_\g)$ with the property that 
        \begin{align}
        \begin{aligned}
        \label{eqn_2023_01_09_1751}
            \mathcal B(\kappa,\lambda) 
        &= \frac{2 (\sqrt{-1})^{\wt(\kappa)/2-1} }{\mathfrak C_{\rm exc}(\hf\otimes\Phi(\Theta)) } \cdot \frac{\Omega_{\Phi(\Theta)_{\lambda\otimes\lambda}}}{\Omega_{\hg_\lambda}^2}\,\\
        &
        = - \frac{2 (\sqrt{-1})^{\wt(\kappa)/2-1} }{\mathfrak C_{\rm exc}(\hf\otimes\Phi(\Theta)) } \cdot \mathfrak c_{g_{\rm Eis}}\cdot \mathfrak c(\lambda \otimes\lambda) \cdot \sqrt{D_K}\cdot \frac{\mathscr O({\Phi_{\lambda\otimes\lambda}})}{\mathscr O({{\Psi_{\hg_\lambda} }})^2}  \cdot \frac{\mathcal L_p^{\rm Katz}(\widetilde{\Phi}_{\lambda\otimes\lambda})}{\mathcal L_p^{\rm Katz}(\widetilde{\Psi}_{\hg_{\lambda}})^2}  \,\in \,\overline{\QQ}
        \end{aligned}
        \end{align}
        for all crystalline $\kappa,\lambda$.       
        \item[ii)] The $p$-adic $L$-function $\cL_p^\Ad(\hf\otimes \Ad^0\hg) \in {\rm Frac}(\cR_\f\,\hatotimes_{\ZZ_p}\, \cR_\g)$ defined on setting
        \[
            \cL_p^\Ad(\hf\otimes \Ad^0\hg)(\kappa,\lambda) := \mathcal B(\kappa,\lambda) \cdot \cL_p^{\rm BDP}(\hf_{/K}\otimes\Psi^{\Ad})(\kappa,\lambda) \cdot \cL_p^{\rm Kit}(\hf\otimes \varepsilon_K)(\kappa,{\rm w}(\kappa)/2)
        \]          
        verifies the interpolation formula of Conjecture~\ref{conj_2022_06_02_0940}.
        \item[iii)] We have
        $$\mathcal L_p^{\rm Katz}(\widetilde{\Psi}_{\hg})^2\cdot  \cL_p^\Ad(\hf\otimes \Ad^0\hg)\in (\cR_\f\, \hatotimes_{\ZZ_p}\, \cR_\g)[1/p]\,.$$
        In particular, the denominator of $ \cL_p^\Ad(\hf\otimes \Ad^0\hg)$ given as in (ii) is a function of $\lambda$ alone.
    \end{proposition}

   We will strengthen this result in Corollary~\ref{eqn_cor_2023_01_09_1840} to prove that $\cL_p^\Ad(\hf\otimes \Ad^0\hg)\in (\cR_\f\, \hatotimes_{\ZZ_p}\, \cR_\g)[1/p]$.  
    \begin{proof}
\item[i)] The first equality in the statement of this portion follows from \eqref{eqn:modifiedHidaperiods}, \eqref{eqn_2023_01_04_1800}, and \eqref{eqn_2022_12_08_1519}, utilized as in the verification of \eqref{eqn_2023_01_06_1238}. The remaining assertions in this portion are clear as per definitions.

\item[ii)] Recall the decomposition
    \[
        T_{\hf_\kappa}^\dagger\otimes \Ad^0(T_{\hg_\lambda}) \simeq 
        \bigl( T_{\hf_\kappa}^\dagger \otimes \Ind_{K/\Q}(\Psi^\Ad_{\lambda\otimes\lambda}) \bigr) 
        \oplus 
        \bigl(T_{\hf_\kappa}^\dagger\otimes \varepsilon_K \bigr)
        \simeq 
        \bigl(T_{\hf_\kappa} \otimes \Ind_{K/\Q}(\Phi_{\lambda\otimes\lambda}) \bigr)(1-c) 
        \oplus 
        \bigl( T_{\hf_\kappa}^\dagger\otimes \varepsilon_K \bigr)
    \]
    of Galois representations (where $c=\wt(\kappa)/2+\wt(\lambda)-1$), and correspondingly, the factorization of algebraic parts of the $L$-values:
    \begin{align}
    \begin{aligned}
    \label{eqn_2022_12_21_1051}
        \frac{\Lambda(\hf_\kappa \otimes\Ad^0(\hg_\lambda), \wt(\kappa)/2)}{ \Omega_{\hg_\lambda}^2 \Omega_{\hf_\kappa}^-} 
        &=
        \frac{\Omega_{\Phi(\Theta)_{\lambda\otimes\lambda}}}{\Omega_{\hg_\lambda}^2}
        \cdot
        \frac{\Lambda(\hf_\kappa\otimes\Phi(\Theta)_{\lambda\otimes\lambda}^\circ, c)}{\Omega_{\Phi(\Theta)_{\lambda\otimes\lambda}}} \cdot \frac{\Lambda(\hf_\kappa\otimes\varepsilon_K, \wt(\kappa)/2)}{\Omega_{\hf_\kappa}^-}\,.
    \end{aligned}
    \end{align}
    By direct calculation, we also have the factorization of the Euler-like factors:
    \begin{equation}
        \label{eqn_2022_12_21_1052}
         \mathcal E_p^\Ad(\hf_\kappa\otimes \Ad^0(\hg_\lambda), \wt(\kappa)/2) = \mathcal E_p^{\Phi(\Theta)}(\hf_\kappa\otimes\Phi(\Theta)_{\lambda\otimes\lambda},c) \cdot \mathcal E_p(\hf_\kappa\otimes \varepsilon_K, \wt(\kappa)/2).
    \end{equation}
    By the interpolation formulae of the relevant $p$-adic $L$-functions (cf. \eqref{eqn_2022_12_07_1302} and Theorem~\ref{thm_31_2022_06_02_0854}) combined with \eqref{eqn_2022_12_21_1051} and \eqref{eqn_2022_12_21_1052},  for all $(\kappa,\lambda) \in \cW_2^\Ad$ we have 
    \begin{align} \label{eqn_2022_12_21_1234}
    \begin{aligned}
    \cL_p^\Ad(\hf\otimes \Ad^0\hg)(\kappa,\lambda) 
    &:= \mathcal B(\kappa,\lambda) \cdot \cL_p^{\rm BDP}(\hf_{/K}\otimes\Psi^{\Ad})(\kappa, \lambda) \cdot \cL_p^{\rm Kit}(\hf\otimes \varepsilon_K)(\kappa,\wt(\kappa)/2)
   \\
        &=
        \frac{\Omega_{\Phi(\Theta)_{\lambda\otimes\lambda}}}{\Omega_{\hg_\lambda}^2}
    \cdot
    \mathcal E_p^{\Phi(\Theta)}(\hf_\kappa\otimes\Phi(\Theta)_{\lambda\otimes\lambda},c) \cdot 
    \frac{\Lambda(\hf_\kappa^\circ \otimes\Phi(\Theta)_{\lambda\otimes\lambda}^\circ, c)}{\Omega_{\Phi(\Theta)_{\lambda\otimes\lambda}}} 
    \\ 
    &\hspace{90pt}\times C_{\hf_\kappa}^- \cdot  \mathcal E_p(\hf_\kappa\otimes \varepsilon_K, \wt(\kappa)/2)\cdot \frac{\Lambda(\hf_\kappa^\circ\otimes\varepsilon_K, \wt(\kappa)/2)}{\Omega_{\hf_\kappa}^-}
   \\
    &=
    C_{\hf_\kappa}^- \cdot \mathcal E_p^\Ad(\hf_\kappa\otimes \Ad^0(\hg_\lambda), \wt(\kappa)/2) \cdot \frac{\Lambda(\hf_\kappa \otimes\Ad^0(\hg_\lambda), \wt(\kappa)/2)}{ \Omega_{\hg_\lambda}^2 \Omega_{\hf_\kappa}^-}\,,  
    \end{aligned}
    \end{align}
    where $\mathcal B(\kappa,\lambda)$ is as in Part (i), as required.

   \item[iii)] Since we have 
   $$\cL_p^{\rm BDP}(\hf_{/K}\otimes\Psi^{\Ad})(\kappa,\lambda) \cdot \cL_p^{\rm Kit}(\hf\otimes \varepsilon_K)(\kappa,{\rm w}(\kappa)/2) \in \cR_\f \,\hatotimes_{\ZZ_p}\,\cR_\g,$$ 
   the proof of this portion follows immediately from the definition of $\cL_p^{\rm BDP}(\hf_{/K}\otimes\Psi^{\Ad})(\kappa,\lambda)$ and the description of the denominator of the factor $\mathcal B (\kappa,\lambda)$ in terms of the indicated Katz $p$-adic $L$-function in Part (i).
    \end{proof}


    \begin{theorem} \label{Thm:8=6+2CM}
        When $\hg$ is a CM form, Theorem \ref{thm_main_6_plus_2} is true. Namely, we have the following  factorization of $p$-adic $L$-functions:
        \begin{equation}
            \label{eqn_2023_01_09_1846}
            \cL_p^\hg(\hf\otimes \hg\otimes \hg^c)^2(\kappa,\lambda,\lambda) = \mathscr C(\kappa)\cdot \cL_p^\Ad(\hf\otimes \Ad^0\hg)(\kappa,\lambda) \cdot {\rm Log}_{\omega_\f}({\rm BK}_{\f}^\dagger)\,.
        \end{equation}
        Here, 
        $\mathscr C(\kappa)$ is a meromorphic function
        and it has the following algebraicity property: For all crystalline $\kappa,\lambda$, 
        $$\mathscr C(\kappa) \cdot \cL_p^{\rm Kit}(\f\otimes\epsilon_K)(\kappa,\wt(\kappa)/2)\cdot {\rm Log}_{\omega_\f}({\rm BK}_{\f}^\dagger)(\kappa)$$
        is a non-zero algebraic multiple of 
        $$\cL_p^{\rm MV}(\f_\kappa\otimes\epsilon_K)(\wt(\kappa)/2)\cdot\log_{\omega_{\f_\kappa}}({\rm BK}_{\f_\kappa}^\dagger)\,,$$ 
        where ${\rm BK}_{\f_\kappa}^\dagger$ is the central critical twist of the optimal Beilinson--Kato element ${\rm BK}_{\f_\kappa}$ attached to $\f_\kappa$, and $\cL_p^{\rm MV}(\f_\kappa\otimes\epsilon_K)$ is the Manin--Vi\v{s}ik $p$-adic $L$-function attached to $\f_\kappa$ and the canonical periods $\Omega_{\f_\kappa}^\pm$.
    \end{theorem}
    \begin{proof}
    Thanks to Corollary~\ref{cor:analyticfactorization}, we have the following factorization:
    \[
        \cL_p^\hg(\hf\otimes \hg\otimes \hg^c)^2(\kappa,\lambda,\lambda)= \mathcal A(\lambda,\lambda) \cdot 
        \cL_p^{\rm BDP}(\hf_{/K} \otimes \Psi^\Ad )(\kappa, \lambda ) \cdot \cL_p^{\rm BDP}(\hf_{/K} \otimes \mathds{1})(\kappa,\lambda )\,,
    \]
    where $\mathcal A(\lambda,\lambda)$ is explicitly described in Proposition \ref{prop:Alambdalambda}. Using Proposition~\ref{prop:factorizationAdBDP}, we further deduce that
    \begin{align*}
        \cL_p^\hg(\hf\otimes \hg\otimes \hg^c)^2(\kappa,\lambda,\lambda)
        &= \frac{\mathcal A(\lambda,\lambda)}{\mathcal B(\kappa,\lambda)} \cdot \mathcal L_p^\Ad(\hf\otimes\Ad^0\hg)(\kappa,\lambda) \cdot \frac{\mathcal L_p^{\rm BDP}(\hf_{/K},\mathds{1})(\kappa,\lambda)}{\mathcal{L}_p^{\rm Kit}(\hf\otimes\varepsilon_K)(\kappa,\wt(\kappa)/2)}\,,
    \end{align*}
   where
    \begin{align}
    \begin{aligned}
        \label{eqn_2022_12_21_1132}
          \dfrac{ \mathcal A(\lambda,\lambda)}{\mathcal B(\kappa,\lambda)} 
        &=  \dfrac{ \dfrac{ \mathfrak C_{\rm exc}(\hf\otimes\hg\otimes\hg^c) }{\mathfrak C_{\rm exc}(\hf\otimes\Phi(\Theta)) \cdot \mathfrak C_{\rm exc}(\hf\otimes\Phi'(\Theta))} 
        \cdot
        \dfrac{\Omega_{ \Phi(\Theta)_{\lambda\otimes\lambda} } }{\Omega_{\Psi_\hg(\Theta)_{\lambda}}^2} 
        \cdot \dfrac{4h_K(1-p^{-1})}{\omega_K \cdot \mathfrak{c}_{g_{\rm Eis}}} \cdot \log_p(u)
        }{ \dfrac{2 (\sqrt{-1})^{\wt(\kappa)/2-1} }{\mathfrak C_{\rm exc}(\hf\otimes\Phi(\Theta)) } \cdot \dfrac{\Omega_{\Phi(\Theta)_{\lambda\otimes\lambda}}}{\Omega_{\hg_\lambda}^2}} \\
        &= \dfrac{ \mathfrak C_{\rm exc}(\hf\otimes\hg\otimes\hg^c) }{ \mathfrak C_{\rm exc}(\hf\otimes\Phi'(\Theta)) } 
        \cdot \frac{ 2h_K(1-p^{-1})}{\omega_K (\sqrt{-1})^{\wt(\kappa)/2-1} \mathfrak c_{g_{\rm Eis}} } \cdot \log_p(u)\,.
    \end{aligned}
    \end{align}
    It follows from Proposition~\ref{prop:factorizationBDPKitagawa} that we have
    \[
        \frac{\mathcal{L}_p^{\rm BDP}(\f_{/K} \otimes \mathds{1}) (\kappa,\lambda)}{\mathcal L_p^{\rm Kit}(\hf\otimes\varepsilon_K)(\kappa,\wt(\kappa)/2)} 
        = \frac{\cL_p^{\hg_K}(\hf\otimes\hg_K)(\kappa,1,\wt(\kappa)/2)}{\mathcal L_p^{\rm Kit}(\hf\otimes\varepsilon_K)(\kappa,\wt(\kappa)/2)} 
        =
        \mathcal C(\kappa) \cdot {\rm Log}_{\omega_\hf}({\rm BK}_\hf^\dagger),
    \]
    so that
    \[
     \mathcal L_p^\hg(\hf\otimes\hg\otimes\hg^c)(\kappa,\lambda,\lambda) = \frac{\mathcal A(\lambda,\lambda) \cdot\mathcal C(\kappa)}{\mathcal B(\kappa,\lambda)} \cdot \mathcal L_p^\Ad(\hf\otimes\Ad^0\hg)(\kappa,\lambda) \cdot {\rm Log}_{\omega_\hf}({\rm BK}_{\hf}^\dagger).
    \]
    It remains to analyze the specializations of 
    $$\mathscr{C}(\kappa,\lambda):=\frac{\mathcal A (\lambda,\lambda)\cdot \mathcal C (\kappa)}{\mathcal B (\kappa,\lambda)}\,.$$ 
    Using \eqref{eqn_2022_12_21_1140}, \eqref{eqn_2022_12_21_1139},  and \eqref{eqn_2022_12_21_1132}, we infer that $\mathscr{C}(\kappa,\lambda)$ equals
    \begin{align}
    \label{eqn_2023_01_06_1436}
    \begin{aligned}
        &\frac{ \mathfrak C_{\rm exc}(\hf\otimes\hg\otimes\hg^c) }{\mathcal C^{\Ad}(\kappa,\lambda) \cdot\mathfrak C_{\rm exc}(\hf\otimes\Phi'(\Theta)) } 
        \cdot \frac{2h_K(1-p^{-1})}{\omega_K (\sqrt{-1})^{\wt(\kappa)/2-1} \mathfrak c_{g_{\rm Eis}} } \cdot \log_p(u)  \cdot 
        \frac{\mathfrak c_{g_{\rm Eis}}\cdot\eta_{g_{\rm Eis}}(v_{g_{\rm Eis}}^-)}{2\omega_{g_{\rm Eis}}(v_{g_{\rm Eis}}^+)} \cdot \frac{w_\hf(\kappa)}{w_{g_{\rm Eis}} } 
        \\
        &\hspace{9cm}\times \frac{ - \mathfrak C_{\rm exc}(\hf\otimes\hg_K) }{(-2\sqrt{-1})^{\wt(\kappa)-1}\cdot C_{\hf_\kappa}^+C_{\hf_\kappa}^-\cdot \mathcal E(\hf_\kappa^\circ,\Ad)}  \\
        &= 
        \mathfrak C_{\rm exc}(\hf\otimes\hg\otimes\hg^c)
        \cdot \frac{h_K(1-p^{-1}) (\sqrt{-1})^{\wt(\kappa)/2}}{2^{\wt(\kappa)-1}\omega_K   } \cdot  \frac{w_\hf(\kappa)}{w_{g_{\rm Eis}} }    \cdot \frac{ 1 }{ C_{\hf_\kappa}^+C_{\hf_\kappa}^-\cdot \mathcal E(\hf_\kappa^\circ,\Ad)} 
        \cdot 
        \frac{\eta_{g_{\rm Eis}}(v_{g_{\rm Eis}}^-)}{\omega_{g_{\rm Eis}}(v_{g_{\rm Eis}}^+)} \cdot \log_p(u)  
        \\
        &= \mathfrak C_{\rm exc}(\hf\otimes\hg\otimes\hg^c) 
        \times \frac{h_K(1-p^{-1}) }{\omega_K  w_{g_\mathrm{Eis}}   } 
        \times 
       \frac{\eta_{g_{\rm Eis}}(v_{g_{\rm Eis}}^-)}{\omega_{g_{\rm Eis}}(v_{g_{\rm Eis}}^+)} \cdot \log_p(u)
       \times \frac{(\sqrt{-1})^{\wt(\kappa)/2} {w_\hf(\kappa)} }{2^{\wt(\kappa)-1}\cdot \mathcal E(\hf_\kappa^\circ,\Ad)}    
         \times \frac{ 1 }{ C_{\hf_\kappa}^+C_{\hf_\kappa}^-}  \,,
    \end{aligned}
    \end{align}
    where we recall that $w_\hf(\kappa)$ and $w_{g_\mathrm{Eis}}$ are the Atkin--Lehner pseudo-eigenvalues, and they verify $w_\hf(\kappa)^2=(-N_\f)^{2-\wt(\kappa)}$, $w_{g_\mathrm{Eis}}^2=-D_K$. It  follows from \cite[Lemma 4.6]{BDV} that  $\dfrac{\eta_{g_{\rm Eis}}(v_{g_{\rm Eis}}^-)}{\omega_{g_{\rm Eis}}(v_{g_{\rm Eis}}^+)} \cdot \log_p(u) $ is a non-zero rational number. Therefore, in the very last line of \eqref{eqn_2023_01_06_1436} all factors but the last take non-zero algebraic values for crystalline $\kappa,\lambda$, which are clearly independent on $\lambda$, so that $\mathscr C(\kappa,\lambda)=\mathscr C(\kappa)$. The claimed algebraicity property of $\mathscr{C}(\kappa)$ follows from the choice of the $p$-optimal Beilinson--Kato elements ${\rm BK}_{\f}^\dagger$ (cf. Proposition~\ref{prop:KatoRec}) and the choice of the $p$-adic periods $C_{\f_\kappa}^\pm$ (cf. Theorem~\ref{thm_31_2022_06_02_0854}), which tells us that
    $$\cL_p^{\rm Kit}(\f)(\kappa,\wt(\kappa)/2)=C_{\f_\kappa}^\pm\cL_p^{\rm MV}(\f_\kappa, \wt(\kappa)/2)\,\,,\,\,\,\, \cL_p^{\rm Kit}(\f\otimes\epsilon_K)(\kappa,\wt(\kappa)/2)=C_{\f_\kappa}^\mp\cL_p^{\rm MV}(\f_\kappa\otimes\epsilon_K, \wt(\kappa)/2)\,.$$
    \end{proof}

\begin{corollary}
    \label{eqn_cor_2023_01_09_1840}
    We have $ \cL_p^\Ad(\hf\otimes \Ad^0\hg) \in (\cR_\f \,\hatotimes_{\ZZ_p}\,\cR_\g)[1/p]$ for the $p$-adic $L$-function given as in Proposition~\ref{prop:factorizationAdBDP}(ii). 
\end{corollary}

\begin{proof}
It follows from \eqref{eqn_2023_01_06_1436} that the numerator of $\mathscr C (\kappa,\lambda)$ is a function of $\kappa$ alone. The same also applies to the factor ${\rm Log}_{\omega_\f}({\rm BK}_{\f}^\dagger)$ in \eqref{eqn_2023_01_09_1846}. As a result, since $ \cL_p^\hg(\hf\otimes \hg\otimes \hg^c)^2 (\kappa,\lambda,\lambda)\in \cR_\f\,\hatotimes_{\ZZ_p}\,\cR_\g$, we infer from \eqref{eqn_2023_01_09_1846} that the denominator of $ \cL_p^\Ad(\hf\otimes \Ad^0\hg)$ is a function of $\kappa$ alone. The proof follows on combining this observation with Proposition~\ref{eqn_2023_01_09_1751}(iii).
\end{proof}

\section{Super-factorization of algebraic \texorpdfstring{$p$}{}-adic \texorpdfstring{$L$}{}-functions}
\label{sec_super_factorization}
In this section, we prove the algebraic variant of Theorem~\ref{Thm:8=6+2CM}, which amounts to the factorization of the corresponding algebraic $p$-adic $L$-functions. Our discussion here supplements and refines,  in the scenario when the family $\g$ has CM, that in the companion paper~\cite{BS_Part1}. 

In what follows, we will crucially rely on the notation, conventions, and constructions in op. cit.



\subsection{Greenberg local conditions}
\label{subsec_examples_local_conditons_CM} In this paragraph, we introduce the  Greenberg local conditions, which we will use to define the Selmer complexes that are to our goals.

Our assumption that $\g$ has CM (alongside our running hypotheses \ref{item_Irr} and \ref{item_Dist}) allows us to identify $\cR_\g$ with a finite flat extension of $\ZZ_p[[\Gamma_\p]]$ and $T_\g$ with ${\rm Ind}_{K/\Q}\Psi$ for some $\cR_\g$-valued character $\Psi$ of $G_K$. Let us put $\Psi^c:=\Psi\circ c$ (so that $\Psi^c(g)=\Psi(cgc^{-1})$ for all $g\in G_K$, where we have denoted by $c$ any lift of $c$ to $G_{\Q}$) and $\Psi_{\rm ad}:=\Psi/\Psi^c$. We remark that \ref{item_Dist} translates in the present setting to the following condition:

\begin{itemize}
 \item[\mylabel{item_Dist_CM}{\bf (Dist$'$)}]  $(\Psi(\p) \mod \mathfrak{m}_\g) \neq (\Psi(\p^c) \mod \mathfrak{m}_\g) ,\qquad\qquad \mathfrak{m}_\g\subset \cR_\g \hbox{ is the maximal ideal}\,.$
 \end{itemize}

When we wish to distinguish the weight variables of $\g$ and $\g^c$, we shall write $\Psi(\bl)$ and $\Psi(\bm)$ for the corresponding characters of $G_K$. Note that we have 
$$T_3^\dagger=T_\f^\dagger\,\widehat\otimes \, {\rm Ind}_{K/\Q}\Psi(\bl )\,\widehat\otimes \, {\rm Ind}_{K/\Q}\Psi(\bm )^{-1}=T_\hf^\dagger\,\widehat\otimes \, {{\rm Ind}_{K/\Q}}\,\underbrace{\left(\Psi({\bf l})\widehat\otimes\Psi^{-1}({\bf m})\oplus \Psi({\bf l})\widehat\otimes\Psi^{-1}({\bf m})^c \right)}_{U_2^\dagger}\,.$$

\subsubsection{} 
We remark that 
\begin{equation}
\label{eqn_U3_vs_U2}
\iota_{2,3}^*\left(T_\hf^\dagger\,\widehat\otimes\,U_2^\dagger\right)=T_\hf^\dagger\,\widehat\otimes \underbrace{\left(\mathds{1}\oplus \Psi_{\rm ad}\right)}_{U_1^\dagger}\qquad,\qquad T_2^\dagger=T_\hf^\dagger\,\widehat\otimes\,{\rm Ind}_{K/\QQ}\,{U_1^\dagger}
\end{equation}
We therefore have a split exact sequence
\begin{equation}
\label{eqn_CM_split_exact_sequence}
    0\lra T_\f\,\widehat\otimes\, {\rm Ind}_{K/\QQ}\mathds{1}\lra T_2^\dagger \lra T_\f\,\widehat\otimes\,{\rm Ind}_{K/\QQ}\Psi_{\rm ad}\lra 0\,.
\end{equation}

\subsubsection{}\label{subsubsubsec_631_08042021}  
Notice that we have a natural injection
$$
{\rm Ind}_{K/\Q}\,(T_{\hf}^\dagger \, \widehat{\otimes} \, \mathds{1}) \, \cong \, T_{\hf}^\dagger\,\widehat\otimes \,{\rm Ind}_{K/\Q}\left(\Psi({\bf l})\otimes_{\cR_\g} \Psi({\bf l})^{-1}\right) 
\lra T_{\hf}^\dagger\widehat\otimes {\rm Ind}_{K/\Q}\Psi({\bf l})\otimes_{\cR_\g} {\rm Ind}_{K/\Q}\Psi({\bf l})^{-1}=T_2^\dagger\,.
$$
Here the isomorphism 
\[
{\rm Ind}_{K/\Q}\,(T_{\hf}^\dagger \, \widehat{\otimes} \, \mathds{1}) \, \cong \, T_{\hf}^\dagger\,\widehat\otimes \,{\rm Ind}_{K/\Q}\left(\Psi({\bf l})\otimes_{\cR_\g} \Psi({\bf l})^{-1}\right)
\]
is induced by the isomorphism $T_{\hf}^\dagger \cong (T_{\hf}^\dagger)^c$ given by $x \mapsto c \cdot x$. This, combined with the natural decomposition\footnote{In fact, when we need to exercise further caution about the tempting identification 
$$
\Zp[\Gal(K/\Q)]\otimes T_{\hf}^\dagger\supset 1\otimes T_{\hf}^\dagger=T_{\hf}^\dagger
$$ 
(which is valid only as $G_K$-modules, but not as $G_\QQ$-modules), we will write this identity as 
\begin{equation}
    \label{eqn_factor_trivial_rep_by_idempotents}
    {\rm Ind}_{K/\Q}\,T_{\hf}^\dagger=\left(\frac{1+c}{2}\otimes T_{\hf}^\dagger\right) \oplus \left(\frac{1-c}{2} \otimes T_{\hf}^\dagger\right)
\end{equation}
which takes place in $\Zp[\Gal(K/\Q)]\otimes T_{\hf}^\dagger$.}
$$
{\rm Ind}_{K/\Q}\,(T_{\hf}^\dagger \, \widehat{\otimes} \, \mathds{1})  =  
({\rm Ind}_{K/\Q}\, T_{\hf}^\dagger) \, \widehat{\otimes} \, \mathds{1} =T_{\hf}^\dagger \, \widehat{\otimes} \,  \mathds{1} \oplus (T_{\hf}^\dagger\otimes\epsilon_K) \, \widehat{\otimes} \, \mathds{1}
$$
yields an exact sequence of $G_\QQ$-modules
\begin{equation}
    \label{eqn_splitting_dual_global}
    0\lra \left(T_{\hf}^\dagger= \frac{1+c}{2}\otimes T_{\hf}^\dagger \right) 
    \, \widehat{\otimes} \, \mathds{1} 
    \lra T_2^\dagger \lra \left(\frac{1-c}{2}\otimes T_\hf^\dagger\right) 
    \, \widehat{\otimes} \, \mathds{1} 
    \oplus (T_\hf^\dagger\, \widehat\otimes \, {\rm Ind}_{K/\Q}\Psi_{\rm ad})=M_2^\dagger\lra 0\,.
\end{equation}
This is one of the reasons why we prefer to work with the splitting \eqref{eqn_factorisationrepresentations_dual_trace} given by the transpose of the trace (rather than the trace itself).

\subsubsection{} 
\label{subsubsec_322_14_05_2021}
We let $D_\p\subset G_K$ denote the choice of a decomposition group at $\p$, which identifies $D_\p$ as the absolute Galois group $G_{K_\p}$ of the corresponding completion $K_\p:=\iota_p(K)\Q_p=\Q_p$ of $K$. Then $D_{\p^c}:=cD_{\p}c^{-1}$ is a decomposition group at $\p^c$, which we identify with $G_{K_{\p^c}}$. The embedding $\iota_p$ also determines a decomposition group $D_p\subset G_{\Q}$, which is identified by $D_\p$ thanks to our choices.  Henceforth, $H^*(\Q_p,-)=H^*(D_p,-)$ and $H^*(K_?,-)=H^*(D_?,-)$ where $?=\p,\p^c$.

For $\eta\in\{\Psi,\Psi^{-1}\}$ of $G_K$, let us write $\{v_\eta\}$ (resp. $cv_{\eta^c}$) denote a basis of the free module on which $G_K$ acts via $\eta$ (resp. via $\eta^c$). 
We will also put $v_{\Psi_{\rm ad}}:=v_\Psi\otimes cv_{\Psi^{c,-1}}$ (resp. $v_{\Psi_{\rm ad}^c}:=cv_{\Psi^c}\otimes v_{\Psi^{-1}}$), so that $\{v_{\Psi_{\rm ad}}\}$ (resp. $\{v_{\Psi_{\rm ad}^c}\}$) is the basis of the free $\LL_\g$-module of rank one on which $G_K$ acts via $\Psi_{\rm ad}$ (resp. $\Psi_{\rm ad}^c$). 

For $\eta\in\{\Psi,\Psi^{-1},\Psi_{\rm ad}\}$, we note that 
$${\rm Ind}_{K/\Q}\eta:=\ZZ_p[\Delta]\otimes \eta=1\otimes\eta \oplus c\otimes\eta={\rm span}\{v_{\eta},c v_{\eta^c}\}$$
where for the final equality, we use the natural identification of the left (diagonal) $G_K$-action on $c\otimes {\eta}$ with the $G_K$-action via ${\eta^c}$. 

With respect to the fixed basis $\{v_{\eta},c v_{\eta^c}\}$ of ${\rm Ind}_{K/\Q}\eta$, an element $g\in G_K$ acts via the diagonal matrix $\begin{bmatrix}
\eta(g) & 0\\
0 & \eta^c(g) 
\end{bmatrix}$ 
whereas $cg\in G_\QQ\setminus G_K$ acts via
$\begin{bmatrix}
0 & \eta^c(g)\\
\eta(g) &  0
\end{bmatrix}$. In particular, $c\cdot v_\eta= c v_{\eta^c}$ and $c\cdot cv_{\eta^c}=v_\eta$. With these choices at hand, observe that 
 
\begin{equation}
    \label{eqn_InfPsiptotimesIndpsiinverse}
   \resizebox{0.95\hsize}{!}{
   ${\rm Ind}_{K/\Q}\Psi \otimes {\rm Ind}_{K/\Q}\Psi^{-1}
    =\underbrace{{\rm span}\left\{\frac{1+c}{2}\cdot v_{\Psi}\otimes v_{\Psi^{-1}}\right\}}_{\mathds{1}} \oplus \underbrace{{\rm span}\left\{\frac{1-c}{2}\cdot v_{\Psi}\otimes v_{\Psi^{-1}}\right\}}_{\epsilon_K} \oplus \underbrace{{\rm span}\left\{v_{\Psi_{\rm ad}}, v_{\Psi_{\rm ad}^c} \right\}}_{{\rm Ind}_{K/\Q}\Psi_{\rm ad}}.$}
\end{equation}

\begin{lemma}[Semi-local Shapiro's lemma]
\label{lemma_semi_local_shapiro}
We have the following quasi-isomorphisms:
 \begin{equation}
\label{eqn_lemma_shapiro_explicit_1}
\resizebox{0.92\hsize}{!}{$
    R\Gamma(D_p,T_\f^\dagger\,\widehat\otimes\,{\rm Ind}_{K/\Q}\Psi)\xrightarrow[\iota_p]{\sim} R\Gamma(D_\p,T_\f^\dagger\,\widehat\otimes\,\Psi)\oplus R\Gamma(D_{\p},T_\f^\dagger\,\widehat\otimes\,\Psi^c)\xrightarrow[{\rm id}\oplus c_*]{\sim}R\Gamma(D_\p,T_\f^\dagger\,\widehat\otimes\,\Psi)\oplus R\Gamma(D_{\p^c},T_\f^\dagger\,\widehat\otimes\,\Psi)\,,$
    }
\end{equation}
\vspace{-0.5cm}
\begin{equation}
\label{eqn_lemma_shapiro_explicit_2}
    R\Gamma(D_p,T_\f^\dagger\,\widehat\otimes\,{\rm Ind}_{K/\Q}\mathds{1})\xrightarrow{\sim} R\Gamma(D_p,T_\f^\dagger\,\widehat\otimes\,\mathds{1})\oplus R\Gamma(D_{p},T_\f^\dagger\,\widehat\otimes\,\epsilon_K)\xrightarrow{\sim} R\Gamma(D_\p,T_\f^\dagger)\oplus R\Gamma(D_{\p^c},T_\f^\dagger)\,.
\end{equation}
\end{lemma}

\begin{proof}
This is clear. We record the explicit description of the isomorphisms (on the level of continuous cochains) \eqref{eqn_lemma_shapiro_explicit_1} and \eqref{eqn_lemma_shapiro_explicit_2}, which will be useful in what follows:

\begin{align}
    \label{eqn_exact_sequence_BDP_splitting_CM_Q_cohom_bisbis}
    \begin{aligned}
    \resizebox{0.91\hsize}{!}{
    \xymatrix{C^\bullet_{\rm cont}(D_p,T_\hf^\dagger\,\widehat\otimes\, {\rm Ind}_{K/\Q}\Psi) \ar@{=}[r] & C^\bullet_{\rm cont}(D_p,  T_\hf^\dagger\,\widehat\otimes\,(1\otimes\Psi)\oplus T_\hf^\dagger \,\widehat\otimes\,(c\otimes \Psi))\ar[ld]_(.45){\sim}\\
    C^\bullet_{\rm cont}(D_p,  T_\hf^\dagger\,\widehat\otimes\,(1\otimes\Psi)\oplus C^\bullet_{\rm cont}(D_p,T_\hf^\dagger \,\widehat\otimes\,(c\otimes \Psi))\ar[r]^(.55){\sim}_(.55){\iota_p}& C^\bullet_{\rm cont}(D_\p,  T_\hf^\dagger\,\widehat\otimes\,\Psi)\oplus C^\bullet_{\rm cont}(D_\p,T_\hf^\dagger \,\widehat\otimes\,\Psi^c)\ar[d]^{\sim}\\
    &C(K_p,T_\hf^\dagger\,\widehat\otimes\,\Psi):=C^\bullet_{\rm cont}(D_\p, T_\hf^\dagger\,\widehat\otimes\,\Psi)\oplus C^\bullet_{\rm cont}(D_{\p^c}, T_\hf^\dagger\,\widehat\otimes\,\Psi) \\
     }
     }
    \end{aligned}
\end{align}

\begin{align}
    \label{eqn_exact_sequence_BDP_splitting_CM_Q_cohom_bis}
    \begin{aligned}
    \resizebox{0.91\hsize}{!}{
    \xymatrix@C=.1cm{C^\bullet_{\rm cont}(D_p,T_\hf^\dagger\otimes {\rm Ind}_{K/\Q}\mathds{1}) \ar@{=}[d]  &&C^\bullet_{\rm cont}(D_p, T_\hf^\dagger\otimes \mathds{1})\oplus C^\bullet_{\rm cont}(D_p, T_\hf^\dagger\otimes\epsilon_K)\ar@{=}[dd]_{{(x\otimes 1,y\otimes \xi_{\epsilon_K})\mapsto (\frac{1+c}{2}\otimes x,\frac{1-c}{2}\otimes y)}}\\
     C^\bullet_{\rm cont}(D_p, 1\otimes T_\hf^\dagger)\oplus  C^\bullet_{\rm cont}(D_p, c\otimes T_\hf^\dagger)\ar[d]_{\sim}&& \\
    C^\bullet_{\rm cont}(D_p, 1\otimes T_\hf^\dagger\oplus c\otimes T_\hf^\dagger)\ar[rd]^{\sim}_{{(1\otimes a,c\otimes b){\mapsto} (a\,\circ\, \iota_p, b\,\circ\, \iota_p^c)} \qquad\qquad} \ar@{<->}[rr]_{\sim}^{\substack{{(1\otimes\frac{1}{2}(x+y), c\otimes\frac{1}{2}(x-y)){\mapsfrom}(\frac{1+c}{2}\otimes x,\frac{1-c}{2}\otimes y)}\\ {(1\otimes a, c\otimes b){\mapsto}(\frac{1+c}{2}\otimes (a+b),\frac{1-c}{2}\otimes (a-b))}}}& &C^\bullet_{\rm cont}(D_p, \frac{1+c}{2}\otimes T_\hf^\dagger)\oplus C^\bullet_{\rm cont}(D_p, \frac{1-c}{2}\otimes T_\hf^\dagger) \\
   & C^\bullet_{\rm cont}(D_\p, T_\hf^\dagger)\oplus C^\bullet_{\rm cont}(D_{\p^c}, T_\hf^\dagger)=:C(K_p,T_\hf^\dagger) &
     }
     }
    \end{aligned}
\end{align}
Here,
$\{\xi_{\epsilon_K}\}\in \Zp(\epsilon_K)$ is the basis that corresponds to $\{\frac{1-c}{2}\}$ under the isomorphism $\Zp(\epsilon_K)\simeq \Zp[\Gal(K/\Q)]^{c=-1}\,.$
\end{proof}

\begin{corollary}
We have the following diagram with Cartesian squares and split exact rows:
\begin{equation}
    \label{eqn_exact_sequence_BDP_splitting_CM_Q_cohom}
    \begin{aligned}
    \resizebox{0.91\hsize}{!}{
    \xymatrix{&&\qquad\,\,\,\, C^\bullet_{\rm cont}(K_p,T_\hf^\dagger)\oplus C^\bullet_{\rm cont}(K_p,T_\hf^\dagger\,\widehat\otimes\, \Psi_{\rm ad})\ar@{=}[d]^{\eqref{eqn_U3_vs_U2}}&& \\
    0\ar[r]& C^\bullet_{\rm cont}(K_p, T_\hf^\dagger)\ar[r]  & C^\bullet_{\rm cont}(K_p,T_\hf^\dagger\,\widehat\otimes\, U_1^\dagger)\ar[r] &C^\bullet_{\rm cont}(K_p, T_\hf^\dagger\otimes\Psi_{\rm ad})\ar[r]&0\\
    0\ar[r]&C^\bullet_{\rm cont}(D_p, T_\hf^\dagger\otimes {\rm Ind}_{K/\Q}\mathds{1})\ar[r] \ar[u]^{\eqref{eqn_lemma_shapiro_explicit_2}}_{\sim} & C^\bullet_{\rm cont}(D_p, T_2^\dagger)\ar[r] \ar[u]^{\eqref{eqn_lemma_shapiro_explicit_2}\oplus \eqref{eqn_lemma_shapiro_explicit_1}}_\sim&C^\bullet_{\rm cont}(D_p, V_\hf^\dagger\otimes {\rm Ind}_{K/\Q}\Psi_{\rm ad})\ar[u]^{\eqref{eqn_lemma_shapiro_explicit_1}}_{\sim}\ar[r]&0
    }}
    \end{aligned}
    \end{equation}
where
\begin{itemize}
\item $C^\bullet_{\rm cont}(K_p,-)\,:=\,C^\bullet_{\rm cont}(D_\p,-)\oplus  C^\bullet_{\rm cont}(D_{\p^c},-)$,
    \item the vertical isomorphism in the middle is deduced from the canonical identification \eqref{eqn_U3_vs_U2}.
\end{itemize}
\end{corollary}

\begin{lemma}
\label{lemma_3_8_15_05_2021}
The semi-local Shapiro's lemma isomorphism
$$C^\bullet_{\rm cont}(D_p, T_2^\dagger) \xrightarrow[\eqref{eqn_lemma_shapiro_explicit_1}\oplus\eqref{eqn_lemma_shapiro_explicit_2}]{\sim} C^\bullet_{\rm cont}(K_p,T_\hf^\dagger)\oplus C^\bullet_{\rm cont}(K_p,T_\hf^\dagger\,\widehat\otimes\, \Psi_{\rm ad})=C^\bullet_{\rm cont}(K_p,T_\hf^\dagger\,\widehat\otimes\, U_1^\dagger)$$
restricts to a canonical identification
$$C^\bullet_{\rm cont}(D_p,F^+_{\hg}T_2^\dagger)\xrightarrow{\sim}C^\bullet_{\rm cont}(D_\p,T_\hf^\dagger)\oplus C^\bullet_{\rm cont}(D_{\p},T_\hf^\dagger\widehat{\otimes}\Psi_{\rm ad})=C^\bullet_{\rm cont}(D_\p,T_\hf^\dagger\,\widehat\otimes\, U_1^\dagger)\,.$$
\end{lemma}

\begin{proof}
Recall that $T_\g={\rm Ind}_{K/\QQ}\Psi$, so that $$T_\g\vert_{D_p}\stackrel{\iota_p}{=}\Psi_{\vert_{D_\p}}\oplus \Psi^c_{\vert_{D_{\p}}}\xrightarrow[{\rm id}\oplus c_*]{\sim}\Psi_{\vert_{D_\p}}\oplus \Psi_{\vert_{D_{\p^c}}}.$$ 
With our present choices, we note that $\Psi_{\vert_{I_\p}}$ is non-trivial whereas $\Psi^c_{\vert_{I_{\p}}}$ is trivial. This in turn tells us that
$$F^+T_\g\vert_{D_p}=\Psi_{\vert_{D_\p}}\oplus\{0\}\,.$$ 
The proof of our lemma follows from the following chain of natural identifications:
\begin{align*}
    C^\bullet_{\rm cont}(D_p,F^+_{\hg}T_2^\dagger)\xrightarrow[\iota_p]{\sim} C^\bullet_{\rm cont}(D_\p, T_\f^\dagger\,\widehat\otimes\,\Psi\otimes (\Psi^{-1}\oplus (\Psi^c)^{-1})=C^\bullet_{\rm cont}(D_\p, T_\f^\dagger)\oplus C^\bullet_{\rm cont}(D_\p, T_\f^\dagger\,\widehat\otimes\,\Psi_{\rm ad})\,.
\end{align*}
\end{proof}

\begin{defn}
\label{defn_BDP_local_conditions}
For $\eta\in\{\mathds{1},\Psi_{\rm ad}\}$ and $?\in\{1,2\}$, we define the Greenberg local conditions $\Delta_{\rm BDP}$ on the $G_K$-representation $X\in \{T_\f^\dagger{}_{\vert_{G_K}}\,\widehat\otimes\, \eta\,,\, T_\f^\dagger{}_{\vert_{G_K}}\,\widehat\otimes\,U_?^\dagger\}$ via the morphisms 
$$\iota_\p^+: C^\bullet_{\rm cont}(G_\p,X)\xrightarrow{\rm id} C^\bullet_{\rm cont}(G_\p,X)\,,$$
$$\iota_{\p^c}^+: \{0\}\lra C^\bullet_{\rm cont}(G_\p,X)$$
corresponding to the choices of 1-step filtrations $F^+_\p X=X$ and $F^+_{\p^c} X=\{0\}$. We put 
$$U_p^\bullet(\Delta_{\rm BDP},T_\f^\dagger{}_{\vert_{G_K}}\,\widehat\otimes\, \eta):=C^\bullet_{\rm cont}(G_\p,T_\f^\dagger{}_{\vert_{G_K}}\,\widehat\otimes\, \eta)\oplus \{0\}\,.$$
\end{defn}

\begin{defn}
\label{defn_propagate_BDP_to_V_f_induced_1_and_Psiad}
For $\eta\in\{\mathds{1},\Psi_{\rm ad}\}$, we let $\Delta_{\rm BDP}$ also denote the Greenberg local conditions on $T_\f^\dagger\,\widehat\otimes\, {\rm Ind}_{K/\QQ}\,\eta$ that one obtains by propagating the local conditions $\Delta_\g$ on $T_2^\dagger$ via the exact sequence \eqref{eqn_CM_split_exact_sequence}. 
\end{defn}

In Lemma~\ref{lemma_propagate_BDP_to_V_f_induced_1_and_Psiad} and Corollary~\ref{cor_semi_local_shapiro_Delta_g_vs_BDP} below, we explain (to justify the nomenclature of these objects) that these two sets of local conditions introduced in Definition~\ref{defn_BDP_local_conditions} and Definition~\ref{defn_propagate_BDP_to_V_f_induced_1_and_Psiad} coincide under the semi-local Shapiro morphism.

\begin{lemma}
\label{lemma_propagate_BDP_to_V_f_induced_1_and_Psiad}
For $\eta\in\{\mathds{1},\Psi_{\rm ad}\}$, the Greenberg local conditions $\Delta_{\rm BDP}$ on $T_\f^\dagger\,\widehat\otimes\, {\rm Ind}_{K/\QQ}\,\eta$ are given in explicit terms via the  following morphisms at $p$:
\begin{align}
\label{eqn_lemma_propagate_BDP_to_V_f_induced_1_and_Psiad_induced_1}
\begin{aligned}
\resizebox{0.92\hsize}{!}{\xymatrix@C=.3cm{
U_p^\bullet(\Delta_{\rm BDP},T_\f^\dagger\otimes {\rm Ind}_{K/\QQ}\,\mathds{1})\ar@{=}[r]&\left\{\left(1\otimes x, c\otimes 0\right): x\in C^\bullet_{\rm cont}(D_p,T_\f^\dagger)\right\}\ar@{^{(}->}[r]\ar[dd]_{\sim}&C^\bullet_{\rm cont}(D_p,1\otimes T_\f^\dagger\oplus c\otimes T_\f^\dagger)\ar@{=}[dd]_{(1\otimes a, c\otimes b){\mapsto}(\frac{1+c}{2}\otimes (a+b),\frac{1-c}{2}\otimes (a-b))}\ar[rrr]^(.6){\sim}_(.6){\substack{{}\\{}\\{(1\otimes a,c\otimes b){\mapsto} (a\,\circ\, \iota_p, b\,\circ\, \iota_p^c)}}}&&& C^\bullet_{\rm cont}(K_p,T_\f^\dagger)\\\\
& \left\{\left(\frac{1+c}{2}\otimes x, \frac{1-c}{2}\otimes x\right): x\in C^\bullet_{\rm cont}(D_p,T_\f^\dagger)\right\}\ar@{^{(}->}[r]& C^\bullet_{\rm cont}(D_p,\frac{1+c}{2}\otimes T_\f^\dagger\oplus\frac{1-c}{2}\otimes T_\f^\dagger)
}
}
\end{aligned}
\end{align}  
and
\begin{align}
\label{eqn_lemma_propagate_BDP_to_V_f_induced_1_and_Psiad_Psiad}
\begin{aligned}
U_p^\bullet(\Delta_{\rm BDP},T_\f^\dagger\,\widehat\otimes\, {\rm Ind}_{K/\QQ}\Psi_{\rm ad})&:= \left\{\left(1\otimes x, c\otimes 0\right): x\in C^\bullet_{\rm cont}(D_p,T_\f^\dagger\,\widehat\otimes\, \Psi_{\rm ad})\right\}\\
&\qquad\xrightarrow[\iota_p]{\sim} C^\bullet_{\rm cont}(D_\p,T_\f^\dagger\,\widehat\otimes\, \Psi_{\rm ad}) \oplus \{0\}\hookrightarrow C^\bullet_{\rm cont}(K_p,T_\f^\dagger\,\widehat\otimes\, \Psi_{\rm ad})\,.
\end{aligned}
\end{align} 
\end{lemma}

\begin{proof}
Both assertions follow from Lemma~\ref{lemma_3_8_15_05_2021} combined with the explicit description of the semi-local Shapiro morphisms in \eqref{eqn_exact_sequence_BDP_splitting_CM_Q_cohom_bisbis} and \eqref{eqn_exact_sequence_BDP_splitting_CM_Q_cohom_bis}.
\end{proof}

\begin{corollary}
\label{cor_semi_local_shapiro_Delta_g_vs_BDP}
For $\eta=\mathds{1}, \Psi_{\rm ad}$, we have the following commutative diagram of cochain complexes, where the vertical arrows are induced from the semi-local Shapiro's lemma, c.f. \eqref{eqn_exact_sequence_BDP_splitting_CM_Q_cohom_bisbis} and \eqref{eqn_exact_sequence_BDP_splitting_CM_Q_cohom_bis}:
\begin{equation}
\label{eqn_prop_semi_local_shapiro_Delta_g_vs_BDP_1}
\begin{aligned}
 \xymatrix{
U_p^\bullet(\Delta_{\rm BDP},T_\f^\dagger\,\widehat\otimes\, {\rm Ind}_{K/\QQ}\eta)\ar[d]_{\rm Lemma~\ref{lemma_propagate_BDP_to_V_f_induced_1_and_Psiad}}^{\sim} \ar[rr]^{i_p^+}_{\rm Definition~\ref{defn_propagate_BDP_to_V_f_induced_1_and_Psiad}}&& C^\bullet_{\rm cont}(D_p, T_\f^\dagger\,\widehat\otimes\, {\rm Ind}_{K/\QQ}\eta)\ar[d]^{\eqref{eqn_exact_sequence_BDP_splitting_CM_Q_cohom_bisbis}{\rm \,and\, }\eqref{eqn_exact_sequence_BDP_splitting_CM_Q_cohom_bis}}_\sim\\
U_p^\bullet(\Delta_{\rm BDP},T_\f^\dagger\vert_{G_K}\,\widehat\otimes\,\eta) \ar[rr]_{\iota_\p^+\oplus \iota_{\p^c}^+}^{\rm Definition~\ref{defn_BDP_local_conditions}}&& C^\bullet_{\rm cont}(K_p, T_\f^\dagger\vert_{G_K}\,\widehat\otimes\,\eta)
}
\end{aligned}
\end{equation}
\end{corollary}


\subsubsection{}
\label{subsubsec_3_2_5_24_05_2021_11_17}In this paragraph, we will explicitly describe the propagation of the Greenberg-local condition $\Delta_{\rm BDP}$ on $T_\hf^\dagger\,\widehat\otimes\,{\rm Ind}_{K/\Q}\mathds{1}$ (given as in Definition~\ref{defn_propagate_BDP_to_V_f_induced_1_and_Psiad} and Lemma~\ref{lemma_propagate_BDP_to_V_f_induced_1_and_Psiad}) to $T_\hf^\dagger\otimes\mathds{1}$ and $T_\hf^\dagger\otimes\epsilon_K$, via the split exact sequence
\begin{equation}
    \label{eqn_induced_split_exact_sequence_trivial}
    0\lra \frac{1+c}{2}\otimes T_\f^\dagger=T_\hf^\dagger\otimes\mathds{1}\lra T_\hf^\dagger\,\widehat\otimes\,{\rm Ind}_{K/\Q}\mathds{1} \lra T_\hf^\dagger\otimes\epsilon_K=\frac{1-c}{2}\otimes T_\f^\dagger\lra 0\,.
\end{equation}
The explicit description in \eqref{eqn_exact_sequence_BDP_splitting_CM_Q_cohom_bis} of the semi-local Shapiro morphism shows that $\Delta_{\rm BDP}$ propagates to the local conditions given by
\begin{align}
\label{eqn_propagated_local_conditions_Delta_g_11}
\begin{aligned}
U^\bullet_{p}(\Delta_{\rm BDP}, T_\hf^\dagger\otimes \mathds{1})=&\{0\}=: U^\bullet_{\rm cont}(\Delta_{0}, T_\hf^\dagger\otimes \mathds{1}) \qquad \hbox{(strict conditions)}\,,\\ 
U^\bullet_{p}(\Delta_{\rm BDP}, T_\hf^\dagger\otimes \epsilon_K)=C^\bullet_{\rm cont}(\Q_p, &\, T_\hf^\dagger\otimes \epsilon_K)=: U^\bullet_p(\Delta_{\emptyset}, T_\hf^\dagger\otimes \epsilon_K) \qquad \hbox{(relaxed conditions)}\,.
\end{aligned}
\end{align}

\subsubsection{}
\label{subsubsec_3_2_6_16_05_2021}
We next analyze the $\Delta_{\hg}$-local conditions propagated through the exact sequence \eqref{eqn_splitting_dual_global}, which is a restatement (in the scenario when $\hg$ is a CM family) of the exact sequence \eqref{eqn_factorisationrepresentations_dual_trace} deduced from the dual-trace morphism.

We have ${\rm tr}^*\Delta_{\g}=\Delta_0$ (strict conditions) for the propagated local conditions on $T_\f^\dagger$. In the remainder of \S\ref{subsubsec_3_2_6_16_05_2021}, we will next explain that 
\begin{align}
\label{eqn_propagated_local_conditions_Delta_g_14_2}
\begin{aligned}
U^\bullet_{p}({\rm tr}^*\Delta_{\hg}, M_2^\dagger)&=U^\bullet_{p}({\rm tr}^*\Delta_{\hg},T_{\hf}^\dagger\otimes\epsilon_K)\oplus U^\bullet_{p}({\rm tr}^*\Delta_{\hg}, {\rm Ind}_{K/\Q}T_\hf^\dagger\,\widehat\otimes\, \Psi_{\rm ad})\\
&=U^\bullet_{p}(\Delta_{\emptyset}, T_{\hf}^\dagger\otimes\epsilon_K)\oplus U^\bullet_{p}(\Delta_{\rm BDP}, T_\hf^\dagger\,\widehat\otimes\, {\rm Ind}_{K/\Q}\,\Psi_{\rm ad})\,,
\end{aligned}
\end{align}
where the first equality is clear. We retain the notation of \S\ref{subsubsec_322_14_05_2021}. As we have seen in the proof of Lemma~\ref{lemma_3_8_15_05_2021}, the local conditions $\Delta_{\hg}$ on $T_2^\dagger:=V_\hf^\dagger\widehat\otimes {\rm Ind}_{K/\Q}\Psi \otimes {\rm Ind}_{K/\Q}\Psi^{-1}$ is given by the $D_p$-stable submodule 
\begin{align*}
   F_\g^+T_2^\dagger:= T_\hf^\dagger\,\widehat\otimes& \,{\rm span}\{v_{\Psi}\}\otimes {\rm Ind}_{K/\Q}\,\Psi^{-1}=T_\hf^\dagger\,\widehat\otimes \,{\rm span}\{v_{\Psi}\otimes \, v_{\Psi^{-1}}, v_{\Psi}\otimes c v_{\Psi^{c,-1}}\} \\
  &\qquad \hookrightarrow T_\hf^\dagger\,\widehat\otimes \,{\rm Ind}_{K/\Q}\Psi \otimes {\rm Ind}_{K/\Q}\Psi^{-1}
    \stackrel{\eqref{eqn_InfPsiptotimesIndpsiinverse}}{=}(T_{\hf}^\dagger\otimes \mathds{1})\oplus (T_{\hf}^\dagger\otimes \epsilon_K)\oplus (T_{\hf}^\dagger\otimes {\rm Ind}_{K/\Q}\Psi_{\rm ad})\,.
\end{align*}
Observing that 
$${\rm span}\left\{\frac{1+c}{2}\cdot (v_\Psi\otimes v_{\Psi^{-1}})\right\}\,\cap  \,{\rm span}\{v_{\Psi}\otimes \, v_{\Psi^{-1}}, v_{\Psi}\otimes c v_{\Psi^{c,-1}}\} =\{0\}, $$
we infer the the image $F_\g^+T_2^\dagger \subset T_2^\dagger$ under the projection 
$T_2^\dagger \xrightarrow{\eqref{eqn_factorisationrepresentations_dual_trace}} M_2^\dagger= T_{\hf}^\dagger\otimes \epsilon_K\oplus T_{\hf}^\dagger\,\widehat\otimes\, {\rm Ind}_{K/\Q}\Psi_{\rm ad}$
modulo the direct summand 
$$T_{\hf}^\dagger\,\widehat\otimes\,\mathds{1}=T_\hf^\dagger\,\widehat\otimes \,{\rm span}\left\{\frac{1+c}{2}\cdot(v_{\Psi}\otimes v_{\Psi^{-1}})\right\}\hookrightarrow T_2^\dagger=T_{\hf}^\dagger\otimes \mathds{1}\oplus T_{\hf}^\dagger\otimes \epsilon_K\oplus T_{\hf}^\dagger\otimes {\rm Ind}_{K/\Q}\Psi_{\rm ad}$$ 
equals
$$\underbrace{T_\hf^\dagger\,\widehat\otimes \,{\rm span}\left\{\frac{1-c}{2}\cdot(v_{\Psi}\otimes v_{\Psi^{-1}})\right\}}_{:=F_\g^+
\left(T_\hf^\dagger\,\widehat\otimes \,\epsilon_K\right)=T_\hf^\dagger\,\widehat\otimes \,\epsilon_K}\oplus \underbrace{T_\hf^\dagger\,\widehat\otimes \,{\rm span}\left\{v_{\Psi}\otimes c v_{\Psi^{c,-1}} \right\}}_{:=F_\g^+\,\left( T_{\hf}^\dagger\,\widehat\otimes\,{\rm Ind}_{K/\Q}\Psi_{\rm ad}\right)}\,.$$
This tells us that 
$ U_p^\bullet({\rm tr}^*\Delta_\g,T_\hf^\dagger\,\widehat\otimes \,\epsilon_K)=U^\bullet_{p}(\Delta_{\emptyset}, T_{\hf}^\dagger\otimes\epsilon_K)$ and $U_p^\bullet({\rm tr}^*\Delta_\g,T_\hf^\dagger\,\widehat\otimes \,\Psi_{\rm ad})\stackrel{\eqref{eqn_lemma_propagate_BDP_to_V_f_induced_1_and_Psiad_Psiad}}{=} U^\bullet_{p}(\Delta_{\rm BDP}, T_{\f}^\dagger\,\widehat\otimes \,\Psi_{\rm ad})$, as required.

\subsection{Factorization: \texorpdfstring{$8=4+4$}{}.}
\label{subseec_7_1_2022_09_18_1414} We shall work under the following hypotheses until the end of this paper.

\subsubsection{Hypotheses}\label{subsubsec_hypothesis_Sec_7} We assume the family $\f$ is non-CM, and that Assumption 4.27 in \cite{BS_Part1} is valid for $\f$. Note that this condition is readily for the CM family $\g$ and it trivially holds if the tame conductor of the Hida family $\f$ is square-free. We assume in addition that \ref{item_Irr} holds for the family $\f$, \ref{item_Dist} for both $\f$ and $\g$. We remark that the condition (Irr+) in \cite[\S6.1.1]{BS_Part1} fails for $\g$, and therefore the results of \S6 in op. cit. does not apply in the present setup.
    
The hypotheses $(H^0)$ and (Tam) introduced in \cite[\S4.4.2]{BS_Part1} are still valid for $T_3^\dagger$,$T_2^\dagger$, $M_2^\dagger$ and $T_\f^\dagger$ thanks to the assumptions listed above. In particular, the main conclusions in \cite[\S4]{BS_Part1} apply and show that the Selmer complexes associated with these Galois representations are perfect. The constructions of \cite[\S5]{BS_Part1} apply as well; we will recall these in \S\ref{subsubsec_711_2022_09_18_1509} below.

Finally, we continue to assume that the rings $\cR_\f$ and $\cR_\g$ are both regular. We note this assumption can be ensured on passing to a smaller wide-open disc in the weight space (cf. \cite[\S5]{BSV}, where $\cR_?$ here, for $?=\f,\g$, plays the role of $\LL_?$ in op. cit.).

We continue to work in the scenario of \S\ref{subsubsec_2017_05_17_1800} so that we have
\begin{itemize}
    \item[\mylabel{item_root_numbers_general_CM}{${\bf (Sign')}$}] 
    \quad $\varepsilon(\hf)=-1$,\quad $\varepsilon(\hf\otimes \epsilon_K)=+1$\,, \quad and \quad $\varepsilon({\hf_{\kappa}}_{/K}\otimes \Psi_\lambda^{\rm ad})= \begin{cases}
         +1& \hbox{ if } (\kappa,\lambda)\in \cW_2^\g\,\\
         -1& \hbox{ if } (\kappa,\lambda)\in \cW_2^\hf\,
     \end{cases}$\,
\end{itemize}
where we recall that $\epsilon_K$ is the unique quadratic Dirichlet character attached to $K$.


\subsubsection{}  We are now ready to begin our work on algebraic $p$-adic $L$-functions. 
\begin{proposition}
    \label{prop_BDP_8_4_4_algebraic} We have the following exact triangles in the derived category:
\item[i)] 
$\widetilde{R\Gamma}_{\rm f}(G_{K,\Sigma_K},T_\f^\dagger,\Delta_{\rm BDP})\otimes_{\cR_\f}\cR_2\lra \widetilde{R\Gamma}_{\rm f}(G_{K,\Sigma_K},T_\f^\dagger\,\widehat\otimes\,U_1^\dagger,\Delta_{\rm BDP})\lra 
\widetilde{R\Gamma}_{\rm f}(G_{K,\Sigma_K},T_\f^\dagger\,\widehat\otimes\,\Psi_{\rm ad},\Delta_{\rm BDP})\,.$
\item[ii)] $\widetilde{R\Gamma}_{\rm f}(G_{K,\Sigma_K},T_\f^\dagger,\Delta_{\rm BDP})\otimes_{\cR_\f}\cR_2  \lra \widetilde{R\Gamma}_{\rm f}(G_{\QQ,\Sigma},T_2^\dagger,\Delta_{\g})\lra \widetilde{R\Gamma}_{\rm f}(G_{K,\Sigma_K},T_\f^\dagger\,\widehat\otimes\,\Psi_{\rm ad},\Delta_{\rm BDP})$\,. 

In particular, we have a factorization 
$$\det\,\widetilde{R\Gamma}_{\rm f}(G_{\QQ,\Sigma},T_2^\dagger,\Delta_{\g})=\varpi_{2,1}^*\det\,\widetilde{R\Gamma}_{\rm f}(G_{K,\Sigma_K},T_\f^\dagger,\Delta_{\rm BDP})\,\cdot\, \det\,\widetilde{R\Gamma}_{\rm f}(G_{K,\Sigma_K},T_\f^\dagger\,\widehat\otimes\,\Psi_{\rm ad},\Delta_{\rm BDP})\,.$$
\end{proposition}

\begin{proof}
The first assertion follows from definitions, whereas the second from the first part combined with \eqref{eqn_U3_vs_U2} and Corollary~\ref{cor_semi_local_shapiro_Delta_g_vs_BDP}. We remark that determinants of the Selmer complexes that make an appearance in the statement of our proposition make sense, since all these complexes are perfect thanks to our running hypotheses listed in \S\ref{subsubsec_hypothesis_Sec_7}.
\end{proof}
\begin{remark}
One may also prove a 3-variable version of Proposition~\ref{prop_BDP_8_4_4_algebraic} over the ring $\cR_3$, factoring the determinant of the Selmer complex $\widetilde{R\Gamma}_{\rm f}(G_{\QQ,\Sigma},T_3^\dagger,\Delta_{\g})$. Note that our main results in \S\ref{sec_super_factorization} does not dwell on this factorization, as opposed to \S\ref{subsec_75_2022_06_02_0900} below where the very first step in our argument (laid out in \S\ref{sec:8=4+4}) is a factorization of a $p$-adic $L$-function in $3$-variables.
\end{remark}

\subsubsection{Leading terms} 
\label{subsubsec_711_2022_09_18_1509}
We recall the ``module of leading terms'' $\delta(T_2^\dagger,\Delta_\g)\subset \cR_2^\dagger$ given as in \cite[\S5.2.2]{BS_Part1} (see also \cite[\S6.1.4]{BS_Part1}). Let us observe that we have 
\[
\widetilde{R\Gamma}_{\rm f}(G_{K,\Sigma_K},T_\f^\dagger,\Delta_{\rm BDP}) \in D_{\rm parf}^{[1,2]}({}_{\cR_\f}{\rm Mod}) \,\,\, 
\textrm{ and } \,\,\, \widetilde{R\Gamma}_{\rm f}(G_{K,\Sigma_K},T_\f^\dagger\,\widehat\otimes\,\Psi_{\rm ad},\Delta_{\rm BDP})\in D_{\rm parf}^{[1,2]}({}_{\cR_2}{\rm Mod})
\]
thanks to Proposition 4.38 in op. cit., and 
\begin{equation}
    \label{eqn_2022_09_18_2032}
    \chi(\widetilde{R\Gamma}_{\rm f}(G_{K,\Sigma_K},T_\f^\dagger,\Delta_{\rm BDP}))=0=\chi(\widetilde{R\Gamma}_{\rm f}(G_{K,\Sigma_K},T_\f^\dagger\,\widehat\otimes\,\Psi_{\rm ad},\Delta_{\rm BDP}))
\end{equation}
for the Euler--Poincar\'e characteristics of the indicated Selmer complexes. As a result, we are in the situation of \cite[\S5.2.2]{BS_Part1} and we have the modules
$$\delta(T_\f^\dagger,\Delta_{\rm BDP})\subset \cR_\f\,,\qquad  \delta(T_\f^\dagger\,\widehat\otimes\,\Psi_{\rm ad},\Delta_{\rm BDP})\subset \cR_2$$
of leading terms.

\begin{proposition}
    \label{prop_2022_09_18_1639}
    $\delta(T_2^\dagger,\Delta_\g)\neq 0$ if and only if $\delta(T_\f^\dagger,\Delta_{\rm BDP})\neq 0\neq  \delta(T_\f^\dagger\,\widehat\otimes\,\Psi_{\rm ad},\Delta_{\rm BDP})$. 
\end{proposition}

\begin{proof}
The exact triangle in Proposition~\ref{prop_BDP_8_4_4_algebraic}(ii) tells us that
\begin{align}
 \notag    0\lra \widetilde{H}^1_{\rm f}(G_{K,\Sigma_K},T_\f^\dagger,\Delta_{\rm BDP})\otimes_{\cR_\f}\cR_2  \lra \widetilde{H}^1_{\rm f}(G_{\QQ,\Sigma},T_2^\dagger,\Delta_{\g})\lra 
     \widetilde{H}^1_{\rm f}&(G_{K,\Sigma_K},T_\f^\dagger\,\widehat\otimes\,\Psi_{\rm ad},\Delta_{\rm BDP})\\
          \label{eqn_2022_09_18_2023}
     &\lra  \widetilde{H}^2_{\rm f}(G_{K,\Sigma_K},T_\f^\dagger,\Delta_{\rm BDP})\otimes_{\cR_\f}\cR_2\,.
\end{align}
It also follows from \cite[Theorem 5.11]{BS_Part1} and \eqref{eqn_2022_09_18_2032} that 
\begin{align*}
    \begin{aligned}
     \label{eqn_2022_09_18_2026}
     \delta(T_2^\dagger,\Delta_\g)\neq 0 \quad &\iff{\quad \widetilde{H}^1_{\rm f}(G_{\QQ,\Sigma},T_2^\dagger,\Delta_{\g})=0}\,,\\\\
     \delta(T_\f^\dagger,\Delta_{\rm BDP})\neq 0\neq  \delta(T_\f^\dagger\,\widehat\otimes\,\Psi_{\rm ad},\Delta_{\rm BDP}) \, &\iff{\, \widetilde{H}^1_{\rm f}(G_{K,\Sigma_K},T_\f^\dagger,\Delta_{\rm BDP})=0=\widetilde{H}^1_{\rm f}(G_{K,\Sigma_K},T_\f^\dagger\,\widehat\otimes\,\Psi_{\rm ad},\Delta_{\rm BDP})}\\
     &\iff{\, {\rm rk}\,\widetilde{H}^2_{\rm f}(G_{K,\Sigma_K},T_\f^\dagger,\Delta_{\rm BDP})=0={\rm rk}\,\widetilde{H}^2_{\rm f}(G_{K,\Sigma_K},T_\f^\dagger\,\widehat\otimes\,\Psi_{\rm ad},\Delta_{\rm BDP})}\,.
    \end{aligned}
\end{align*}
The proof of our proposition follows by combining these equivalences with \eqref{eqn_2022_09_18_2023}.
\end{proof}

In \S\ref{subsubsec_2022_09_18_2046}, we will give an explicit criterion for the validity of the non-vanishing condition in Proposition~\ref{prop_2022_09_18_1639} in terms of Beilinson--Flach elements and Beilinson--Kato elements.

\begin{corollary}
\label{cor_2022_09_19_1230}
Suppose that $\delta(T_2^\dagger,\Delta_\g)\neq 0$. Then the $\cR_\f$-module $\widetilde{H}^2_{\rm f}(G_{K,\Sigma_K},T_\f^\dagger,\Delta_{\rm BDP})$ as well as both $\cR_2$-modules $\widetilde{H}^2_{\rm f}(G_{K,\Sigma_K},T_\f^\dagger\,\widehat\otimes\,\Psi_{\rm ad},\Delta_{\rm BDP})$ and $\widetilde{H}^2_{\rm f}(G_{\QQ,\Sigma},T_2^\dagger,\Delta_{\g})$ are torsion with projective dimension $1$. Moreover, we have a factorization
$$\det\,\widetilde{H}^2_{\rm f}(G_{\QQ,\Sigma},T_2^\dagger,\Delta_{\g})= \det\,\widetilde{H}^2_{\rm f}(G_{K,\Sigma_K},T_\f^\dagger\,\widehat\otimes\,\Psi_{\rm ad},\Delta_{\rm BDP})\,\cdot\,\varpi_{2,1}^*\det\,\widetilde{H}^2_{\rm f}(G_{K,\Sigma_K},T_\f^\dagger,\Delta_{\rm BDP})$$
of algebraic $p$-adic $L$-functions.
\end{corollary}

\begin{proof}
This is an immediate consequence of Corollary~4.41 and Theorem~5.11 in \cite{BS_Part1}, together with Proposition~\ref{prop_BDP_8_4_4_algebraic}(ii) and Proposition~\ref{prop_2022_09_18_1639}.
\end{proof}


\subsection{Factorization: \texorpdfstring{$4=2+2$}{}}
\label{subsec_factorize_BDP_padic_L_function}
Our main goal in \S\ref{subsec_factorize_BDP_padic_L_function} is to factor the algebraic ${\rm BDP}^2$ $p$-adic $L$-function associated with the Rankin--Selberg convolution $\f_{/K}\otimes \mathds{1}$.  Note that the $p$-adic analytic variant of this result has been proved in \S\ref{subsec_8_2_2022_09_13_1731} under a less stringent hypothesis.
\begin{theorem}
\label{thm_main_BDP_factorization}
In addition to our running hypothesis listed in \S\ref{subsubsec_hypothesis_Sec_7}, suppose that $\delta(T_2^\dagger,\Delta_\g)\neq 0$. We assume also that the conditions \eqref{item_MC}, \eqref{item_BI} and \eqref{item_non_anom} hold true. Then,
\begin{equation}
\label{eqn_2022_09_20_1331}
    \det\widetilde{R\Gamma}_{\rm f}(G_{K,\Sigma_K},T_\f^\dagger,\Delta_{\rm BDP})=\det\widetilde{H}^2_{\rm f}(G_{K,\Sigma_K},T_\f^\dagger,\Delta_{\rm BDP})={\rm Exp}^*_{F^-T_\f^\dagger\otimes\epsilon_K}(\delta(T_\f^\dagger\otimes\epsilon_K,\Delta_\emptyset))\cdot {\rm Log}_{\omega_\f}(\delta(T_\f^\dagger,\Delta_\emptyset))\,.
\end{equation}
\end{theorem}


The remainder of \S\ref{subsec_factorize_BDP_padic_L_function} is dedicated to a proof of Theorem~\ref{thm_main_BDP_factorization}. 

\subsubsection{}
\label{subsubsec_1_18_05_2021_subsec_factorize_BDP_padic_L_function} 
We recall that we have an isomorphism of Selmer complexes
\begin{equation}
\label{eqn_2022_09_19_1231}
    \widetilde{R\Gamma}_{\rm f}(G_{\QQ,\Sigma},T_\f^\dagger\,\otimes\,{\rm Ind}_{K/\QQ}\mathds{1},\Delta_{\rm BDP})\xrightarrow[\mathfrak{sh}]{\,\,\sim\,\,} \widetilde{R\Gamma}_{\rm f}(G_{K,\Sigma_K},T_\f^\dagger,\Delta_{\rm BDP}) 
\end{equation}
in the derived category thanks to Shapiro's lemma and Corollary~\ref{cor_semi_local_shapiro_Delta_g_vs_BDP}.  

Combining \eqref{eqn_factor_trivial_rep_by_idempotents} and \eqref{eqn_propagated_local_conditions_Delta_g_11}, we deduce the following exact triangle in the derived category:
\begin{equation}
    \label{eqn_5_1_19_05_2021}
    \widetilde{R\Gamma}_{\rm f}(G_{\QQ,\Sigma},T_\f^\dagger,\Delta_0)\lra \widetilde{R\Gamma}_{\rm f}(G_{\QQ,\Sigma},T_\f^\dagger\,\otimes\,{\rm Ind}_{K/\QQ}\mathds{1},\Delta_{\rm BDP})\lra 
\widetilde{R\Gamma}_{\rm f}(G_{\QQ,\Sigma},T_\f^\dagger\,\otimes\,\epsilon_K,\Delta_{\emptyset})\,.
\end{equation}
On the level of cohomology, we have the exact sequence
\begin{align}
    \label{eqn_5_2_19_05_2021}
 \notag  \widetilde{H}^1_{\rm f}(G_{\QQ,\Sigma},T_\f^\dagger,\Delta_0)&\lra \widetilde{H}^1_{\rm f}(G_{\QQ,\Sigma},T_\f^\dagger\,\otimes\,{\rm Ind}_{K/\QQ}\mathds{1},\Delta_{\rm BDP})\lra  \widetilde{H}^1_{\rm f}(G_{\QQ,\Sigma},T_\f^\dagger\,\otimes\,\epsilon_K,\Delta_{\emptyset})\xrightarrow{\delta^1_K}  \widetilde{H}^2_{\rm f}(G_{\QQ,\Sigma},T_\f^\dagger,\Delta_0)\\
   &\qquad\qquad\qquad\quad\lra \widetilde{H}^2_{\rm f}(G_{\QQ,\Sigma},T_\f^\dagger\,\otimes\,{\rm Ind}_{K/\QQ}\mathds{1},\Delta_{\rm BDP})\lra 
\widetilde{H}^2_{\rm f}(G_{\QQ,\Sigma},T_\f^\dagger\,\otimes\,\epsilon_K,\Delta_{\emptyset})\lra 0
\end{align}
where the surjection on the very right follows from the vanishing of $\widetilde{H}^3_{\rm f}(G_{\QQ,\Sigma},T_\f^\dagger,\Delta_0)$, which is a consequence of the fact that the family $\f$ is cuspidal.

We may apply the discussion in \cite[\S6.1.2]{BS_Part1} with the twisted family $\f \otimes\epsilon_K$ in place of $\f$ and obtain the module
$\delta(T_\f^\dagger\otimes\epsilon_K,\Delta_\emptyset)\subset \widetilde{H}^1_{\rm f}(G_{\QQ,\Sigma},T_\f^\dagger\,\otimes\,\epsilon_K,\Delta_{\emptyset})$ of leading terms.

\begin{lemma}
\label{lemma_2022_09_19_1135}
Suppose that $\delta(T_\f^\dagger,\Delta_{\rm BDP})\neq 0$. Then:
\item[i)] $\widetilde{H}^1_{\rm f}(G_{\QQ,\Sigma},T_\f^\dagger\,\otimes\,{\rm Ind}_{K/\QQ}\mathds{1},\Delta_{\rm BDP})=0$ and $\widetilde{H}^2_{\rm f}(G_{\QQ,\Sigma},T_\f^\dagger\,\otimes\,{\rm Ind}_{K/\QQ}\mathds{1},\Delta_{\rm BDP})$ is $\cR_\f$-torsion.
\item[ii)] $\widetilde{H}^1_{\rm f}(G_{\QQ,\Sigma},T_\f^\dagger,\Delta_0)=0$ and the $\cR_\f$-module $\widetilde{H}^2_{\rm f}(G_{\QQ,\Sigma},T_\f^\dagger,\Delta_0)$ is of rank one.
\item[iii)] The $\cR_\f$-module $\widetilde{H}^1_{\rm f}(G_{\QQ,\Sigma},T_\f^\dagger\,\otimes\,\epsilon_K,\Delta_{\emptyset})$ is free of rank one and $\widetilde{H}^2_{\rm f}(G_{\QQ,\Sigma},T_\f^\dagger\,\otimes\,\epsilon_K,\Delta_{\emptyset})$ is  $\cR_\f$-torsion. Moreover, the connecting morphism $\delta^1_K$ in \eqref{eqn_5_2_19_05_2021} is injective.
\item[iv)] $\res_p\left(\delta(T_\f^\dagger\otimes\epsilon_K,\Delta_\emptyset)\right)\neq 0$.
\end{lemma}

\begin{proof}
Assertion (i) follows from Corollary~\ref{cor_2022_09_19_1230} combined with \eqref{eqn_2022_09_19_1231}. Part (ii) is immediate from (i) and the global Euler--Poincar\'e characteristic formula
$\chi(\widetilde{R\Gamma}_{\rm f}(G_{\QQ,\Sigma},T_\f^\dagger,\Delta_0))=1\,.$ The claim in (iii) that  $\widetilde{H}^2_{\rm f}(G_{\QQ,\Sigma},T_\f^\dagger\,\otimes\,\epsilon_K,\Delta_{\emptyset})$ is  $\cR_\f$-torsion follows from \eqref{eqn_5_2_19_05_2021} and Part (i). The global Euler--Poincar\'e characteristic formula $\chi(\widetilde{R\Gamma}_{\rm f}(G_{\QQ,\Sigma},T_\f^\dagger,\Delta_\emptyset))=-1$
then shows that the $\cR_\f$-module $\widetilde{H}^1_{\rm f}(G_{\QQ,\Sigma},T_\f^\dagger\,\otimes\,\epsilon_K,\Delta_{\emptyset})$ is of rank one. The fact that it is free can be proved by arguing as in the proof of \cite[Proposition~6.2(i)]{BS_Part1}. The proof that $\delta^1_K$ is injective follows from (i). Finally, Theorem~5.1.2 in op. cit. combined with (iii) shows that $\delta(T_\f^\dagger\otimes\epsilon_K,\Delta_\emptyset)\neq 0$. To conclude the proof of Part (iv), it suffices to show that $\widetilde{H}^1_{\rm f}(G_{\QQ,\Sigma},T_\f^\dagger\otimes \epsilon_K,\Delta_0)=0$. This is clear thanks to Part (iii), since the $G_\QQ$-representation $T_\f^\dagger\otimes \epsilon_K$ is self-Tate-dual and we have
$${\rm rank}\, \widetilde{H}^1_{\rm f}(G_{\QQ,\Sigma},T_\f^\dagger\otimes \epsilon_K,\Delta_0)={\rm rank}\,\widetilde{H}^2_{\rm f}(G_{\QQ,\Sigma},T_\f^\dagger\otimes \epsilon_K,\Delta_\emptyset)$$
by global duality.
\end{proof}


The exact sequence \eqref{eqn_5_2_19_05_2021} therefore simplifies (whenever $\delta(T_2^\dagger,\Delta_\g)\neq 0$) as 
\begin{align}
    \label{eqn_5_3_19_05_2021}
    \begin{aligned}
   0\lra & \,\widetilde{H}^1_{\rm f}(G_{\QQ,\Sigma},T_\f^\dagger\,\otimes\,\epsilon_K,\Delta_{\emptyset})\xrightarrow{\delta^1_K}  \widetilde{H}^2_{\rm f}(G_{\QQ,\Sigma},T_\f^\dagger,\Delta_0)\\
   &\qquad\lra \widetilde{H}^2_{\rm f}(G_{\QQ,\Sigma},T_\f^\dagger\,\otimes\,{\rm Ind}_{K/\QQ}\mathds{1},\Delta_{\rm BDP})\lra 
\widetilde{H}^2_{\rm f}(G_{\QQ,\Sigma},T_\f^\dagger\,\otimes\,\epsilon_K,\Delta_{\emptyset})\lra 0\,.
    \end{aligned}    
\end{align}

\begin{lemma}
    \label{lemma_2022_09_19_1253}
    $\delta(T_\f^\dagger,\Delta_{\rm BDP})\neq 0$ if and only if $\res_p\left(\delta(T_\f^\dagger\otimes\epsilon_K,\Delta_\emptyset)\right)\neq 0\neq \res_p\left(\delta(T_\f^\dagger,\Delta_\emptyset)\right)$.
\end{lemma}

\begin{proof}
Suppose first that $\delta(T_\f^\dagger,\Delta_{\rm BDP})\neq 0$. 
The conclusion that $\res_p\left(\delta(T_\f^\dagger\otimes\epsilon_K,\Delta_\emptyset)\right)\neq 0$ is proved as part of Lemma~\ref{lemma_2022_09_19_1135}(iv). 
To prove the non-vanishing $\res_p\left(\delta(T_\f^\dagger,\Delta_\emptyset)\right)\neq 0$, 
we remark that $\widetilde{H}^1_{\rm f}(G_{\QQ,\Sigma},T_\f^\dagger,\Delta_{0})=0$ thanks to Lemma~\ref{lemma_2022_09_19_1135}(ii). Since the $G_\QQ$-representation $T_\f^\dagger$ is self-Tate-dual, we have
$$
{\rm rank}\, \widetilde{H}^1_{\rm f}(G_{\QQ,\Sigma},T_\f^\dagger,\Delta_0)={\rm rank}\,\widetilde{H}^2_{\rm f}(G_{\QQ,\Sigma},T_\f^\dagger,\Delta_\emptyset)
$$
by global duality. 
This shows that $\widetilde{H}^2_{\rm f}(G_{\QQ,\Sigma},T_\f^\dagger,\Delta_\emptyset)$ is $\cR_\f$-torsion and we infer (using \cite{BS_Part1}, Theorem~5.12) that $\delta(T_\f^\dagger,\Delta_\emptyset)$ is non-zero. Moreover, since
$
\{0\}=\widetilde{H}^1_{\rm f}(G_{\QQ,\Sigma},T_\f^\dagger,\Delta_0)=\ker(\widetilde{H}^1_{\rm f}(G_{\QQ,\Sigma},T_\f^\dagger,\Delta_\emptyset) \xrightarrow{\res_p} H^1(G_p,T_\f^\dagger))
$, we conclude that $\res_p\left(\delta(T_\f^\dagger,\Delta_\emptyset)\right)\neq 0$, as required.

Suppose now that $\delta(T_\f^\dagger\otimes\epsilon_K,\Delta_\emptyset)\neq 0 \neq \res_p\left(\delta(T_\f^\dagger,\Delta_\emptyset)\right)$. It follows from this input and \eqref{eqn_5_3_19_05_2021} that 
$$\widetilde{H}^2_{\rm f}(G_{\QQ,\Sigma},T_\f^\dagger\,\otimes\,{\rm Ind}_{K/\QQ}\mathds{1},\Delta_{\rm BDP})\simeq \widetilde{H}^2_{\rm f}(G_{K,\Sigma_K},T_\f^\dagger,\Delta_{\rm BDP})$$
is $\cR_\f$-torsion, which is equivalent to the non-vanishing  $\delta(T_\f^\dagger,\Delta_{\rm BDP})\neq 0$ of the module of leading terms.
\end{proof}

Under mild hypotheses, one can study the non-vanishing of $\res_p\left(\delta(T_\f^\dagger\otimes\epsilon_K,\Delta_\emptyset)\right)$  and $\res_p\left(\delta(T_\f^\dagger,\Delta_\emptyset)\right)$ with the aid of Beilinson--Kato elements. We will elaborate on this point in \S\ref{subsubsec_2022_09_18_2046}.

\subsubsection{} 
\label{subsubsec_5_3_2_20_05_2021}
We will explicitly describe the injection $\delta^1_K$. This map arises from the exact sequence of complexes
$$\resizebox{\hsize}{!}{\xymatrix@C=.5cm{&&\widetilde{C}^\bullet_{\rm f}(G_{\QQ,\Sigma},T_\f^\dagger\otimes 1,\Delta_\emptyset)\oplus \widetilde{C}^\bullet_{\rm f}(G_{\QQ,\Sigma},T_\f^\dagger\,\otimes\,c,\Delta_{0})\ar@{=}[d]^{\eqref{eqn_lemma_propagate_BDP_to_V_f_induced_1_and_Psiad_induced_1}}&&\\
0\ar[r]&\widetilde{C}^\bullet_{\rm f}(G_{\QQ,\Sigma},T_\f^\dagger\otimes \frac{1+c}{2},\Delta_0)\ar[r]& \widetilde{C}^\bullet_{\rm f}(G_{\QQ,\Sigma},T_\f^\dagger\,\otimes\,{\rm Ind}_{K/\QQ}\mathds{1},\Delta_{\rm BDP})\ar[r]^(.55){\pi_{\rm BDP}}& 
\widetilde{C}^\bullet_{\rm f}(G_{\QQ,\Sigma},T_\f^\dagger\,\otimes\,\frac{1-c}{2},\Delta_{\emptyset})\ar[r]& 0
}}$$
as a connecting morphism. Let 
\begin{equation}
    \label{eqn_define_cocycle_xf}
    x_{\rm f}:=\left(x, \left(x_v\otimes \frac{1-c}{2}\right), \left(\lambda_v\otimes \frac{1-c}{2}\right) \right)\in \widetilde{Z}^1_{\rm f}\left(G_{\QQ,\Sigma},T_\f^\dagger\,\otimes\,\frac{1-c}{2},\Delta_{\emptyset}\right)
\end{equation}
be a cocycle:
$$x\in {Z}^1\left(G_{\QQ,\Sigma},T_\f^\dagger\,\otimes\,\frac{1-c}{2}\right)\,,\,\, x_v\in {U}^1_v(\Delta_\emptyset,T_\f^\dagger)\,,\,\,\lambda_v\in T_\f^\dagger$$
$$dx_v=0\,,\qquad \res_v(x)-\iota_v^+(x_v)\otimes\frac{1-c}{2}=d\lambda_v\otimes\frac{1-c}{2} \qquad \qquad \forall v\in \Sigma\,.$$
Then the cochain
$$y_{\rm f}:=\left(x, \left(x_v\otimes \frac{1-c}{2}+x_v\otimes \frac{1+c}{2}=x_v\otimes 1\right), \left(\lambda_v\otimes \frac{1-c}{2}\right) \right)\in \widetilde{C}^1_{\rm f}(G_{\QQ,\Sigma},T_\f^\dagger\,\otimes\,{\rm Ind}_{K/\QQ}\mathds{1},\Delta_{\rm BDP})$$
maps to $x_{\rm f}$ under the morphism $\pi_{\rm BDP}$. Observe then that 
\begin{align*}
    dy_{\rm f}&=\left(dx=0, \left(dx_v\otimes 1 =0\right), \left( \underbrace{\left(d\lambda_v\otimes \frac{1-c}{2}-\res_v(x)+\iota_v^+(x_v)\otimes\frac{1-c}{2}\right)}_{=0}+\iota_v^+(x_v)\otimes \frac{1+c}{2}\right) \right)\\
    &=\left(0, 0, \left(\iota_v^+(x_v)\otimes \frac{1+c}{2}\right) \right)\in \widetilde{C}^2_{\rm f}(G_{\QQ,\Sigma},T_\f^\dagger\,\otimes\,{\rm Ind}_{K/\QQ}\mathds{1},\Delta_{\rm BDP})\,.
\end{align*}
Since $x_v$ is a cocyle, then so is $(0,0,\left(\iota_v^+(x_v)\otimes \frac{1+c}{2}\right))\in \widetilde{C}^2_{\rm f}(G_{\QQ,\Sigma},T_\f^\dagger\otimes 1,\Delta_0)$ and it is clear that 
$$\iota_{\rm BDP}\left(0,0,\left(\iota_v^+(x_v)\otimes \frac{1+c}{2}\right)\right)=dy_{\rm f}\,.$$
\begin{proposition}
\label{prop_delta1_K_explicit}
For a cocycle $x_{\rm f}$ given as in \eqref{eqn_define_cocycle_xf}, we have
$$\delta_K^1([x_{\rm f}])=[(0,0,(\iota_v^+(x_v)))]\,\,\in\,\, \widetilde{H}^2_{\rm f}(G_{\QQ,\Sigma},T_\f^\dagger,\Delta_0)\,.$$
\end{proposition}

\begin{proof}
This is clear by the definition of $\delta_K^1$ as the connecting homomorphism.
\end{proof}

\subsubsection{} Consider the exact sequence of complexes
$$0\lra \widetilde{C}^\bullet_{\rm f}(G_{\QQ,\Sigma},T_\f^\dagger,\Delta_{0})\lra \widetilde{C}^\bullet_{\rm f}(G_{\QQ,\Sigma},T_\f^\dagger,\Delta_{\emptyset})\lra C^\bullet(G_p,T_\f^\dagger)\lra 0$$
(c.f. \cite{nekovar06}, \S5.3.1.3). In particular, we have an exact sequence
\begin{equation}
\label{eqn_exact_nekovar_5_3_1_3_21_05_2021}
    \widetilde{H}^1_{\rm f}(G_{\QQ,\Sigma},T_\f^\dagger,\Delta_{\emptyset})\xrightarrow{\res_p} H^1(G_p,T_\f^\dagger)\xrightarrow{\partial^1_\f} \widetilde{H}^2_{\rm f}(G_{\QQ,\Sigma},T_\f^\dagger,\Delta_{0})\, \qquad\qquad [x_p]\mapsto [(0,0,x_p)]\,.
\end{equation} 

\begin{proposition}
\label{prop_delta1_K_explicit_bis}
In the setting of Theorem~\ref{thm_main_BDP_factorization}, the morphism $\delta_K^1$ factors as
$$\xymatrix{
\widetilde{H}^1_{\rm f}(G_{\QQ,\Sigma},T_\f^\dagger\otimes \epsilon_K,\Delta_{\emptyset})\ar[r]^(.53){\delta^1_K}\ar@{^(->}[d]_{\res_p}& \widetilde{H}^2_{\rm f}(G_{\QQ,\Sigma},T_\f^\dagger,\Delta_{0})\\
H^1(G_p,T_\f^\dagger\otimes \epsilon_K)\ar[r]^(.53){\sim}_(.53){{\rm tw}_{\epsilon_K}}& H^1(G_p,T_\f^\dagger)\ar[u]_{\partial_\f^1}
}$$
where the bottom horizontal twisting isomorphism is induced from $\xi_{\epsilon_K}\mapsto 1$, where $\xi_{\epsilon_K}$ is the generator of the module $\ZZ_p(\epsilon_K)$ that corresponds to $\frac{1-c}{2}\in \ZZ_p[\Delta]$.
\end{proposition}

\begin{proof}
This is clear from the description of $\partial^1_\f$ in \eqref{eqn_exact_nekovar_5_3_1_3_21_05_2021} and that of $\delta^1_K$ in Proposition~\ref{prop_delta1_K_explicit}. We note that the left vertical injection thanks to the proof of Lemma~\ref{lemma_2022_09_19_1135}(iv).
\end{proof}

\begin{remark}
We invite the reader to compare Proposition~\ref{prop_delta1_K_explicit_bis} to its counterpart \cite[Proposition~6.7]{BS_Part1} in the scenario when $\g$ does not have CM.
\end{remark}

\subsubsection{}
\label{subsubsec_6_3_4_21_05_2021}
Our goal in \S\ref{subsubsec_6_3_4_21_05_2021} is to prove the following statement.

\begin{proposition}
\label{prop_21_05_2021_6_11}
In the setting of \S\ref{subsubsec_1_18_05_2021_subsec_factorize_BDP_padic_L_function}, the map $\delta^1_K$ factors as 
\begin{equation}
    \label{eqn_20052021_5_4_bis_bis}
    \begin{aligned}
        \xymatrix{
    \widetilde{H}^1_{\rm f}(G_{\QQ,\Sigma},T_\f^\dagger\otimes \epsilon_K,\Delta_{\emptyset})\ar@{^{(}->}[rr]^{\delta^1_K} \ar@{^{(}->}[rd]_{\res_{/{\rm Pan}}}&& \widetilde{H}^2_{\rm f}(G_{\QQ,\Sigma},T_\f^\dagger,\Delta_{0})\\
     &\dfrac{H^1(G_p,T_\f^\dagger)}{\res_p(\widetilde{H}^1_{\rm f}(G_{
     \QQ,S},T_\f^\dagger,\Delta_{\rm Pan}))}\ar@{^{(}->}[ru]_(.58){\partial^1_\f}&
    }
    \end{aligned}
\end{equation}
\end{proposition}

\begin{proof}
As we have seen in \eqref{eqn_5_3_19_05_2021}, the map $\delta^1_K$ is injective under our running assumptions. Proposition~\ref{prop_delta1_K_explicit_bis} tells us that $\delta^1_K$ factors as 
\begin{equation}
    \label{eqn_21052021_6_7_bis_proof_1}
    \begin{aligned}
        \xymatrix{
    \widetilde{H}^1_{\rm f}(G_{\QQ,\Sigma},T_\f^\dagger\otimes \epsilon_K,\Delta_{\emptyset})\ar@{^{(}->}[r]^(.55){\delta^1_K} \ar@{^{(}->}[d]_{{\rm tw}_{\epsilon_K}\circ\,\res_p\quad}& \widetilde{H}^2_{\rm f}(G_{\QQ,\Sigma},T_\f^\dagger,\Delta_{0})\\
     {H^1(G_p,T_\f^\dagger)}\ar[ru]_(.5){\partial^1_\f}&
    }
    \end{aligned}
\end{equation}
 By the description of the connecting morphism $\partial_\f^1$ as the connecting morphism in the exact triangle \eqref{eqn_exact_nekovar_5_3_1_3_21_05_2021}, we observe that $\partial_\f^1$ factors as
\begin{equation}
    \label{eqn_21052021_6_8_bis_proof_2}
    \begin{aligned}
        \xymatrix{
   {H^1(G_p,T_\f^\dagger)}\ar@{->>}[r] &\dfrac{H^1(G_p,T_\f^\dagger)}{\res_p(\widetilde{H}^1_{\rm f}(G_{
     \QQ,S},T_\f^\dagger,\Delta_{\emptyset}))}\ar@{^{(}->}[r]^-{\partial^1_\f}  &\widetilde{H}^2_{\rm f}(G_{\QQ,\Sigma},T_\f^\dagger,\Delta_{0})\,.
    }
    \end{aligned}
\end{equation}
By \cite[Proposition~6.2(i)]{BS_Part1}, we have $\widetilde{H}^1_{\rm f}(G_{
     \QQ,S},T_\f^\dagger,\Delta_{\emptyset})=\widetilde{H}^1_{\rm f}(G_{
     \QQ,S},T_\f^\dagger,\Delta_{\rm Pan})$ under our running assumptions. This combined with the diagrams \eqref{eqn_21052021_6_7_bis_proof_1} and \eqref{eqn_21052021_6_8_bis_proof_2} (and defining ${\res}_{/{\rm Pan}}$ as the composite of ${\rm tw}_{\epsilon_K}\circ\,\res_p$ and the surjection $ {H^1(G_p,T_\f^\dagger)}\twoheadrightarrow \dfrac{H^1(G_p,T_\f^\dagger)}{\res_p(\widetilde{H}^1_{\rm f}(G_{
     \QQ,S},T_\f^\dagger,\Delta_{\emptyset}))}$) we conclude our proof.
\end{proof}

\subsubsection{} According to Proposition~\ref{prop_21_05_2021_6_11}, the composite map
$$\res_p^-: \widetilde{H}^1_{\rm f}(G_{\QQ,\Sigma},T_\f^\dagger\otimes \epsilon_K,\Delta_{\emptyset})\xrightarrow{\res_p} H^1(G_p,T_\f^\dagger\otimes \epsilon_K) \lra H^1(G_p,F^-T_\f^\dagger\otimes \epsilon_K)\xrightarrow[{\rm tw}_{\epsilon_K}]{\sim}H^1(G_p,F^-T_\f^\dagger)$$
factors as
$\widetilde{H}^1_{\rm f}(G_{\QQ,\Sigma},T_\f^\dagger\otimes \epsilon_K,\Delta_{\emptyset})\xrightarrow{\res_{/{\rm Pan}}} \dfrac{H^1(G_p,T_\f^\dagger)}{\res_p(\widetilde{H}^1_{\rm f}(G_{
     \QQ,S},T_\f^\dagger,\Delta_{\emptyset}))} \twoheadrightarrow H^1(G_p,F^-T_\f^\dagger)\,.$

\subsubsection{}
We assume the following:
\begin{itemize}
\item[\mylabel{item_MC}{\bf MC})]
 $\f$ admits a crystalline specialization of weight $k\equiv 2 \pmod{p-1}$, and either there exists a prime $q|| N$ such that $\overline{\rho}_\f(I_q)\neq \{1\}$, or there exists a real quadratic field $F$ verifying the conditions of \cite[Theorem 4]{xinwanwanhilbert}.
 \item[\mylabel{item_BI}{\bf BI})] $\rho_\f(G_{\QQ(\zeta_{p^\infty})})$ contains a conjugate of ${\rm SL}_2(\ZZ_p)$.     \item[\mylabel{item_non_anom}{\bf NA})] $a_p(\f)-1\in \cR_\f^\times$\,.
 \end{itemize}
We note that the hypothesis \eqref{item_MC} and the big image condition \eqref{item_BI}, which is a strengthening of \ref{item_Irr}, guarantee that the results towards the Iwasawa main conjectures for $\f$ established in \cite{skinnerurbanmainconj,xinwanwanhilbert} apply.

Thanks to these assumptions, we have $\widetilde{H}^1_{\rm f}(G_{\QQ,S},T_\f^\dagger,\Delta_{\rm Pan})=\widetilde{H}^1_{\rm f}(G_{\QQ,S},T_\f^\dagger,\Delta_{\emptyset})$ (cf. \cite[Proposition~6.2(i)]{BS_Part1}). Thence, we have the following exact sequence:
\begin{align}
    \label{eqn_exact_seq_21052021_6_9}
    \begin{aligned}
    0\lra \dfrac{H^1(G_p,F^+T_\f^\dagger)}{\res_p\left(\widetilde{H}^1_{\rm f}(G_{\QQ,\Sigma},T_\f^\dagger,\Delta_{\emptyset})\right)}
    \lra 
    \dfrac{H^1(G_p,T_\f^\dagger)}{\res_p\left(\widetilde{H}^1_{\rm f}(G_{\QQ,\Sigma},T_\f^\dagger,\Delta_{\emptyset})\right)}\lra
   H^1(G_p,F^-T_\f^\dagger) \lra  0\,.
   \end{aligned}
\end{align}
 As the map $\res_{/\rm Pan}$ is injective (cf. Proposition~\ref{prop_21_05_2021_6_11}) and $\widetilde{H}^1_{\rm f}(G_{\QQ,\Sigma},T_\f^\dagger\otimes \epsilon_K,\Delta_{\emptyset})$ is torsion-free, it follows that the submodule
$$ \,\res_{/\rm Pan}(\delta(T_\f^\dagger\otimes\epsilon_K,\Delta_\emptyset))\subset \dfrac{H^1(G_p,F^+T_\f^\dagger)}{\res_p\left(\widetilde{H}^1_{\rm f}(G_{\QQ,\Sigma},T_\f^\dagger,\Delta_{\emptyset})\right)}$$
generated by $\res_{/\rm Pan}(\delta(T_\f^\dagger\otimes\epsilon_K,\Delta_\emptyset))$ is torsion-free. In particular,
\begin{align*}
    &\res_{/\rm Pan}\,(\delta(T_\f^\dagger\otimes\epsilon_K,\Delta_\emptyset))\,\,\bigcap \,\, {\rm im}\left(  \dfrac{H^1(G_p,F^+T_\f^\dagger)}{\res_p\left(\widetilde{H}^1_{\rm f}(G_{\QQ,\Sigma},T_\f^\dagger,\Delta_{\emptyset})\right)}\lra \dfrac{H^1(G_p,T_\f^\dagger)}{\res_p\left(\widetilde{H}^1_{\rm f}(G_{\QQ,\Sigma},T_\f^\dagger,\Delta_{\emptyset})\right)} \right) 
    \\
    &\qquad\qquad =  \res_{/\rm Pan}(\delta(T_\f^\dagger\otimes\epsilon_K,\Delta_\emptyset))\,\bigcap\,\, \ker\left(  \dfrac{H^1(G_p,T_\f^\dagger)}{\res_p\left(\widetilde{H}^1_{\rm f}(G_{\QQ,\Sigma},T_\f^\dagger,\Delta_{\emptyset})\right)}\lra {H^1(G_p,F^-T_\f^\dagger)} \right)=\{0\}
\end{align*}
and $  \res_{/\rm Pan}\,(\delta(T_\f^\dagger\otimes\epsilon_K,\Delta_\emptyset))$ maps isomorphically onto its image $  \res_{p}^-(\delta(T_\f^\dagger\otimes\epsilon_K,\Delta_\emptyset))\subset H^1(G_p,F^-T_\f^\dagger)$.
Therefore, we have an exact sequence
\begin{equation}
    \label{eqn_exact_seq_2152021_6_10}
    0\lra \dfrac{H^1(G_p,F^+T_\f^\dagger)}{\res_p\left(\widetilde{H}^1_{\rm f}(G_{\QQ,\Sigma},T_\f^\dagger,\Delta_{\emptyset})\right)}
    \lra 
    \dfrac{\dfrac{H^1(G_p,T_\f^\dagger)}{\res_p\left(\widetilde{H}^1_{\rm f}(G_{\QQ,\Sigma},T_\f^\dagger,\Delta_{\emptyset})\right)}}{  \res_{/\rm Pan}\,(\delta(T_\f^\dagger\otimes\epsilon_K,\Delta_\emptyset))}
    \lra
   \dfrac{H^1(G_p,F^-T_\f^\dagger)}{  \res_{p}^-\,(\delta(T_\f^\dagger\otimes\epsilon_K,\Delta_\emptyset))} \lra  0
\end{equation}
that one obtains from \eqref{eqn_exact_seq_21052021_6_9} via the discussion above.

\begin{proposition}
\label{prop_useful_step_in_factorization_21_05_2021_CM}
In the setting of \S\ref{subsubsec_1_18_05_2021_subsec_factorize_BDP_padic_L_function}, we have
\begin{align*}
{\rm char}\left(\dfrac{\widetilde{H}^2_{\rm f}(G_{\QQ,\Sigma},T_\f^\dagger,\Delta_{0})}{\delta^1_K(\delta(T_\f^\dagger\otimes\epsilon_K,\Delta_\emptyset))}\right)
=
{\rm char}\left( \dfrac{H^1(G_p,F^-T_\f^\dagger)}{  \res_p^-(\delta(T_\f^\dagger\otimes\epsilon_K,\Delta_\emptyset))}\right)
{\rm char}\left( \dfrac{H^1(G_p,F^+T_\f^\dagger)}{\res_p(\delta(T_\f^\dagger,\Delta_\emptyset))}\right) \,.
\end{align*}
\end{proposition}

\begin{proof}
Combining Proposition~\ref{prop_21_05_2021_6_11} and \eqref{eqn_exact_seq_2152021_6_10}, we infer that
\begin{equation}
\label{eqn_proof_step_1_prop_useful_step_in_factorization_21_05_2021_CM}
 \resizebox{0.91\hsize}{!}{$
{\rm char}\left(\dfrac{\widetilde{H}^2_{\rm f}(G_{\QQ,\Sigma},T_\f^\dagger,\Delta_{0})}{\delta^1_K(\delta(T_\f^\dagger\otimes\epsilon_K,\Delta_\emptyset))}\right)=
{\rm char}\,\left( \dfrac{H^1(G_p,F^-T_\f^\dagger)}{  \res_p^-(\delta(T_\f^\dagger\otimes\epsilon_K,\Delta_\emptyset))}\right)\\
\times {\rm char}\left( \dfrac{H^1(G_p,F^+T_\f^\dagger)}{\res_p\left(\widetilde{H}^1_{\rm f}(G_{\QQ,\Sigma},T_\f^\dagger,\Delta_{\emptyset})\right)}\right) 
{\rm char}\left({\rm coker}(\partial^1_\f) \right)$}\,.
\end{equation}
This combined with Equation (6.47) Proposition~6.2(ii) of \cite{BS_Part1} shows that
\begin{align}
\label{eqn_proof_step_2_prop_useful_step_in_factorization_21_05_2021_CM}
\begin{aligned}
{\rm char}\left(\dfrac{\widetilde{H}^2_{\rm f}(G_{\QQ,\Sigma},T_\f^\dagger,\Delta_{0})}{\delta^1_K(\delta(T_\f^\dagger\otimes\epsilon_K,\Delta_\emptyset))}\right)={\rm char}\left( \dfrac{H^1(G_p,F^-T_\f^\dagger)}{  {\res}_{p}^-(\delta(T_\f^\dagger\otimes\epsilon_K,\Delta_\emptyset))}\right)
{\rm char}\left( \dfrac{H^1(G_p,F^+T_\f^\dagger)}{\res_p(\delta(T_\f^\dagger,\Delta_\emptyset))}\right) 
\end{aligned}
\end{align}
as required.
\end{proof}

\subsubsection{}
\label{subsubsec_6_3_6_26_05_2021_11_29}
We now complete the proof of Theorem~\ref{thm_main_BDP_factorization}. The exact sequence \eqref{eqn_5_3_19_05_2021} shows that 
\begin{align}
\begin{aligned}
\label{eqn_2022_09_20_0948}
 &{\rm char} \left(\widetilde{H}^2_{\rm f}(G_{\QQ,\Sigma},T_\f^\dagger\,\otimes\,{\rm Ind}_{K/\QQ}\mathds{1},\Delta_{\rm BDP}) \right)
    \cdot 
    {\rm char} \left(\widetilde{H}^1_{\rm f}(G_{\QQ,\Sigma},T_\f^\dagger\,\otimes\,\epsilon_K,\Delta_{\emptyset})\Big{/}  \delta(T_\f^\dagger\otimes\epsilon_K,\Delta_\emptyset) \right) \\
    &\qquad\qquad\qquad = {\rm char} \left( \widetilde{H}^2_{\rm f}(G_{\QQ,\Sigma},T_\f^\dagger,\Delta_{0}))\Big{/}  \delta_K^1\left(\delta(T_\f^\dagger\otimes\epsilon_K,\Delta_\emptyset)\right)  \right)\cdot 
    {\rm char} \left( \widetilde{H}^2_{\rm f}(G_{\QQ,\Sigma},T_\f^\dagger\,\otimes\,\epsilon_K,\Delta_{\emptyset}) \right)\,.
\end{aligned}
\end{align}
Using (in this order) \cite[Corollary~4.41]{BS_Part1}, Equation~\eqref{eqn_2022_09_19_1231}, \cite[Proposition~6.2]{BS_Part1} with \eqref{eqn_2022_09_20_0948}, Proposition~\ref{prop_useful_step_in_factorization_21_05_2021_CM}, Equations~6.8 and 6.13 in op. cit., we therefore have,
\begin{align}
    \label{eqn_5_3_5_5_5_19_05_2021}
    \begin{aligned}
    \det\left(\widetilde{H}^2_{\rm f}(G_{K,\Sigma_K},T_\f^\dagger,\Delta_{\rm BDP}) \right)
={\rm Exp}^*_{F^-T_\f^\dagger\otimes \epsilon_K}(\delta(T_\f^\dagger\otimes\epsilon_K,\Delta_\emptyset))\cdot {\rm Log}_{\omega_\f}(\delta(T_\f^\dagger,\Delta_\emptyset))\,.
\end{aligned}
\end{align}
The proof of Theorem~\ref{thm_main_BDP_factorization} is now complete. \qed

\subsubsection{}
\label{subsubsec_2022_09_18_2046}
We give an explicit criterion for the validity of the non-vanishing  conditions 
\begin{align}
\begin{aligned}
 \delta(T_2^\dagger,\Delta_\g)\,\, &\stackrel{{\rm Prop.}\, \ref{prop_2022_09_18_1639}}{\iff{\,\,}}\delta(T_\f^\dagger,\Delta_{\rm BDP})\neq 0\neq  \delta(T_\f^\dagger\,\widehat\otimes\,\Psi_{\rm ad},\Delta_{\rm BDP})\\
    &\stackrel{{\rm Lemma}\, \ref{lemma_2022_09_19_1253}}{\iff{\,\,}} \res_p\left(\delta(T_\f^\dagger\otimes\epsilon_K,\Delta_\emptyset)\right)\neq 0\neq \res_p\left(\delta(T_\f^\dagger,\Delta_\emptyset)\right) \hbox{ and } \delta(T_\f^\dagger\,\widehat\otimes\,\Psi_{\rm ad},\Delta_{\rm BDP})\neq 0
\end{aligned}
\end{align}
in terms of various $p$-adic $L$-functions. 

\begin{proposition}
\label{prop_2022_09_20_1057}
Suppose that \ref{item_Irr} and \ref{item_Dist}  hold for the family $\f$. Then:
\item[i)] $\res_p\left(\delta(T_\f^\dagger\otimes\epsilon_K,\Delta_\emptyset)\right)\neq 0$ if the specialization $\cL_p^{\rm Kit}(\f)(\kappa, \rmw(\kappa)/2)$ of the Mazur--Kitagawa $p$-adic $L$-function $\cL_p^{\rm Kit}(\f)$ to the central critical line is non-zero.
\item[ii)] $\res_p\left(\delta(T_\f^\dagger,\Delta_\emptyset)\right)\neq 0$ if the $p$-local restriction $\res_p({\rm BK}_\f^\dagger)$ is non-zero. 
\end{proposition}

The requirement $\res_p({\rm BK}_\f^\dagger)\neq 0$ is equivalent to the existence of a crystalline specialization $\kappa \in \cW_\f$ with the following property: Let $f_0$ denote the new form attached to the $p$-old form $\f_\kappa$ and let $\alpha_0,\beta_0$ denote the roots of the Hecke polynomial of $f_0$ at $p$. Then the $p$-stabilized eigenform $f_0^{\beta_0}$ is non-theta-critical (in the sense of Coleman) and we have either 
$L_{p,\alpha_0}^\prime(f_0,\frac{w(\kappa)}{2})\neq 0$ or $L_{p,\beta_0}^\prime(f_0,\frac{w(\kappa)}{2})\neq 0$ for the derivatives of the Manin--Vi\v{s}ik and Pollack--Stevens $p$-adic $L$-functions at the central critical point. A conjecture of Greenberg (cf. \cite{Greenberg_1994_families}; see also \cite{TrevorArnold_Greenberg_Conj}, Conjectures 1.3 and 1.4) predicts that both sufficient conditions recorded in Proposition~\ref{prop_2022_09_20_1057} hold true at all times, under the hypothesis \ref{item_root_numbers_general_CM} on the global root numbers.

\begin{proof}[Proof of Proposition~\ref{prop_2022_09_20_1057}]
We recall that $\res_p\left(\delta(T_\f^\dagger,\Delta_\emptyset)\right)\neq 0$ if and only if the $\cR_\f$-module $\widetilde{H}^2_{\rm f}(G_{\QQ,\Sigma},T_\f^\dagger\,\otimes\,\epsilon_K,\Delta_{\emptyset})$ is torsion. This latter condition holds true under the running hypothesis thanks to \cite[Theorem 14.5]{kato04}. This completes the proof of the first part.

To prove the second assertion, we remark that $\res_p\left(\delta(T_\f^\dagger,\Delta_\emptyset)\right)\neq 0$ if and only if if and only if the $\cR_\f$-module $\widetilde{H}^2_{\rm f}(G_{\QQ,\Sigma},T_\f^\dagger\,\otimes\,\epsilon_K,\Delta_{\emptyset})$ is torsion, and this latter condition holds true if $\res_p({\rm BK}_\f^\dagger)\neq 0$, thanks to the Beilinson--Kato Euler system (utilized at arithmetic specializations $\f_\kappa$ of the family $\f$).

\end{proof}

\begin{proposition}
    \label{prop_2022_09_20_1303}
Suppose that the conditions \ref{item_Irr} and \ref{item_Dist}  hold for the family $\f$. Assume in addition that we have $\Psi_{\rm ad}^2 \not\equiv \mathds{1} \mod \mathfrak{m}_\g$. Then $\delta(T_\f^\dagger\,\widehat\otimes\,\Psi_{\rm ad},\Delta_{\rm BDP})\neq 0$ provided that the ${\rm BDP}^2$ $p$-adic $L$-function $ \cL_p^{\rm BDP}(\hf_{/K} \otimes \Psi^\Ad )(\kappa, \lambda )$ (cf. \S\ref{subsubsec_BDPsquared}) is non-zero.
\end{proposition}

An extension of Greenberg's conjecture above alluded to above leads one to predict that the non-vanishing requirement on $ \cL_p^{\rm BDP}(\hf_{/K} \otimes \Psi^\Ad )(\kappa, \lambda )$ in  Proposition~\ref{prop_2022_09_20_1303} hold true at all times, under the hypothesis \ref{item_root_numbers_general_CM} on the global root numbers.

\begin{proof}
    Recall that we have  $\delta(T_\f^\dagger\,\widehat\otimes\,\Psi_{\rm ad},\Delta_{\rm BDP})\neq 0$ if the $\cR_2$-module $\widetilde{H}^2_{\rm \f}(G_{K,\Sigma_K},T_\f^\dagger\,\widehat\otimes\,\Psi_{\rm ad},\Delta_{\rm BDP})$ is torsion, and this follows as an application of the Beilinson--Flach Euler system and the reciprocity laws that relate it to the $p$-adic $L$-function $\cL_p^{\rm BDP}(\hf_{/K} \otimes \Psi^\Ad )(\kappa, \lambda )$, c.f. \cite{BL_Forum, BL_PR_Volume}. 
\end{proof}

\begin{corollary}
    \label{cor_2022_09_20_1303}
    Suppose that the conditions \ref{item_Irr} and \ref{item_Dist}  hold for the family $\f$ and assume in addition that we have $\Psi_{\rm ad}^2 \not\equiv \mathds{1} \mod \mathfrak{m}_\g$. If $\cL_p^{\rm BDP}(\hf_{/K} \otimes \Psi^\Ad )(\kappa, \lambda )\neq 0\neq \cL_p^{\rm Kit}(\f)(\kappa, \rmw(\kappa)/2)$ and ${\rm Log}_{\omega_\f}({\rm BK}_\f^\dagger)\neq 0$, then $\delta(T_2^\dagger,\Delta_{\g})\neq 0$. Moreover, the factorization~\eqref{eqn_2022_09_20_1331} holds true under these hypotheses.
\end{corollary}


\subsection{Comparison of two degree-\texorpdfstring{$6$}{} \texorpdfstring{$p$}{}-adic \texorpdfstring{$L$}{}-functions}
\label{subsec_compare_degree_6_padic_L_functions} We conclude our paper with Theorem~\ref{thm_main_8_4_4_factorization_bis}, which is the algebraic counterpart to our factorization result (Theorem~\ref{Thm:8=6+2CM}).

\begin{proposition}
\label{prop_19_05_2021_5_11}
Under the hypotheses of Corollary~\ref{cor_2022_09_20_1303} we have
\begin{align*}
   {\rm Exp}_{F^-T_\f^\dagger}\left(\delta(M_2^\dagger, {\rm tr}^*\Delta_\g)\right)=\varpi_{2,1}^*{\rm Exp}_{F^-T_\f^\dagger}\left(\delta(T_\f^\dagger\otimes\epsilon_K,\Delta_\emptyset)\right)\cdot\det\left(\widetilde{H}^2_{\rm f}(G_{K,\Sigma_K},T_\f^\dagger\,\widehat\otimes\,\Psi_{\rm ad},\Delta_{\rm BDP})\right)\,.
\end{align*}
\end{proposition}

\begin{proof}
It follows from \eqref{eqn_propagated_local_conditions_Delta_g_14_2} that 
\begin{equation}
\label{eqn_2022_09_20_1356}
    \widetilde{R\Gamma}_{\rm f}(G_{\QQ,\Sigma},M_2^\dagger, {\rm tr}^*\Delta_\g)\, 
    = \,\left(\widetilde{R\Gamma}_{\rm f}(G_{\QQ,\Sigma},T_\f^\dagger\otimes\epsilon_K,\Delta_\emptyset)\otimes_{\varpi_{2,1}^*}\cR_2\right)\,\,\oplus\,\, \widetilde{R\Gamma}_{\rm f}(G_{K,\Sigma_K},T_\f^\dagger\,\widehat\otimes\,\Psi_{\rm ad},\Delta_{\rm BDP})\,,
\end{equation}
and hence, 
\begin{equation}
\label{eqn_2022_09_20_1456}
\delta(M_2^\dagger, {\rm tr}^*\Delta_\g)=\left(\delta(T_\f^\dagger\otimes\epsilon_K,\Delta_\emptyset)\otimes_{\varpi_{2,1}^*}\cR_2\right)\cdot \delta(T_\f^\dagger\,\widehat\otimes\,\Psi_{\rm ad},\Delta_{\rm BDP})\,.
\end{equation}
On identifying the trivializations ${\rm Exp}_{F^-T_\f^\dagger\otimes\epsilon_K}$ and ${\rm Exp}_{F^-T_\f^\dagger}$ via the $p$-local twisting isomorphism ${\rm tw}_{\epsilon_K}$ given as in Proposition~\ref{prop_delta1_K_explicit_bis}, we conclude that
\begin{align*}
    \begin{aligned}
    {\rm Exp}_{F^-T_\f^\dagger}\left(\delta(M_2^\dagger, {\rm tr}^*\Delta_\g)\right)&=\varpi_{2,1}^*{\rm Exp}_{F^-T_\f^\dagger}\left(\delta(T_\f^\dagger\otimes\epsilon_K,\Delta_\emptyset)\right)\cdot \delta(T_\f^\dagger\,\widehat\otimes\,\Psi_{\rm ad},\Delta_{\rm BDP})\\
    &=\varpi_{2,1}^*{\rm Exp}_{F^-T_\f^\dagger}\left(\delta(T_\f^\dagger\otimes\epsilon_K,\Delta_\emptyset)\right)\cdot \det\left(\widetilde{H}^2_{\rm f}(G_{K,\Sigma_K},T_\f^\dagger\,\widehat\otimes\,\Psi_{\rm ad},\Delta_{\rm BDP})\right)\,,
    \end{aligned}
\end{align*}
as required.
\end{proof}

\begin{theorem}
\label{thm_main_8_4_4_factorization_bis}
Under the hypotheses of Corollary~\ref{cor_2022_09_20_1303} we have
$${\rm Log}_{\omega_\f}\left(\delta(T_2^\dagger,\Delta_{+})\right)=\delta(T_2^\dagger,\Delta_\g)\,{=}\, {\rm Exp}_{F^-T_\f^\dagger}\left(\delta(M_2^\dagger, {\rm tr}^*\Delta_\g)\right) \cdot \varpi_{2,1}^*{\rm Log}_{\omega_\f}(\delta(T_\f^\dagger,\Delta_\emptyset))$$
holds true also when $\g$ has CM.
\end{theorem}

\begin{proof}
This is an immediate consequence of Corollary~\ref{cor_2022_09_19_1230}, Corollary~\ref{cor_2022_09_20_1303} and Proposition~\ref{prop_19_05_2021_5_11}.
\end{proof}

\bibliographystyle{amsalpha}
\bibliography{references}
\end{document}